\documentclass[dvips]{article}
\usepackage[latin1]{inputenc}
\usepackage[francais]{babel}
\usepackage{amssymb}
\usepackage{amsmath}
\usepackage{amsthm}
\usepackage{dsfont}
\usepackage{mathrsfs}
\usepackage{graphicx}
\usepackage{epsfig}
\usepackage{psfrag}
\usepackage{caption}
\usepackage{a4wide}

\DeclareGraphicsExtensions{.eps}

\newtheorem{thm}{Théorème}[section]
\newtheorem{prop}[thm]{Proposition}
\newtheorem{lem}[thm]{Lemme}

\theoremstyle{definition}
\newtheorem{defn}[thm]{Définition}

\theoremstyle{remark}
\newtheorem{rmk}[thm]{Remarque}

\newtheoremstyle{citing}
	{3pt}
	{3pt}
	{\itshape}
	{}
	{\bfseries}
	{.}
	{.5em}
	{\thmnote{#3}}

\theoremstyle{citing}
\newtheorem*{introthm}{}
\newtheorem*{introdefn}{}

\date{Le \today}
\author{\textsc{Jullian} Yann}
\title{Construction du c{\oe}ur compact d'un arbre réel par substitution d'arbre}

\addcontentsline{toc}{section}{Introduction}

\begin{document}
\pagenumbering{arabic}
\begin{center}
\vbox{
\vspace*{4.7cm}
{\LARGE Construction du c{\oe}ur compact d'un arbre réel par substitution d'arbre}

\vspace*{0.5cm}

{\large Yann {\scshape Jullian}}

\vspace*{0.8cm}

\begin{abstract}
Etant donné un automorphisme de groupe libre $\sigma$ et un représentant topologique
train-track de son inverse, on peut construire un arbre réel $T$ appelé
arbre répulsif de $\sigma$. Le groupe libre agit sur $T$ par isométries.
La dynamique engendrée par $\sigma$ peut être représentée par l'action
du groupe libre restreinte à un sous-ensemble compact bien choisi
du complété métrique de $T$. Cet article construit ce sous-ensemble
sur une classe d'exemples en introduisant des opérations appelées substitutions
d'arbre ; on mettra en évidence les relations entre la construction
par substitution d'arbre et la dynamique symbolique sous-jacente.
\end{abstract}

}
\end{center}

\setcounter{tocdepth}{2} 
\tableofcontents{}

\section*{Introduction}

On distingue deux grandes classes de systèmes symboliques ; les décalages de type fini (\cite{LM}), et les systèmes substitutifs (\cite{Que}).
Les derniers sont souvent de bons candidats pour décrire des systèmes auto-similaires, c'est-à-dire des systèmes
dont la dynamique globale se retrouve localement par l'application premier retour (\cite{AI}). On se pose ici le problème
inverse ; étant donné un système dynamique substitutif, peut-on interpréter géométriquement sa dynamique ?
Cette question a été largement étudiée et abordée sous des angles multiples, et chaque méthode
vient avec son propre jeu de restrictions. On citera par exemple
les fractals de Rauzy (\cite{Rau}, \cite{AI}, \cite{CanSie2}) ou les échanges d'intervalles (\cite{Kea}, \cite{Vee}, \cite{Ra}), dont la dynamique symbolique est engendrée par
un échange de domaines. Une condition suffisante pour l'existence d'un fractal de Rauzy (et d'un échange de
domaines défini à ensembles de mesures nulles près sur ce fractal) est donnée dans \cite{AI} et \cite{CanSie2} ;
si la matrice d'incidence de la substitution est unimodulaire Pisot et si la substitution vérifie la condition
de forte coïncidence, alors on peut définir un fractal de Rauzy qui supporte la dynamique de la substitution.
Dans le cas des échanges d'intervalles, il est par exemple nécessaire que la matrice d'incidence de la substitution soit symplectique (\cite{Via}).
En règle générale, la question de l'existence d'un fractal de Rauzy (et d'un échange de
domaines bien défini) ou d'un échange d'intervalles permettant de représenter la dynamique d'un système substitutif est
encore largement ouverte. On note que ces représentations sont basées sur la minimalité des
systèmes dynamiques symboliques engendrés par une classe de substitutions ; les substitutions primitives.
On rappelera les propriétés les plus importantes de ces systèmes en section \ref{sectioncombmots}.

Toute substitution sur un alphabet $A$ peut s'étendre en un endomorphisme du groupe libre de base $A$.
Tandis que la théorie générale des endomorphismes de groupe libre est encore incertaine, celle des
automorphismes est beaucoup plus développée (\cite{CV}, \cite{BH}, \cite{BFH}), et c'est dans ce cadre que nous nous placerons. En général,
la dynamique engendrée par un automorphisme est plus compliquée que celle d'une substitution ;
des annulations peuvent se produire. Dans \cite{BH}, M. Bestvina et M. Handel donnent des méthodes
pour contrôler ces annulations, et définissent notamment les représentants topologiques
(applications $f:G\to G$, où $G$ est un graphe topologique) train-track
des automorphismes de groupe libre. Pour une métrique bien choisie sur $G$, une application
train-track étend uniformément les arcs de $G$ (en multipliant leurs longueurs par un facteur constant).

La dynamique des automorphismes de groupe libre est souvent représentée par des actions de groupe libre
sur des arbres réels. Se servant des résultats de \cite{BH}, D. Gaboriau, A. Jaeger, G. Levitt et M. Lustig
associent dans \cite{GJLL} un arbre réel $T_\alpha$
à tout automorphisme $\alpha$ de groupe libre. L'arbre $T_\alpha$ est muni d'une action
(non-triviale, minimale et avec des stabilisateurs d'arcs triviaux) du groupe libre par isométries,
et l'action de l'automorphisme $\alpha$ est représentée par une homothétie sur $T_\alpha$. Lorsque le
train-track représentant $\alpha$ est strictement dilatant,
l'action est à orbites denses. Rejoignant les travaux de G. Levitt et M. Lustig dans \cite{LL03},
on peut alors construire une application équivariante surjective $Q$ de $\partial F$ (le bord de Gromov du groupe libre $F$)
dans $\overline{T_\alpha}\cup \partial T_\alpha$ (où $\overline{T_\alpha}$ est le complété métrique de $T_\alpha$
et $\partial T_\alpha$ est son bord de Gromov). Cette application traduit l'action du groupe libre sur son bord
en termes d'isométries sur l'arbre. D'autres définitions de $Q$
peuvent être trouvées dans \cite{CHLII} et \cite{CHL09}. Dans \cite{CHL09}, elle est
utilisée pour définir un compact (l'ensemble limite ou c{\oe}ur) (inclus dans $\overline{T}$) associé à tout
arbre $T$ défini avec une action (par isométries) du groupe libre très petite, minimale, et à orbites denses.
On s'intéresse aux dynamiques induites par l'action du groupe libre sur ces compacts.\\

La section \ref{sectionsubarsimp} définit une nouvelle notion : les substitutions d'arbre.
Combinatoirement, les substitutions d'arbre peuvent être vues comme des généralisations des substitutions sur les mots ;
notamment, toute substitution induit naturellement une substitution d'arbre.
Dans \cite[chapitre 4, 5]{THE}, on étudie les propriétés de ces objets,
et on explique comment elles permettent d'obtenir des compacts invariants par constructions
graphe-dirigées (au sens de \cite{MauWil}).
Dans le cadre de cet article, elles sont utilisées comme un moyen simple de construire
des arbres réels compacts auto-similaires. On associera un arbre réel à tout arbre simplicial ;
la section \ref{sectionsubarel} décrit un espace adapté à ces réalisations. La substitution d'arbre produira ainsi
une suite convergente d'arbre réel et on s'intéressera à l'arbre limite.

Etant donnée une substitution, l'objectif est de construire une substitution d'arbre, et
d'obtenir grâce à elle un arbre réel auto-similaire et une partition de celui-ci ; cette
partition nous permettra de définir un échange de domaines conjugué au
système dynamique engendré par la substitution initiale.
L'objet principal de cet article est de proposer une telle construction pour
la famille particulière d'exemples définie au paragraphe suivant.
On insistera sur le fait que la substitution d'arbre permet de mettre en évidence
les propriétés géométriques et dynamiques de l'arbre limite. On prouvera notamment que les
points de branchement (points de valence au moins trois) de cet arbre sont exactement les points dont l'orbite est codée
par un décalé du point fixe de la substitution initiale.
On effectuera également une étude détaillée de la combinatoire du système symbolique,
en particulier des facteurs bispéciaux de son langage ; on montrera comment ces derniers interviennent
dans la construction des points de branchement et permettent d'expliquer  
en quoi la substitution d'arbre \og reflète \fg~ la dynamique engendrée par la substitution.
Ce travail est effectué en section \ref{sectionclasse}.

Chaque substitution $\sigma$ de la famille considérée est primitive inversible, et on la considère
comme un automorphisme de groupe libre. Dans la section \ref{sectionrepulsif},
on définira l'arbre $T$ de \cite{GJLL} associé à l'automorphisme
$\sigma^{-1}$, inverse de $\sigma$, et on montrera que le compact obtenu
en section \ref{sectionclasse} peut être vu comme une partie de $\overline{T}$, le complété métrique de $T$.
Cette partie est à rapprocher de l'ensemble limite décrit dans \cite{CHL09}.

\subsection*{Enoncé des résultats}
Soit $d\ge 3$. On note $A$ l'alphabet $A=\{1, 2,\dots, d\}$ et $A^*$ l'ensemble des mots finis à lettres dans $A$ ;
le mot vide est noté $\epsilon$. La substitution (morphisme du monoïde $A^*$) primitive $\sigma$ est définie par :
\begin{center}
	\begin{tabular}{cccclc}
	$\sigma$ & : & $1$ & $\mapsto$ & $12$ & \\
	&& $k$ & $\mapsto$ & $(k+1)$ & pour $2\le k\le d-1$\\
	&& $d$ & $\mapsto$ & $1$ & \\
	\end{tabular}
\end{center}
L'application décalage est l'application $S$ de $A^{\mathds{N}}$ dans $A^{\mathds{N}}$ qui
à un mot $V = (V_i)_{i\in \mathds{N}}$ associe $S(V) = (V_{i+1})_{i\in \mathds{N}}$.
Soit $\omega$ le mot de $A^{\mathds{N}}$ défini par ~$\omega = \lim\limits_{n\to +\infty}\sigma^n(1)$ ;
on note $\Omega^+$ l'adhérence de l'orbite de $\omega$ sous l'action de $S$.
Le système dynamique symbolique $(\Omega^+, S)$ engendré par $\sigma$ est minimal et uniquement ergodique (cf. \cite{Que}).\\

Dans la section \ref{sectionsubarsimp}, nous introduirons la notion de substitution d'arbre.
Un arbre simplicial (graphe connexe sans cycle) est la donnée d'un couple $(\mathcal{V}, \mathcal{E})$, où $\mathcal{V}$ est un ensemble de sommets
(pris dans un ensemble non dénombrable quelconque) et $\mathcal{E}$ est
une partie de $\mathcal{V}\times \mathcal{V}\times A_\tau$ (où $A_\tau$ est un alphabet) ; les arêtes sont orientées et colorées par les éléments de $A_\tau$.
L'ensemble des arbres finis (au sens du nombre d'arêtes) ainsi définis est noté $\mathscr{S}_0(A_\tau)$, et
on note $\mathscr{S}_E(A_\tau)$ l'ensemble des arbres de $\mathscr{S}_0(A_\tau)$ constitués d'une unique arête.

\begin{introdefn}[Définition \ref{defn:subarb}]
Une \textbf{substitution d'arbre} est une application $\tau$ de $\mathscr{S}_E(A_\tau)$ dans $\mathscr{S}_0(A_\tau)$ telle que :
\begin{itemize}
	\item pour tout $X\in \mathscr{S}_E(A_\tau)$, les sommets de $X$ sont des sommets de $\tau(X)$,
	\item les images par $\tau$ de deux arbres de $\mathscr{S}_E(A_\tau)$ de même couleur sont égales à renommage des sommets près.
\end{itemize}
\end{introdefn}
L'application $\tau$ s'étend naturellement en une application de $\mathscr{S}_0(A_\tau)$ dans $\mathscr{S}_0(A_\tau)$ en prenant l'union des images des arêtes.
On pourra par exemple se reporter aux figures \ref{fig:introfig} et \ref{fig:emmasubdebut}.\\

Nous construirons une substitution d'arbre associée à $\sigma$. On note $A_{\tau} = \{1, \dots, d, (d+1), \dots, (2d-2)\}$ ;
l'alphabet $A_{\tau}$ contient $A$, ainsi que $d-2$ lettres supplémentaires.
Pour tout $i\in A_{\tau}$, l'arbre $X_i = (\{x, y\}, \{(x, y, i)\})$ est un élément de $\mathscr{S}_E(A_{\tau})$, et on définit
$\tau$ par :
	\begin{itemize}
	\item $\tau (X_1) = X_d$,
	\item l'image de $X_2$ est représentée sur la figure \ref{fig:introfig},
	\item $\tau (X_i) = X_{i-1}$ si $3\le i\le d$,
	\item $\tau (X_{d+1}) = X_1$,
	\item $\tau (X_i) = X_{i-1}$ si $d+2\le i\le 2d-2$.
	\end{itemize}
\begin{figure}[h!]
\begin{center}
	\begin{psfrags}
	\psfrag{1}{\LARGE{$1$}}
	\psfrag{d}{\LARGE{$d$}}
	\psfrag{d1}{\LARGE{$(d+1)$}}
	\psfrag{d2}{\LARGE{$(d+2)$}}
	\psfrag{k}{\LARGE{$k$}}
	\psfrag{kb}{\Large{$d+3\le k\le 2d-3$}}
	\psfrag{2d}{\LARGE{$(2d-2)$}}
	\psfrag{x}{\LARGE{$x$}}
	\psfrag{y}{\LARGE{$y$}}
	\psfrag{Y}{\huge{$\tau(X_2)$}}
	\scalebox{0.65}{\includegraphics{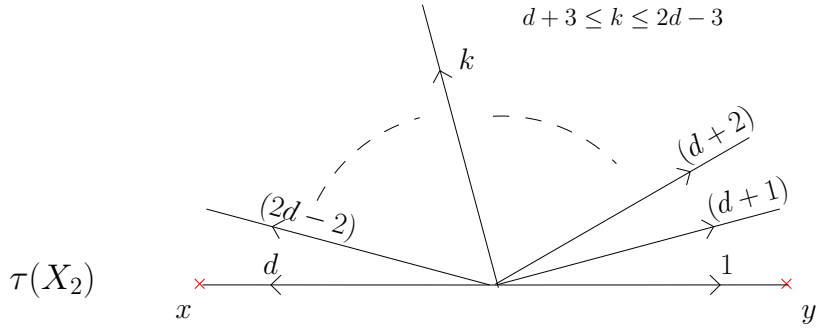}}
	\end{psfrags}
\end{center}
\vspace{-4mm}
\caption{Substitution d'arbre associée à $\sigma$.}
\label{fig:introfig}
\end{figure}
On appelle $T_0^s$ (l'exposant $s$ indique qu'il s'agit d'un arbre simplicial) l'arbre constitué d'un sommet $x_0$
et de $d$ arêtes colorées $1, 2, \dots, d$ sortant de $x_0$. Pour tout $n\in \mathds{N}$, on définit $T_n^s = \tau^n(T_0^s)$,
et on note $\mathcal{V}_n^s$ l'ensemble des sommets de $T_n^s$.
La figure \ref{fig:emmasubdebut} illustre l'action de $\tau$ sur $T_0^s$ (dans le cas $d=3$).
\begin{figure}[h!]
\begin{center}
	\begin{psfrags}
	\psfrag{a}{\LARGE{\textbf{$1$}}}
	\psfrag{b}{\LARGE{\textbf{$2$}}}
	\psfrag{c}{\LARGE{\textbf{$3$}}}
	\psfrag{d}{\LARGE{\textbf{$4$}}}
	\psfrag{x0}{\Large{$x_0$}}
	\psfrag{x1}{\Large{$x_1$}}
	\psfrag{x2}{\Large{$x_2$}}
	\psfrag{x3}{\Large{$x_3$}}
	\psfrag{y1}{\Large{$y_1$}}
	\psfrag{y2}{\Large{$y_2$}}
	\psfrag{z1}{\Large{$z_1$}}
	\psfrag{z2}{\Large{$z_2$}}
	\psfrag{L0}{\huge{$T_0^s$}}
	\psfrag{L1}{\huge{$T_1^s$}}
	\psfrag{L2}{\huge{$T_2^s$}}
	\psfrag{L3}{\huge{$T_3^s$}}
	\scalebox{0.6}{\includegraphics{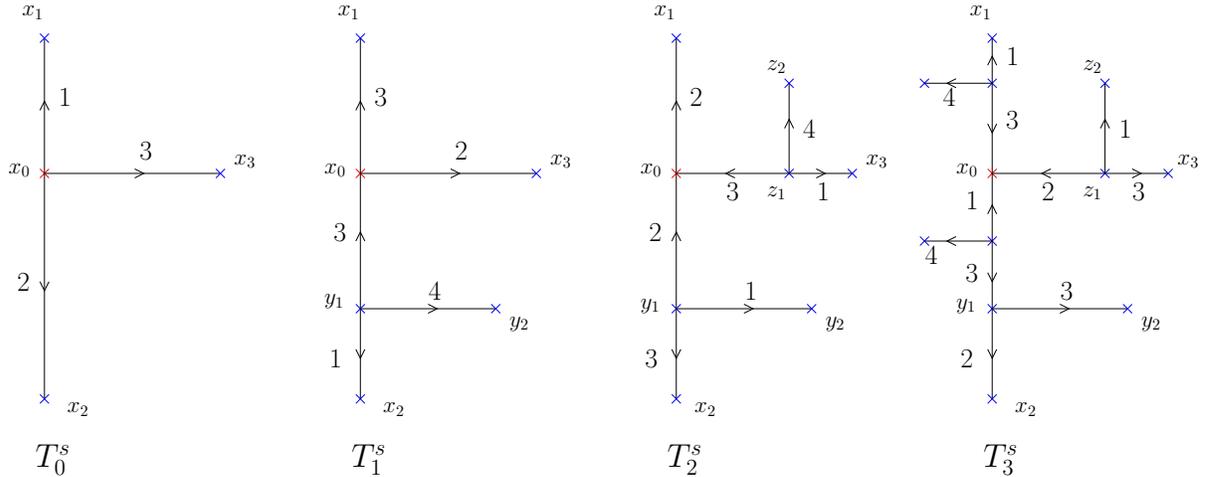}}
	\end{psfrags}
\end{center}
\vspace{-4mm}
\caption{Représentation de $T_0^s, T_1^s, T_2^s, T_3^s$ dans le cas $d=3$.}
\label{fig:emmasubdebut}
\end{figure}

Un arbre réel $T$ est un espace métrique géodésique et $0$-hyperbolique ; le degré d'un point $x$ de $T$ est le nombre de composantes connexes de $T\setminus\{x\}$.
L'espace $\mathscr{R}^d$ (décrit dans la section \ref{sectionsubarel}) est un arbre réel qui contient tous les arbres réels
possédant un nombre fini de points de branchement (points de degré $>2$), et dont les points sont de degré maximal $(2d)$.
Réaliser l'arbre simplicial $T_n^s$ consiste à considérer chaque arête comme un segment de $\mathds{R}$ ; on obtient ainsi une partie de $\mathscr{R}^d$.
Par définition des substitutions d'arbre, $\mathcal{V}_n^s\subset \mathcal{V}_{n+1}^s$ pour tout $n\in \mathds{N}$. Si on impose en plus que
deux arêtes de $T_n^s$ de même couleur sont envoyées sur deux segments de même longueur, alors la suite $(T_n)_n$ de réalisations
est définie de manière unique à homothétie près (\cite[section 5.5]{THE}). La suite $(T_n)_n$ ainsi obtenue est une suite
de Cauchy (pour la distance de Hausdorff) et converge vers un arbre réel compact $T_{\tau}$ de $\overline{\mathscr{R}^d}$, le
complété métrique de $\mathscr{R}^d$.
On définira, en utilisant les propriétés combinatoires de la substitution d'arbre, une application surjective $f_Q : \Omega^+\to T_{\tau}$,
et on verra que $f_Q$ réalise une conjugaison entre le système symbolique et un système d'isométries partielles sur $T_{\tau}$.
La construction de cette application passe par une étude détaillée de la combinatoire du système symbolique et des points de branchement de la substitution d'arbre. On montrera
en particulier que $f_Q$ est une bijection de l'ensemble des décalés du point fixe de la substitution $\sigma$ dans l'ensemble des points
de branchement de $T_{\tau}$, et on mettra en évidence une relation entre ces points de branchement et les facteurs bispéciaux du langage.\\

Une des ambitions de cet article est de poser les bases d'une théorie plus générale. Etant donnée une substitution inversible,
on voudrait pouvoir se servir des propriétés du système dynamique engendré que l'on connaît bien pour construire de
manière systématique une substitution d'arbre associée. A cet effet, on met en évidence, sur l'exemple proposé,
une relation fondamentale entre le système dynamique $(\Omega^+, S)$ et la substitution d'arbre.

Un arc $[s, t]$ (segment géodésique reliant $s$ à $t$) d'un arbre $T_n$ est dit \textbf{simple} si $s$ et $t$ sont les seuls points de $[s, t]$ de degré $\ne 2$ (dans $T_n$).
Tout arbre $T_n$ se décompose naturellement en arcs simples, et cette décomposition se traduit (par application de la substitution d'arbre)
en une partition (modulo un ensemble fini) de $T_{\tau}$. La partition de $T_{\tau}$ ainsi obtenue se relève (par $f_Q^{-1}$) en une partition de $\Omega^+$
qui est dite \textbf{déterminée} par $T_n$.

On note $\mathfrak{L}(\Omega^+)$ l'ensemble des facteurs (sous-mots finis) des éléments de $\Omega^+$.
Pour tout mot $u$ de $\mathfrak{L}(\Omega^+)$, on note $P_u$ l'ensemble des mots $V$ de $\Omega^+$ tels que $uV$ appartient encore à $\Omega^+$
et pour tout $m\in \mathds{N}^*$, on note $\mathscr{P}_m = \{P_u ; |u| = m\}$.

On peut maintenant énoncer les résultats principaux.
Soit $d\ge 3$, soit $\sigma$ la substitution définie par $\sigma(1)=12$, $\sigma(k)=(k+1)$ (pour $2\le k\le d-1$)
et $\sigma(d)=1$, et soit $(\Omega^+, S)$ le système dynamique engendré par $\sigma$. On définit une
substitution d'arbre $\tau$ associée à $\sigma$ et on décrit, grâce aux arbres $T_n^s=\tau^n(T_0^s)$,
une suite $(T_n)_n$ d'arbres réels convergente vers un arbre réel $T_{\tau}$.
\begin{introthm}[Théorème \ref{thm:emmagenarpar}]
Tout arbre $T_n$ détermine une partition $\mathscr{P}_m$
(pour un certain $m$ défini explicitement en fonction de $n$) de $\Omega^+$.
\end{introthm}
On étudie les partitions $\mathscr{P}_m$ atteintes (toutes ne le sont pas). On note $\mu$ l'unique mesure
de probabilité du système symbolique et on s'intéresse au cardinal de l'ensemble $\mathscr{P}_m / \sim$ où $P_u \sim P_v$ lorsque $\mu(P_u) = \mu(P_v)$.
Ce cardinal prend la valeur $(d)$ si $m=1$, la valeur $(2d-2)$ si et seulement si $\mathfrak{L}(\Omega^+)$ contient un mot $u\ne \epsilon$ bispécial
(il existe au moins $2$ prolongements à droite et au moins $2$ prolongements à gauche de $u$ dans $\mathfrak{L}(\Omega^+)$)
de longueur $m-1$, et la valeur $(2d-1)$ sinon. On montre alors le théorème suivant.
\begin{introthm}[Théorème \ref{thm:determination}]
$T_0$ détermine $\mathscr{P}_1$. La partition $\mathscr{P}_m$, $m > 1$, est déterminée par un arbre $T_n$ si et seulement si $\# (\mathscr{P}_m / \sim)=2d-2$.
\end{introthm}\ 

La substitution $\sigma$ considérée est inversible, et on note encore $\sigma$ l'automorphisme
du groupe libre de base $A$ (désormais noté $F_d$) engendré par la substitution, et $\sigma^{-1}$ son inverse.

On note $T_{\Phi^{-1}}$ l'arbre invariant de $\sigma^{-1}$ (comme défini dans \cite{GJLL}).
On utilise les résultats énoncés dans \cite{LL03} pour construire une application $Q$ surjective de $\partial F_d$ dans $\overline{T}_{\Phi^{-1}}\cup \partial T_{\Phi^{-1}}$
(où $\overline{T}_{\Phi^{-1}}$ est le complété métrique de $T_{\Phi^{-1}}$ et $\partial T_{\Phi^{-1}}$ est son bord de Gromov).
On considère $\Omega^+$ comme une partie de $\partial F_d$.
Si $V=aV^\prime\in \Omega^+$ avec $a\in A$, la propriété d'équivariance vérifiée par $Q$ assure notamment que
$Q(V^\prime) = a^{-1}Q(V)$ (ici $a^{-1}$ est la translation associé à l'action du groupe libre sur $T_{\Phi^{-1}}$),
ce qui nous permet à nouveau de représenter géométriquement l'action du décalage sur $\Omega^+$.
\begin{introthm}[Théorème \ref{thm:bijiso}]
Il existe une bijection isométrique de $T_{\tau}$ dans $Q(\Omega^+)$.
\end{introthm}
La démonstration repose sur le fait que l'application $f_Q$ définie précédemment grâce à la substitution d'arbre \og copie \fg~ l'application $Q$ :
si $d_\infty$ est la distance sur $Q(\Omega^+)$ et $d_{T_{\tau}}$ la distance sur $T_{\tau}$, on a
\begin{center}
	$\forall~ V_1, V_2\in \Omega^+,~d_\infty(Q(V_1), Q(V_2)) = d_{T_{\tau}}(f_Q(V_1), f_Q(V_2))$.
\end{center}

On conclut (section \ref{sec:generalisation}) par une discussion sur la généralisation des résultats à d'autres automorphismes.

\section{Combinatoire des mots}\label{sectioncombmots}

	\subsection{Système dynamique symbolique}\label{sec:systdyn}
	On rappelle ici des notions classiques de dynamique symbolique ; pour plus d'informations sur le sujet, on pourra par exemple consulter \cite{Que} ou \cite{Fog}.\\
		On considère un alphabet fini $A$, et on note
		$A^*$ l'ensemble des mots finis sur $A$ (le mot vide est noté $\epsilon$),
		$A^{-\mathds{N}^*}$ l'ensemble des mots infinis à gauche,
		$A^{\mathds{N}}$ l'ensemble des mots infinis à droite, et
		$A^{\mathds{Z}}$ l'ensemble des mots bi-infinis.
		Ces ensembles sont munis de la topologie produit.
		On écrira un mot $W$ de $A^{\mathds{Z}}$ en le pointant entre $W_{-1}$ et $W_0$ ;~~ $W = \dots W_{-2} W_{-1} .W_0 W_1 W_2\dots$.

		Soit $w = w_0w_1\dots w_p$ un mot de $A^*$. La longueur de $w$ est $|w| = p+1$.
		Le mot $w$ est \textbf{préfixe} d'un mot $v\in A^*$ (resp. $V\in A^{\mathds{N}}$) s'il existe $v^\prime\in A^*$ (resp. $V^\prime\in A^{\mathds{N}}$)
		tel que $v = wv^\prime$ (resp. $V = wV^\prime$).
		Le mot $w$ est \textbf{suffixe} d'un mot $u\in A^*$ (resp. $U\in A^{-\mathds{N}^*}$) s'il existe $u^\prime\in A^*$ (resp. $U^\prime\in A^{-\mathds{N}^*}$)
		tel que $u = u^\prime w$ (resp. $U = U^\prime w$).

		Le \textbf{langage} $\mathfrak{L}(W)$ d'un mot $W$ est l'ensemble de tous les mots finis qui apparaissent dans
		$W$. Un élément de $\mathfrak{L}(W)$ est appelé \textbf{facteur} de $W$.
		On dit qu'un facteur $u$ de $\mathfrak{L}(W)$ est \textbf{spécial à gauche} (resp. \textbf{à droite}) s'il existe deux éléments distincts $a$ et $b$ de $A$ tels que
		les mots $au$ et $bu$ (resp. $ua$ et $ub$) sont encore dans $\mathfrak{L}(W)$. Un facteur est \textbf{bispécial} s'il est à la fois spécial à gauche et
		à droite.

		On note $S$ l'application \textbf{décalage} sur $A^\mathds{N}$ (ou $A^\mathds{Z}$) qui à tout mot
		$W = (W_i)_{i}$ (pour $i\in\mathds{N}$ ou $\mathds{Z}$) associe le mot $S(W) = (W_{i+1})_{i}$.
		Le \textbf{système dynamique bilatère} engendré par un mot $W$ est le couple $(\Omega (W), S)$,
		où $\Omega(W) = \{W^\prime \in A^\mathds{Z} ; \mathfrak{L}(W^\prime) \subset \mathfrak{L}(W) \}$. Notons que l'ensemble $\Omega (W)$ est
		l'adhérence dans $A^\mathds{Z}$ de l'orbite
		de $W$ sous l'action de $S$ ; il est compact pour la topologie induite par celle de $A^\mathds{Z}$ et la restriction de
		$S$ à $\Omega (W)$, encore notée $S$, est un homéomorphisme.

		On définit de même un \textbf{système dynamique unilatère} $\Omega^+ (W)$ en se plaçant dans $A^\mathds{N}$ (dans ce cas $S$ n'est pas une bijection).
		Un mot $V$ de $\Omega^+ (W)$ sera dit \textbf{spécial à gauche} s'il existe deux éléments distincts $a$ et $b$ de $A$ tels que
		les mots $aV$ et $bV$ sont encore dans $\Omega^+ (W)$.\\

		Une \textbf{substitution} est un morphisme $\sigma$ pour la concaténation du monoïde libre $A^*$, qui envoie
		$A$ sur $A^*\setminus \{\epsilon\}$, et tel qu'il existe une lettre $a$ de $A$ pour laquelle $\lim\limits_{n\to +\infty}|\sigma^n(a)| = +\infty$.
		La substitution se prolonge de manière naturelle aux ensemble $A^{-\mathds{N}^*}$, $A^{\mathds{N}}$ et $A^\mathds{Z}$ par concaténation.

			Si $\sigma$ est une substitution sur $\{1, 2, \dots, d\}$, on note $M_\sigma$ la \textbf{matrice d'incidence} de $\sigma$, dont
			le coefficient $(i, j)$ est le nombre d'occurrences de la lettre $i$ dans $\sigma(j)$.

			Une matrice est dite \textbf{primitive} s'il existe une puissance de cette matrice dont les coefficients sont tous strictement positifs.
			Le théorème de Perron-Frobenius implique que $M_\sigma$ admet alors une valeur propre dominante simple,
			qui est réelle positive et un vecteur propre associé
			à cette valeur propre à coefficients strictement positifs.
			Une substitution est \textbf{primitive} si et seulement si sa matrice d'incidence est primitive.

			Toute substitution primitive $\sigma$ possède un mot périodique bi-infini (mot $W\in A^{\mathds{Z}}$ tel que $\sigma^k(W) = W$ pour un certain $k\in \mathds{N}^*$).
			Soit $W$ un mot périodique par $\sigma$. Par primitivité, l'ensemble  $\Omega(W)$ ne dépend pas de $W$ ; on le note $\Omega$.
			Le couple $(\Omega, S)$ est le \textbf{système dynamique symbolique} engendré par $\sigma$. De même, $\mathfrak{L}(W)$ ne dépend pas
			de $W$ et on le note $\mathfrak{L}(\Omega)$. Si la substitution est primitive, le système $(\Omega, S)$ est minimal (il ne
			possède pas de fermé invariant par $S$ non trivial) et uniquement ergodique (il existe une unique mesure de probabilité invariante
			par $S$) (voir \cite{Que}).

			Le \textbf{système dynamique symbolique unilatère} possède des propriétés similaires, si ce n'est que le
			décalage $S$ n'est pas une bijection sur cet ensemble. On note que la projection canonique du système bilatère sur le
			système unilatère est injective sauf sur un ensemble dénombrable où elle est fini-à-un (\cite{Que}).

		\subsection{Automate et développement en préfixes-suffixes}\label{sec:adps}

			On pourra se référer à \cite[chapitre 2]{Sie} et \cite{CanSie} pour plus de détails sur cette section.\\
			Soit $\sigma$ une substitution primitive non $S$-périodique (si $W\in \Omega$ vérifie $\sigma^k(W) = W$ pour un certain $k\in \mathds{N}^*$,
			alors pour tout $h\in \mathds{N}^*$, $S^h(W)\ne W$)
			sur un alphabet de cardinal fini $A$ et soit $(\Omega, S)$, $\Omega\subset A^{\mathds{Z}}$, le système symbolique
			engendré par $\sigma$. On définit l'automate $A_\sigma$ associé à la substitution $\sigma$ par :
				\begin{itemize}
				\item $A$ est l'ensemble des états ; tous les états sont initiaux,
				\item $P = \{(p, a, s)\in A^*\times A\times A^* ; \exists b\in A ; \sigma(b) = pas\}$ est l'ensemble des couleurs,
				\item il existe une flèche entre les états $a$ et $b$ colorée par $e=(p, a, s)$ si $\sigma(b) = pas$.
				\end{itemize}
			Cet automate est fortement connexe si $\sigma$ est primitive. Un exemple est donné figure \ref{fig:apsex}.

			\begin{figure}[h!]
				\begin{center}
				\begin{psfrags}
					\psfrag{1}{\LARGE{$1$}}
					\psfrag{2}{\LARGE{$2$}}
					\psfrag{3}{\LARGE{$3$}}
					\psfrag{e1e}{\Large{$(\epsilon, 1, \epsilon)$}}
					\psfrag{e3e}{\Large{$(\epsilon, 3, \epsilon)$}}
					\psfrag{12e}{\Large{$(1, 2, \epsilon)$}}
					\psfrag{e12}{\Large{$(\epsilon, 1, 2)$}}
					\scalebox{0.6}{\includegraphics{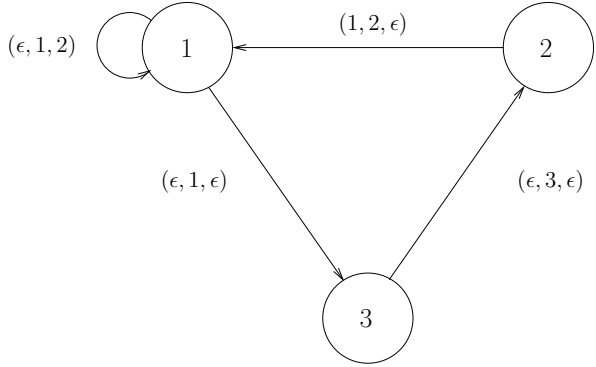}}
				\end{psfrags}
				\end{center}
				\vspace{-4mm}
				\caption{Automate des préfixes-suffixes associé à $\sigma : 1\mapsto 12, 2\mapsto 3, 3\mapsto 1$.}
				\label{fig:apsex}
			\end{figure}

			Un élément $(e_i)_{i\ge 0}\in P^\mathds{N}$ est dit \textbf{admissible} s'il s'agit d'un mot infini reconnu par l'automate des
			préfixes-suffixes. On note $D$ l'ensemble des éléments admissibles de $P^\mathds{N}$.
			\begin{prop}\label{prop:coolprefsuff}
			Si $(p_i, a_i, s_i)_{i\ge 0}$ est un élément de $D$, alors pour tout
			$i\ge 0$, $\sigma(a_{i+1}) = p_ia_is_i$ et pour tout $k\in \mathds{N}$,
			$\sigma^{k}(a_k) = \sigma^{k-1}(p_{k-1})\dots \sigma(p_1)p_0a_0s_0\sigma(s_1)\dots \sigma^{k-1}(s_{k-1})$.
			Notamment, $s_0\sigma(s_1)\dots \sigma^{k-1}(s_{k-1})$ est un suffixe strict de $\sigma^{k}(a_k)$ et
			$\sigma^{k-1}(p_{k-1})\dots \sigma(p_1)p_0$ est un préfixe strict de $\sigma^{k}(a_k)$.
			\end{prop}

			Dans \cite{CanSie}, V. Canterini et A. Siegel explicitent une application $\Gamma: \Omega\to D$ qui met en évidence
			la structure auto-similaire de $\Omega$. Si $W\in \Omega$, $W = U.V$ (où $U\in A^{-\mathds{N}^*}$ et $V\in A^{\mathds{N}}$)
			et $\Gamma(W) = (p_i, a_i, s_i)_{i\in \mathds{N}}$,
			\begin{itemize}
				\item si $(s_i)_{i\in \mathds{N}}$ n'est pas ultimement constante égale à $\epsilon$, alors
				$V = \lim\limits_{n\to +\infty} a_0s_0\sigma(s_1)\dots \sigma^n(s_n)$,
				\item si $(p_i)_{i\in \mathds{N}}$ n'est pas ultimement constante égale à $\epsilon$, alors
				$U = \lim\limits_{n\to +\infty} \sigma^n(p_n)\dots \sigma(p_1)p_0$.
			\end{itemize}
			Les différents cas sont abordés au paragraphe suivant.
			\begin{thm}[\cite{CanSie}]\label{thm:gamps}
			Si $\sigma$ est primitive et non $S$-périodique, l'application $\Gamma$ est continue surjective de $\Omega$ sur $D$.
			Elle est injective sur $\Omega\setminus \bigcup\limits_{n\in \mathds{Z}} S^n(\Omega_{per})$ (où $\Omega_{per}\subset \Omega$
			est l'ensemble des mots périodiques par $\sigma$) et est donc injective en mesure.
			De plus, pour tout $d$ de $D$, $\#(\Gamma^{-1}(\{d\}))\le \#(\Omega_{per})$.
			\end{thm}
			La suite $\Gamma(W)$ est le \textbf{développement en préfixes-suffixes} de $W$.\\

			On note $D_{min}, D_{max}$ et $D_\epsilon$ les sous-ensembles de $D$ définis par :
			\begin{itemize}
				\item $D_{min} = \{(p_i, a_i, s_i)_{i\ge 0} \in D ; \forall i\in \mathds{N}, p_i = \epsilon\}$
				\item $D_{max} = \{(p_i, a_i, s_i)_{i\ge 0} \in D ; \forall i\in \mathds{N}, s_i = \epsilon\}$
				\item $D_\epsilon = \{(p_i, a_i, s_i)_{i\ge 0} \in D ; (\exists i_0\in \mathds{N} ; \forall i\ge i_0, p_i=\epsilon)$
				ou $(\exists i_0\in \mathds{N} ; \forall i\ge i_0, s_i=\epsilon)$\}.
			\end{itemize}\vspace{1mm}
			L'ensemble des mots bi-infinis périodiques par $\sigma$ est encore noté $\Omega_{per}$.
			\begin{thm}[\cite{CanSie}]\label{thm:prgamma}
			Les égalités suivantes sont vérifiées.\\
			\begin{itemize}
			\item $\Gamma(\Omega_{per}) = D_{min}$ et $\Gamma^{-1}(D_{min}) = \Omega_{per}$\\
			\item $\Gamma(S^{-1}(\Omega_{per})) = D_{max}$ et $\Gamma^{-1}(D_{max}) = S^{-1}(\Omega_{per})$\\
			\item $\Gamma(\bigcup\limits_{n\in \mathds{Z}} S^n(\Omega_{per})) = D_\epsilon$ et $\Gamma^{-1}(D_\epsilon) = \bigcup\limits_{n\in \mathds{Z}} S^n(\Omega_{per})$.
			\end{itemize}
			\end{thm}

	\subsection{Le groupe libre et son bord}
	On pourra se reporter à \cite{LS} pour une introduction aux groupes libres, et à \cite{BeKa} pour des définitions et résultats concernant
	les bords des groupes hyperboliques.

	On note $F_n$ le groupe libre de rang $n\ge 2$. Une base $A$ étant fixé, on peut voir $F_n = F(A)$ comme l'ensemble des mots (finis)
	$u=u_0\dots u_k$, $(u_i\in A\cup A^{-1})$ réduits ($u_i\ne u_{i+1}^{-1}$) ; l'élément neutre est noté $\epsilon$. La loi du groupe
	est la concaténation-réduction.

	Le bord (de Gromov) $\partial F_n$ de $F_n$ peut être vu comme l'ensemble $\partial F(A)$ des mots $V = V_0\dots V_k\dots$ infinis à droite,
	réduits sur l'alphabet $A\cup A^{-1}$. On munit $A\cup A^{-1}$ de la topologie discrète et $(A\cup A^{-1})^{\mathds{N}}$ de la topologie
	produit. Le bord $\partial F_n\subset (A\cup A^{-1})^{\mathds{N}}$ hérite de la topologie induite : c'est un ensemble de Cantor.

	Un mot $w=w_0w_1\dots w_p\in F(A)$ est \textbf{préfixe} d'un mot $v=v_0v_1\dots v_q\in F(A)$ (resp. $V = (V_i)_{i\in \mathds{N}}\in \partial F(A)$)
	si pour tout $j$, $(0\le j\le p)$, $w_j = v_j$ (resp. $w_j = V_j$).
	Le mot $w$ est \textbf{suffixe} d'un mot $u=u_0u_1\dots u_r\in F(A)$ si pour tout $j$, $(r-p\le j\le r)$, $w_j = u_j$.

	Le groupe
	libre agit (continûment) sur son bord par translations à gauche : si $u=u_0\dots u_k \in F(A)$ et $V = V_0\dots V_k\dots \ \in \partial F(A)$, alors
	$uV = u_0\dots u_{k-i-1}V_{i+1}\dots V_k\dots \ \in \partial F(A)$ où $V_0\dots V_i = u_k^{-1}\dots u_{k-i}^{-1}$ est le plus long préfixe
	commun à $u^{-1}$ et $V$.

\section{Substitutions d'arbre et arbres réels}\label{sectionsubarsimp}

	Le but de cette section est d'introduire des opérations de substitution sur des arbres simpliciaux ; on peut voir celles-ci comme
	des extensions des substitutions usuelles (morphismes de monoïdes). Elles sont étudiées plus en détails dans \cite[chapitre 4]{THE} ; ici,
	on s'en servira pour construire des suites d'arbres simplicaux.

	\subsection{Arbres simpliciaux}\label{subsec:arbsimp}
	Un \textbf{graphe orienté} $G$ se compose d'un ensemble $\mathcal{V}$ de sommets et d'un ensemble $\mathcal{E}$ d'arêtes, $\mathcal{E}$ étant une partie
	de $\mathcal{V}\times \mathcal{V}$. Le \textbf{degré} d'un sommet $x$ est le nombre d'arêtes adjacentes à $x$.
	\begin{itemize}
	\item Un \textbf{chemin} de $G$ est une liste $p=(x_0, x_1, \dots, x_k)$ de sommets de $\mathcal{V}$ telle que
	pour tout $0\le i \le k-1$, on a $((x_i, x_{i+1})\in \mathcal{E}$ ou $(x_{i+1}, x_i)\in \mathcal{E})$.
	\item On appelle \textbf{cycle} un chemin $p=(x_0, x_1, \dots, x_k)$ tel que $x_0=x_k$.
	\item On dit qu'un graphe orienté est \textbf{connexe} si et seulement si il existe un chemin entre chaque paire de sommets.
	\end{itemize}
	Un \textbf{arbre simplicial (orienté)} est un \textbf{graphe (orienté) connexe sans cycle}. Si $x$ et $x'$ sont deux sommets quelconques d'un arbre
	simplicial, il existe un unique chemin minimal (au sens du nombre de sommets) $(x=x_0, x_1, \dots, x_k=x')$ reliant $x$ à $x'$ ; l'entier $k$
	est la \textbf{distance} entre $x$ et $x'$.\\

		On considère un alphabet $A$ de cardinal fini.
		Etant donné un ensemble $\mathscr{V}$ infini non dénombrable quelconque, on appelle \textbf{graphe coloré par $A$ à sommets dans $\mathscr{V}$}
		un couple $(\mathcal{V}, \mathcal{E})$ tel que $\mathcal{V}$ est une partie de $\mathscr{V}$, et $\mathcal{E}$ est une partie de $\mathcal{V}\times \mathcal{V}\times A$.

		On appelle $\mathscr{S}_0(A)$ l'ensemble des arbres finis (au sens du nombre d'arêtes), orientés, colorés par $A$ et à sommets dans $\mathscr{V}$.
		On note $\mathscr{S}_E(A)$ l'ensemble des arbres de $\mathscr{S}_0(A)$ ne possédant qu'une seule arête.

		Dans la suite de cette section, pour tout élément $X$ de $\mathscr{S}_0(A)$, l'ensemble des sommets de $X$ sera noté $\mathcal{V}_X$ et l'ensemble des arêtes colorées
		de $X$ sera noté $\mathcal{E}_X$.\\

		Si $A=\{a_0,\dots a_d\}$, on note $\overline{A} = \{\overline{a_0},\dots \overline{a_d}\}$.
		A tout élément $X$ de $\mathscr{S}_0(A)$, on associe une fonction
		\[
			\gamma_X : \mathcal{V}_X\times \mathcal{V}_X \to (A\cup\overline{A})^*
		\]
		où $(A\cup\overline{A})^*$ est le monoïde libre engendré par $A\cup\overline{A}$.
		Si $(x, x^{\prime}, a)$ est une arête de $X$, alors $\gamma_X(x, x^{\prime}) = a$ et $\gamma_X(x^{\prime}, x) = \overline{a}$.
		Si $(x, x^{\prime})$ est une paire de sommets quelconques de $X$ et $(x, x_1, x_2, \dots, x_{k-2}, x_{k-1}, x_k=x^{\prime})$ est la liste des sommets formant le chemin le plus
		court de $x$ à $x^{\prime}$, alors on définit
		\begin{center}
			$\gamma_X(x, x^{\prime}) = \gamma_X(x, x_1)\gamma_X(x_1, x_2)\dots \gamma_X(x_{k-2}, x_{k-1})\gamma_X(x_{k-1}, x^{\prime})$.
		\end{center}
		La fonction $\gamma_X$ est appelée \textbf{fonction chemin} de $X$.\\

		On introduit la notion de discernement, qui sera essentielle par la suite. On considère un arbre $X$ de $\mathscr{S}_0(A)$,
		et on choisit arbitrairement un sommet $x_0$ de $X$. La propriété de discernement a pour but d'assurer que tout sommet $x$ de $X$ est déterminé de manière unique par le
		mot $\gamma_X(x_0, x)$.
		\begin{defn}
		Si pour tout couple $x, x^\prime$ de sommets d'un arbre $X$ de $\mathscr{S}_0(A)$, le mot $\gamma_{X} (x, x^\prime)$
		de $(A\cup \overline{A})^*$ ne contient aucune paire $\overline{\alpha}\alpha$,
		pour $\alpha\in A$, alors on dit de $X$ que c'est un arbre \textbf{discerné}.
		\end{defn}

        \subsection{Substitutions d'arbre}
	Nous introduisons maintenant l'objet principal de cet article : les substitutions d'arbre.
	\begin{defn}\label{defn:subarb}
	Une \textbf{substitution d'arbre} est une application $\tau$ de $\mathscr{S}_E(A)$ dans $\mathscr{S}_0(A)$ vérifiant les propriétés suivantes.
	\begin{itemize}
		\item Si $X=(\{x_1, x_2\}, \{(x_1, x_2, a)\})$ est un élément de $\mathscr{S}_E(A)$, d'image $\tau(X)$, alors
		$x_1, x_2\in \mathcal{V}_{\tau(X)}$.
		\item Pour toute lettre $a$ de $A$, si $X=(\{x_1, x_2\}, \{(x_1, x_2, a)\})$ et $Y=(\{y_1, y_2\}, \{(y_1, y_2, a)\})$ sont deux éléments de $\mathscr{S}_E(A)$, alors
		il existe une bijection $f$ de $\mathcal{V}_{\tau(X)}$ dans $\mathcal{V}_{\tau(Y)}$ vérifiant
		\begin{itemize}
			\item $f(x_1) = y_1$ et $f(x_2) = y_2$,
			\item pour tout couple $(x, x^{\prime})$ de $\mathcal{V}_{\tau(X)}$, $\gamma_{\tau(X)}(x, x^{\prime}) = \gamma_{\tau(Y)}(f(x), f(x^{\prime}))$.
		\end{itemize}
		\item Soient $X$ et $Y$ deux éléments quelconques distincts de $\mathscr{S}_E(A)$, $\tau(X)$ et $\tau(Y)$ étant leurs images
		respectives par l'application $\tau$. Alors
		\begin{itemize}
			\item $(\mathcal{V}_{\tau(X)}\setminus \mathcal{V}_X) \cap \mathcal{V}_{\tau(Y)} = \emptyset$, et
			\item $(\mathcal{V}_{\tau(Y)}\setminus \mathcal{V}_Y) \cap \mathcal{V}_{\tau(X)} = \emptyset$.
		\end{itemize}
		\item Pour tout $a\in A$, on note $X_{a}=(\{x_1, x_2\}, \{(x_1, x_2, a)\})$. S'il existe une liste
		$(a_1, \dots, a_{k-1}, a_k=a_1)$ ($k\le \# A +1$) d'éléments de $A$ telle que pour tout $1\le j\le k-1$,
		$(x_1, x_2, a_{j+1})$ ou (exclusif) $(x_2, x_1, a_{j+1})$ est une arête de $\tau (X_{a_j})$, alors
		les degrés de $x_1$ et $x_2$ dans $\tau (X_{a_j})$ sont égaux à $1$ quel que soit $1\le j\le k-1$.
	\end{itemize}
	\end{defn}
	On étend maintenant la définition de substitution d'arbre : la nouvelle application, également
	notée $\tau$, est une application de $\mathscr{S}_0(A)$ dans $\mathscr{S}_0(A)$.
	Soit $X$ un élément de $\mathscr{S}_0(A)$. On note $t$ l'application qui à tout élément $(x_1, x_2, a)$ de $\mathcal{E}_X$ associe l'élément
	$(\{x_1, x_2\}, \{(x_1, x_2, a)\})$ de $\mathscr{S}_E(A)$. L'image de $X$ par $\tau$ est alors :
	\begin{center}
	$\tau(X) = (\bigcup\limits_{p\in \mathcal{E}_X} \mathcal{V}_{\tau(t(p))}, \bigcup\limits_{p\in \mathcal{E}_X} \mathcal{E}_{\tau(t(p))})$.
	\end{center}
	Un exemple est donné figure \ref{fig:subex}.\\

	\begin{figure}[h!]
	\begin{center}
	\begin{psfrags}
		\psfrag{Xa}{\huge{$X_a$}}
		\psfrag{Xb}{\huge{$X_b$}}
		\psfrag{Xc}{\huge{$X_c$}}
		\psfrag{Xd}{\huge{$X_d$}}
		\psfrag{Ya}{\huge{$\tau(X_a)$}}
		\psfrag{Yb}{\huge{$\tau(X_b)$}}
		\psfrag{Yc}{\huge{$\tau(X_c)$}}
		\psfrag{Yd}{\huge{$\tau(X_d)$}}
		\psfrag{Zb}{\huge{$\tau^2(X_b)$}}
		\psfrag{ZZb}{\huge{$\tau^3(X_b)$}}
		\psfrag{x}{\LARGE{$x_1$}}
		\psfrag{y}{\LARGE{$x_2$}}
		\psfrag{z1}{\LARGE{$z_1$}}
		\psfrag{z2}{\LARGE{$z_2$}}
		\psfrag{a}{\huge{$a$}}
		\psfrag{b}{\huge{$b$}}
		\psfrag{c}{\huge{$c$}}
		\psfrag{d}{\huge{$d$}}
		\scalebox{0.64}{\includegraphics{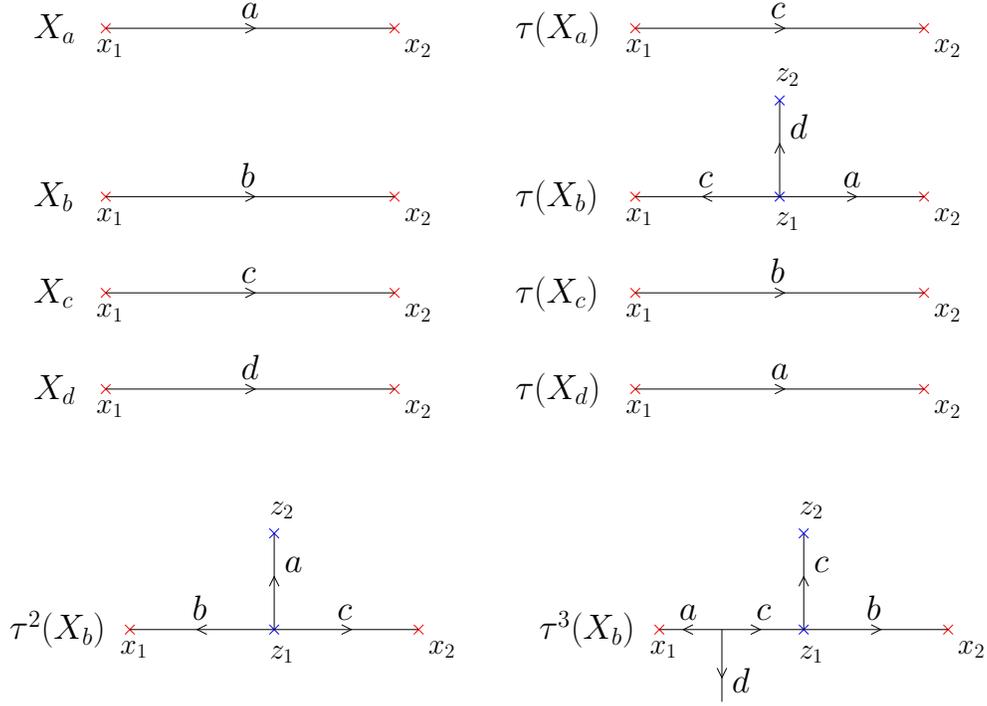}}
	\end{psfrags}
	\end{center}
	\vspace{-4mm}
	\caption{Un exemple de substitution d'arbre.}
	\label{fig:subex}
	\end{figure}

	\begin{defn}
	Soit $A = \{1, \dots, d\}$ un alphabet, $\tau$ une substitution d'arbre sur $\mathscr{S}_0(A)$ et $X_a=(\{x_1, x_2\}, \{(x_1, x_2, a)\})$ un arbre de $\mathscr{S}_E(A)$.
	On appelle \textbf{tronc} de $\tau(X_a)$ le sous-arbre uniquement constitué des arêtes du chemin de $x_1$ à $x_2$. La
	\textbf{matrice tronc} $M_t$ est la matrice $d\times d$ telle que $M_t(i, j)$ est le nombre
	d'arêtes de couleur $i$ dans le tronc de $\tau(j)$.
	\end{defn}
	La matrice tronc de l'exemple donné en figure \ref{fig:subex} est
	$$
	\begin{bmatrix}
		0 & 1 & 0 & 1\\
		0 & 0 & 1 & 0\\
		1 & 1 & 0 & 0\\
		0 & 0 & 0 & 0
	\end{bmatrix}
	$$

	\subsection{Arbres réels}\label{sectionsubarel}

	Un \textbf{arbre réel} $T$ est un espace métrique où deux points quelconques sont joints par un unique arc,
	et cet arc est isométrique à un segment de $\mathds{R}$. On notera $[s, t]$ l'arc joignant $s$ à $t$.
	Le \textbf{degré} d'un point $x$ dans un arbre $T$ est le nombre de composantes connexes de $T\setminus \{x\}$.
	On dit que $x$ est un \textbf{point de branchement} s'il est de degré supérieur ou égal à $3$ ; c'est un \textbf{point terminal}
	s'il est de degré $1$.

	Dans cette section, on présente un espace métrique $\mathscr{R}^k$ adapté aux constructions d'arbres réels au moyen de substitutions d'arbre.
	L'ensemble $\mathscr{R}^k$ est un arbre réel qui contient notamment tous les arbres réels possédant un nombre fini de points de branchement,
	et dont les points sont de degré maximal $(2k)$.
	Par la suite, on se servira des substitutions d'arbre pour définir des suites d'arbres simpliciaux et on
	réalisera ces suites en leurs associant des suites d'arbres réels de $\mathscr{R}^k$. On verra dans la section
	\ref{sectionclasse} que ces suites peuvent être convergentes (pour la distance de Hausdorff) vers des arbres réels compacts de $\overline{\mathscr{R}^k}$,
	le complété métrique de $\mathscr{R}^k$.\\

		Pour tout $k\in \mathds{N}^*$, le \textbf{produit libre} $\mathds{R}*\dots *\mathds{R}$ de $k$ copies de $\mathds{R}$ est noté $\mathscr{R}^k$.
		L'ensemble $\mathscr{R}^k$ est un groupe dans lequel tout élément $t$ (différent de l'origine) a une unique écriture réduite (finie)
		$x_0^{t_0}x_1^{t_1}\dots x_q^{t_q}$, avec
		\begin{itemize}
			\item $x_i\in \{0,\dots, k-1\}$ et $t_i\in \mathds{R}^*$ pour tout $0\le i\le q$,
			\item $x_i\ne x_{i+1}$ pour tout $0\le i\le q-1$.
		\end{itemize}
		Cette notation distingue implicitement les copies de $\mathds{R}$, bien qu'elles jouent le même rôle ;
		la copie $j$ sera notée $\mathds{R}_j$. L'origine est notée $O$.

		On définit la distance $d$ invariante par translation à gauche sur $\mathscr{R}^k$ telle que tout point
		$x_0^{t_0}x_1^{t_1}\dots x_q^{t_q}$ (en écriture réduite) est à distance $|t_0|+|t_1|+\dots +|t_q|$ de l'origine.
		L'élément inverse de $t=x_0^{t_0}x_1^{t_1}\dots x_q^{t_q}$ est $t^{-1}=x_q^{-t_q}\dots x_1^{-t_1}x_0^{-t_0}$.
		Si $s$ et $t$ sont deux éléments
		de $\mathscr{R}^k$, on obtient de l'invariance par translation à gauche que $d(s, t)=d(O, s^{-1}t)$ ; l'écriture
		réduite de $s^{-1}t$ permet alors d'obtenir la distance de $s$ à $t$. L'ensemble $\mathscr{R}^k$ muni de cette distance est un
		arbre réel.

		On notera $\overline{\mathscr{R}^k}$ le complété métrique de $\mathscr{R}^k$.
		En plus de $\mathscr{R}^k$, l'ensemble $\overline{\mathscr{R}^k}$ contient donc tous les points $t$ dont l'écriture réduite
		$x_0^{t_0}x_1^{t_1}\dots x_n^{t_n}\dots$ est infinie à droite et vérifie
		\begin{itemize}
			\item $x_i\in \{0,\dots, k-1\}$, $t_i\in \mathds{R}^*$ et $x_i\ne x_{i+1}$ pour tout $i\in \mathds{N}$,
			\item $\sum\limits_{n\ge 0} |t_n|$ est finie.
		\end{itemize}
		L'espace métrique $(\overline{\mathscr{R}^k}, d)$ est également un arbre réel.

		On munit l'ensemble des parties de $\overline{\mathscr{R}^k}$
		de la distance de Hausdorff associée à $d$. Cette distance est notée $\delta$.
		L'espace métrique $(\overline{\mathscr{R}^k}, d)$ étant complet, l'ensemble des
		compacts non vide de $\overline{\mathscr{R}^k}$ est complet pour $\delta$.

		On notera enfin $\mathscr{T}^k$ l'ensemble des compacts connexes non vides de $\overline{\mathscr{R}^k}$.
		Les éléments de $\mathscr{T}^k$ sont des arbres réels.

		\begin{prop}\label{completree}
		Quel que soit $k\in \mathds{N}^*$, $(\mathscr{T}^k, \delta)$ est complet.
		\end{prop}
		\begin{proof}
		Soit $(T_n)_n$ une suite de Cauchy d'élements de $\mathscr{T}^k$. L'ensemble des compacts non vides de
		$\overline{\mathscr{R}^k}$ étant complet pour la distance de Hausdorff, il existe un compact $T$ limite de la suite $(T_n)_n$.

		Il s'agit maintenant de vérifier que $T$ est connexe. On suppose qu'il ne l'est pas ; il existe $P_1, P_2$ deux
		compacts de $\overline{\mathscr{R}^k}$ et $\epsilon > 0$ tels que $T=P_1\cup P_2$ et
		$\delta ([P_1]_\epsilon, [P_2]_\epsilon) > 0$. Il existe $n_0$ tel que pour tout $n\ge n_0$,
		$\delta (T_n, T) < \epsilon$. On déduit de la connexité des $T_n$ que pour tout $n\ge n_0$,
		$T_n\subset [P_1]_\epsilon$ ou (exclusif) $T_n\subset [P_2]_\epsilon$, ce qui est impossible. 
		\end{proof}

\section{Construction de c{\oe}urs compacts par substitutions d'arbre}\label{sectionclasse}

	Soit $d$ un entier supérieur ou égal à $3$ et $A=\{1, 2, \dots, d\}$ un alphabet. On note $\sigma$ la substitution
	définie sur $A$ par
	\begin{center}
		\begin{tabular}{cccclc}
		$\sigma$ & : & $1$ & $\mapsto$ & $12$ & \\
		&& $k$ & $\mapsto$ & $(k+1)$ & pour $2\le k\le d-1$\\
		&& $d$ & $\mapsto$ & $1$ & \\
		\end{tabular}
	\end{center}
	La matrice d'incidence $M_\sigma$ de $\sigma$ est définie par
	\begin{figure}[h!]
	\begin{center}
		\begin{psfrags}
		\psfrag{0}{\huge{$0$}}
		\psfrag{1}{\huge{$1$}}
		\psfrag{M}{\huge{$M_\sigma \ =$}}
		\scalebox{0.6}{\includegraphics{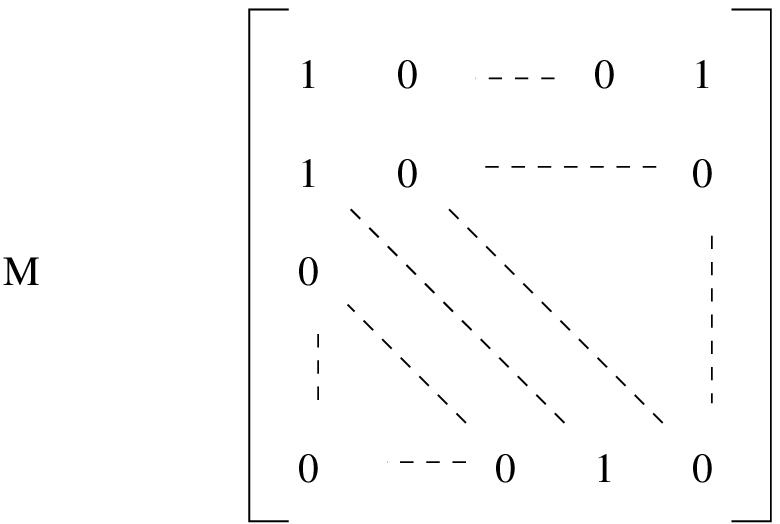}}
		\end{psfrags}
	\end{center}
	\vspace{-4mm}
	\label{fig:incmat}
	\end{figure}\ \\
	Cette matrice est primitive et admet donc une valeur propre réelle dominante $\lambda$ ($>1$) qui vérifie $\lambda^d = \lambda^{d-1}+1$.

	La substitution $\sigma$ est primitive quel que soit $d$ ; on note $(\Omega, S)$ (resp. $(\Omega^+, S)$) le système dynamique symbolique
	bilatère (resp. unilatère) engendré par $\sigma$ et $\mathfrak{L}(\Omega)$ son langage.

	On verra également $\sigma$ comme un automorphisme de $F_d = F(A)$, le groupe libre de base $A$.
	L'application inverse à $\sigma$ est définie par
	\begin{center}
		\begin{tabular}{ccccl}
		$\sigma^{-1}$ & : & $1$ & $\mapsto$ & $d$\\
		&& $2$ & $\mapsto$ & $d^{-1}1$\\
		&& $k$ & $\mapsto$ & $(k-1)$ pour $3\le k\le d$.
		\end{tabular}
	\end{center}
	La coordonnée $(i, j)$ de la matrice d'incidence $M_{\sigma^{-1}}$ de $\sigma^{-1}$ est la somme des nombres
	d'occurrences de $i$ et $i^{-1}$ dans $\sigma^{-1}(j)$ (elle est donc différente de $M_{\sigma}^{-1}$).
	\begin{figure}[h!]
	\begin{center}
		\begin{psfrags}
		\psfrag{0}{\huge{$0$}}
		\psfrag{1}{\huge{$1$}}
		\psfrag{M}{\huge{$M_{\sigma^{-1}} \ =$}}
		\scalebox{0.6}{\includegraphics{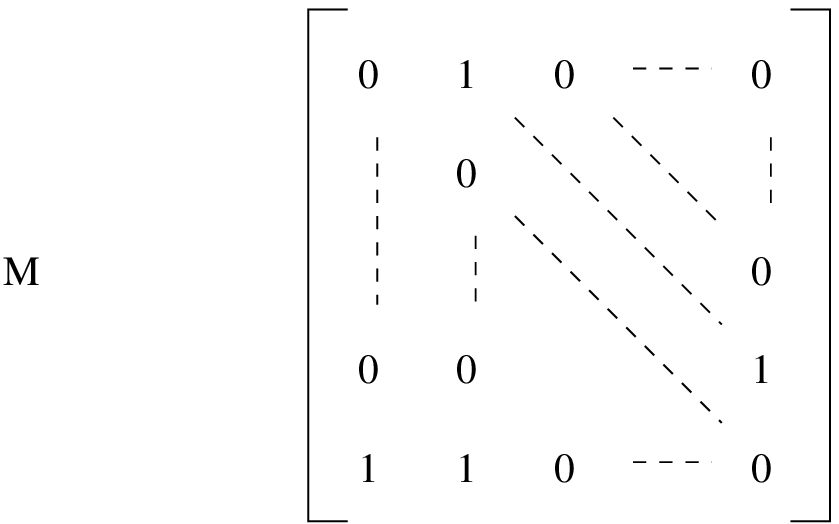}}
		\end{psfrags}
	\end{center}
	\vspace{-4mm}
	\label{fig:incmatinv}
	\end{figure}\ \\
	Cette matrice est primitive et admet une valeur propre réelle dominante $\eta$ ($>1$) qui vérifie $\eta^d = \eta+1$.
	Le vecteur $\mathbf{V_{\sigma^{-1}}}=[1\ \eta^{d-1}\ \eta^{d-2}\ \dots\ \eta^2\ \eta]$ est un vecteur propre à gauche associé à $\eta$.

		\subsection{Substitution d'arbre associée à $\sigma$}
		On définit une substitution d'arbre $\tau$ et un arbre initial $T_0^s$. On va associer à chaque arbre
		$T_n^s = \tau(T_0^s)$, un arbre réel $T_n$ de $\mathscr{R}^d$. La suite $(T_n)_n$ obtenue
		sera convergente vers un arbre réel compact $T_{\tau}$ de $\mathscr{T}^d$ sur lequel on définira un système
		d'isométries partielles qui représentera l'action du décalage sur $\Omega^+$.

			\subsubsection{Substitution d'arbre simplicial}
			On considère l'alphabet $A_\tau=\{1, \dots, d, d+1, \dots, 2d-2\}$ ; $A_\tau$ contient $A$ et $d-2$ lettres supplémentaires
			dont le sens combinatoire sera explicité plus loin (par l'application $p_*$).
			Si $X_i$ est un élément de $\mathscr{S}_E(A_\tau)$, $X_i = (\{x, y\}, \{(x, y, i)\})$, où $i\in A_\tau$,
			$\tau$ est la substitution d'arbre de $\mathscr{S}_0(A_\tau)$ définie par :
				\begin{itemize}
				\item $\tau (X_1) = X_d$,
				\item l'image de $X_2$ est représentée figure \ref{fig:emmasimpsub},
				\item $\tau (X_i) = X_{i-1}$ si $3\le i\le d$,
				\item $\tau (X_{d+1}) = X_1$,
				\item $\tau (X_i) = X_{i-1}$ si $d+2\le i\le 2d-2$.
				\end{itemize}
			\begin{figure}[h!]
			\begin{center}
				\begin{psfrags}
				\psfrag{1}{\LARGE{$1$}}
				\psfrag{d}{\LARGE{$d$}}
				\psfrag{d1}{\LARGE{$(d+1)$}}
				\psfrag{d2}{\LARGE{$(d+2)$}}
				\psfrag{k}{\LARGE{$k$}}
				\psfrag{kb}{\Large{$d+3\le k\le 2d-3$}}
				\psfrag{2d}{\LARGE{$(2d-2)$}}
				\psfrag{x}{\LARGE{$x$}}
				\psfrag{y}{\LARGE{$y$}}
				\psfrag{Y}{\huge{$\tau(X_2)$}}
				\scalebox{0.65}{\includegraphics{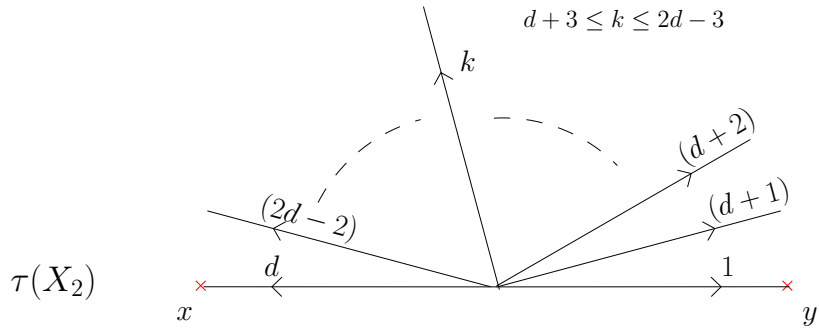}}
				\end{psfrags}
			\end{center}
			\vspace{-4mm}
			\caption{Substitution d'arbre associée à $\sigma$.}
			\label{fig:emmasimpsub}
			\end{figure}
			\begin{rmk}
			On peut définir la matrice d'incidence $M_\tau$ associée à $\tau$ ; $M_\tau$ est une matrice
			$(2d-2)\times (2d-2)$ et $M_\tau(i, j)$ est le nombre d'arêtes colorées par $i$ dans $\tau(X_j)$.
			Le spectre de $M_\tau$ est constitué du spectre de $M_\sigma$ et de $d-2$ valeurs propres de module $1$.
			\end{rmk}
			La matrice tronc $M_t$ est définie par
			\begin{itemize}
				\item $M_t (d, 1) = M_t (d, 2) = 1$,
				\item $M_t (1, d+1) = 1$,
				\item pour tout $2\le i\le 2d-2$ où $i\ne d+1$, $M_t(i-1, i) = 1$,
				\item $M_t$ est nulle partout ailleurs.
			\end{itemize}
			\begin{figure}[h!]
			\begin{center}
				\begin{psfrags}
				\psfrag{0}{\huge{$0$}}
				\psfrag{1}{\huge{$1$}}
				\psfrag{M}{\huge{$M_t \ =$}}
				\scalebox{0.6}{\includegraphics{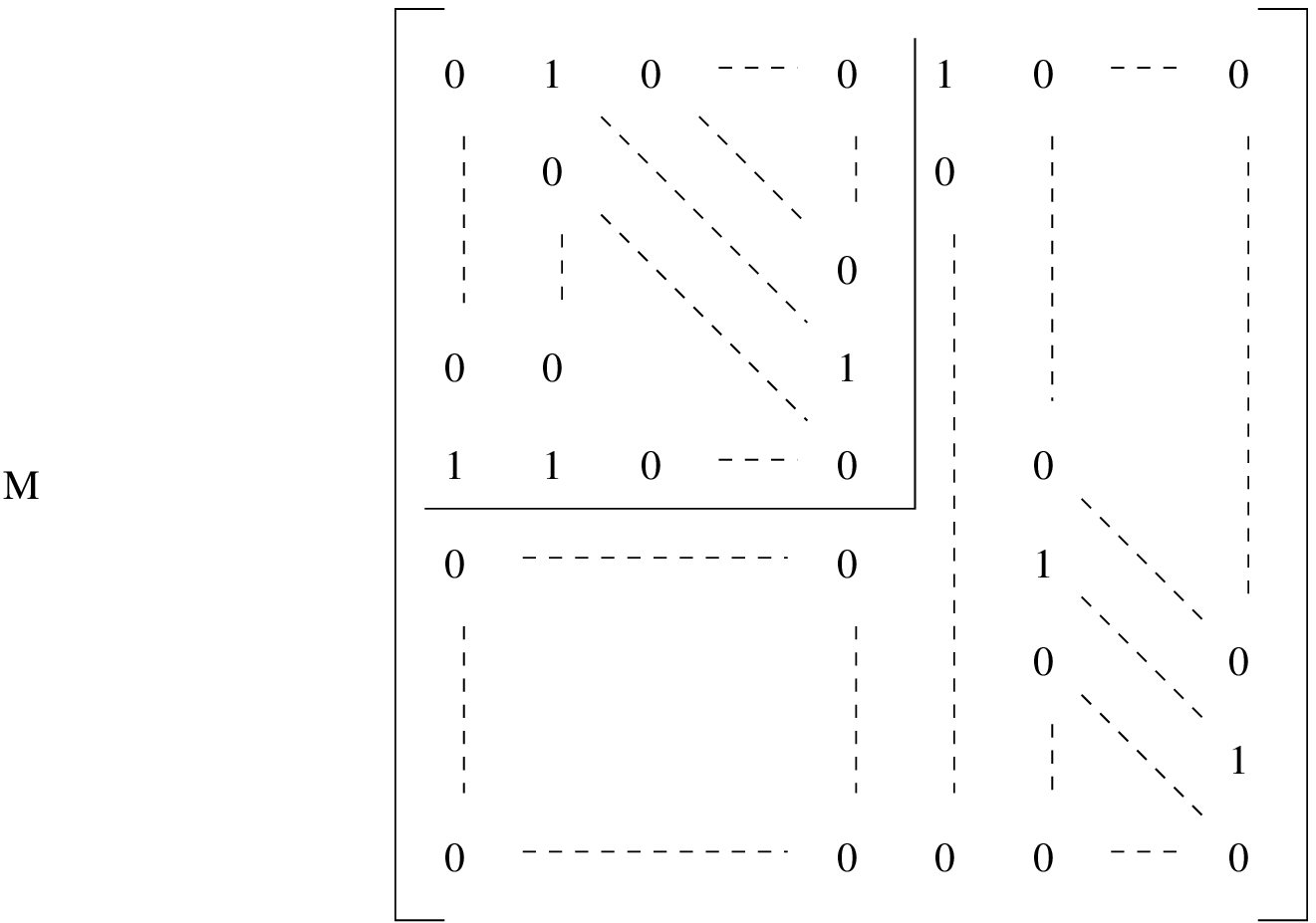}}
				\end{psfrags}
			\end{center}
			\vspace{-4mm}
			\label{fig:troncmat}
			\end{figure}\ \\
			$M_t$ est une matrice $(2d-2)\times (2d-2)$ égale à $M_{\sigma^{-1}}$ sur $\{1\dots d\}\times \{1\dots d\}$.
			Le spectre de $M_t$ contient le spectre de $M_{\sigma^{-1}}$ et $0$ (d'ordre $d-2$). Les deux matrices ont donc
			la même valeur propre dominante $\eta$. Le vecteur
			\begin{center}
				$\mathbf{V_t} = [1\ \ \eta^{d-1}\ \ \eta^{d-2}\ \ \dots\ \ \eta^2\ \ \eta\ \ \eta^{-1}\ \ \eta^{-2}\ \ \dots\ \ \eta^{-(d-3)}\ \ \eta^{-(d-2)}]$
			\end{center}
			est un vecteur propre à gauche associé à $\eta$ ; on se servira de $\mathbf{V_t}$ pour définir les longueurs des arêtes lors de la réalisation.\\

			On choisit un élément $X_2$ de $\mathscr{S}_E(A_\tau)$ d'arête colorée $2$ et on définit $T_0^s = \tau^{d-1}(X_2)$.
			\begin{prop}\label{prop:emmaL0s}
			$T_0^s$ est constitué des $d$ arêtes $(x_0, x_j, j)$ pour $1\le j\le d$ (voir figure \ref{fig:emmaboule1}).
			Par la suite, le sommet $x_0$ sera appelé \textbf{racine} de $T_0^s$.
			\end{prop}
			On définit également $T_n^s = \tau^n(T_0^s)$. On dira encore que $x_0$ est la racine de $T_n^s$ (quel que soit $n\in \mathds{N}$).
			\begin{rmk}
			L'exposant $s$ est là pour signifier que $T_n^s$ est un arbre simplicial.
			\end{rmk}
			Pour tout $k,n\in \mathds{N}$, on notera $B_k(T_n^s)$ la \textbf{boule} de rayon $k$ autour de $x_0$ de $T_n^s$, c'est-à-dire le sous-arbre
			de $T_n^s$ constitué des sommets à distance $\le k$ de $x_0$.

			\begin{prop}\label{prop:emmadiscer}
			Pour tout $n\in \mathds{N}$, $T_n^s$ est un arbre discerné (voir section \ref{subsec:arbsimp}).
			\end{prop}
			\begin{proof}
			Il faut vérifier qu'aucun chemin de la forme $\overline{k}k$ ou $k\overline{k}$, $1\le k\le 2d-2$ n'apparaît dans les chemins de $T_n^s$.
			On se reporte à la figure \ref{fig:emmaboule1} qui représente les boules de rayon $1$ (autour de la racine) des arbres $T_n^s$ ($n\le d$)
			et on note que pour tout $k\ge d$, $B_1(T_k^s)$ et $B_1(T_d^s)$ sont égaux à renommage des sommets près. Les chemins de longueur $2$
			représentés sur les arbres de la figure \ref{fig:emmaboule1} sont les seuls possibles, et on en conclut que
			pour tout $n\in \mathds{N}$, $T_n^s$ est un arbre discerné. 
			\end{proof}
			\begin{figure}[h!]
			\begin{center}
				\begin{psfrags}
				\psfrag{1}{\huge{$1$}}
				\psfrag{2}{\huge{$2$}}
				\psfrag{k}{\huge{$k$}}
				\psfrag{d}{\huge{$d$}}
				\psfrag{d-1}{\Large{$(d-1)$}}
				\psfrag{d-h}{\Large{$(d-h)$}}
				\psfrag{d-h1}{\Large{$(d-(h+1))$}}
				\psfrag{d-l}{\Large{$(d-l)$}}
				\psfrag{L0s}{\huge{$T_0^s$}}
				\psfrag{L1s}{\huge{$B_1(T_1^s)$}}
				\psfrag{Lh+1s}{\huge{$B_1(T_{h+1}^s)$}}
				\psfrag{Ld-1s}{\huge{$B_1(T_{d-1}^s)$}}
				\psfrag{Lds}{\huge{$B_1(T_{d}^s)$}}
				\psfrag{kb1}{\large{$(3\le k\le d-1)$}}
				\psfrag{kb2}{\large{$(3\le k\le d-2)$}}
				\psfrag{kbh+1}{\large{$(3\le k\le d-(h+2))$}}
				\psfrag{lbh+1}{\large{$(1\le l\le h-1)$}}
				\psfrag{kbd-1}{\large{$(3\le k\le d-1)$}}
				\psfrag{kbd}{\large{$(2\le k\le d-2)$}}
				\scalebox{0.6}{\includegraphics{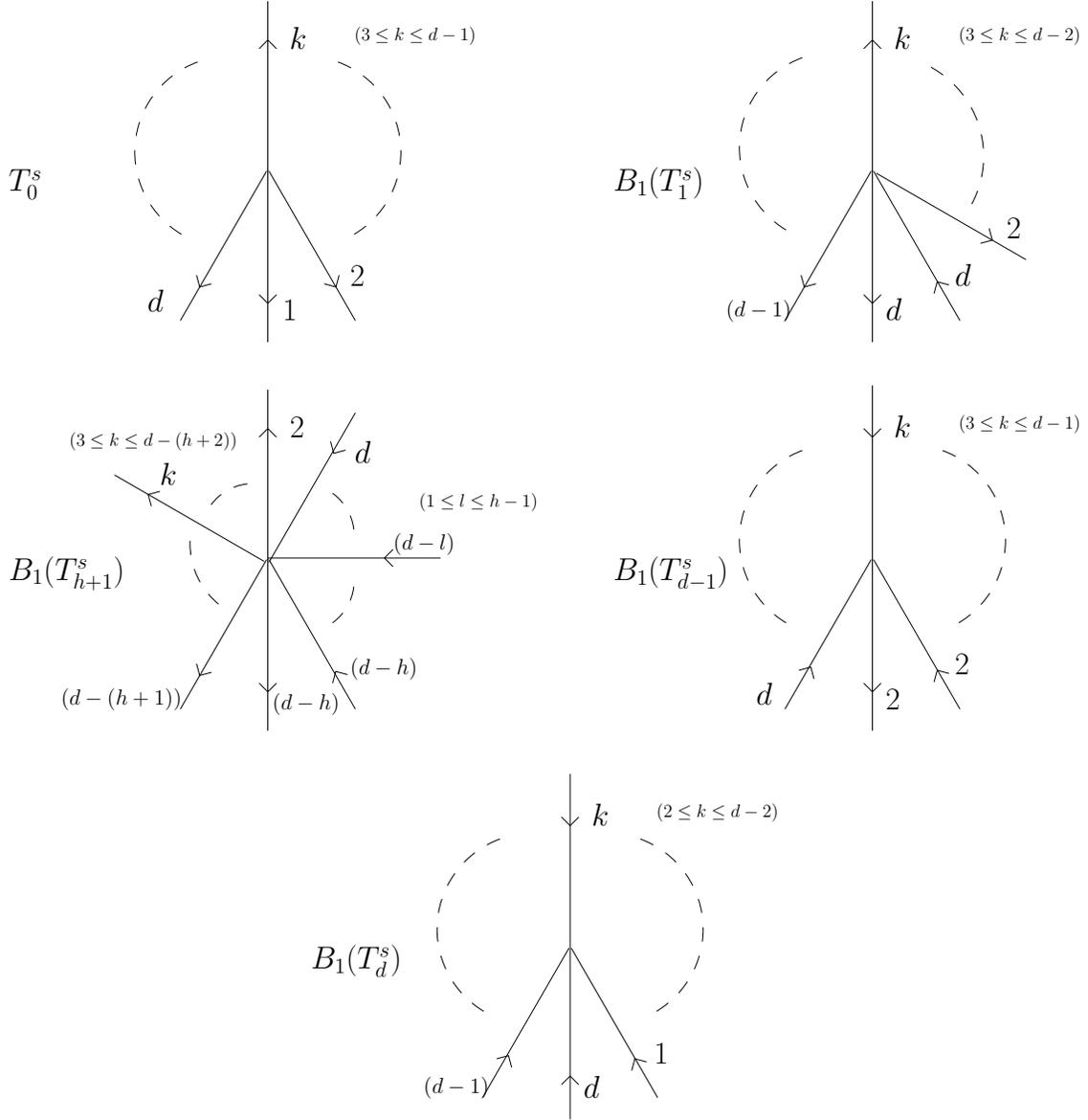}}
				\end{psfrags}
			\end{center}
			\vspace{-4mm}
			\caption{Application de $\tau$ sur une boule de rayon $1$.}
			\label{fig:emmaboule1}
			\end{figure}

			Par la suite, on voudra voir les chemins des arbres simpliciaux comme des éléments de $F_d$ (au lieu
			de $(A_\tau\cup \overline{A_\tau})^*$). Le morphisme $p_* : (A_\tau\cup \overline{A_\tau})^*\to F_d$
			permet de passer de la structure de monoïde à celle de groupe, et interprète combinatoirement
			les lettres $d+1,\dots, 2d-2$ en fonction des lettres de $A$ ; il est défini par :
			\begin{itemize}
				\item $p_*(k) = k$ pour $k\in \{1, \dots, d\}$,
				\item $p_*((d+k)) = \sigma^k(1)$ pour $k\in \{1, \dots, d-2\}$,
				\item $p_*(\overline{k}) = p_*(k)^{-1}$ pour $k\in \{1, \dots, 2d-2\}$.
			\end{itemize}
			\begin{prop}\label{prop:troncagrees}
			Pour tout $1\le i\le d$, si $X_i = (\{x, y\}, \{(x, y, i)\})$ et $\gamma_{\tau(X_i)}$ est la fonction chemin (cf. section \ref{subsec:arbsimp}) de $\tau(X_i)$, alors on a
			\begin{center}
				$p_* (\gamma_{\tau(X_i)} (x, y)) = \sigma^{-1}(i)$.
			\end{center}
			Le tronc de $\tau(X_i)$ (pour $1\le i\le d$) est complètement déterminé par $\sigma^{-1}$.
			\end{prop}
			On remarque que pour tout $n\in \mathds{N}$, les sommets de $T_n^s$ sont de degré soit $1$, soit $d$.
			\begin{defn}
			Pour tout entier $n$, on appelle
			\begin{itemize}
				\item $\mathcal{V}_n^s$ l'ensemble des sommets de $T_n^s$,
				\item $\mathcal{W}_n^s$ l'ensemble des points de branchement de $T_n^s$ (les sommets de degré $d$),
				\item $\gamma_n:\mathcal{V}_n^s\times \mathcal{V}_n^s\to (A_\tau\cup \overline{A_\tau})^*$ la fonction chemin de $T_n^s$
				(cf. section \ref{subsec:arbsimp}).
			\end{itemize}
			\end{defn}
			Si $x_0$ est la racine de $T_n^s$, pour tout élément $x$ de $\mathcal{V}_n^s$, il existe un chemin minimal de $x_0$ à $x$ et un mot $\gamma_n(x_0, x)$ de
			$(A_\tau\cup \overline{A_\tau})^*$ associé.
			\begin{prop}\label{prop:emmaprinj}
			Si $x$ et $y$ sont deux sommets distincts de $T_n^s$, alors $p_*(\gamma_n(x_0, x))\ne p_*(\gamma_n(x_0, y))$.
			\end{prop}
			\begin{proof}
			L'arbre $T_n^s$ est discerné ; on en déduit directement que $\gamma_n(x_0, x)\ne \gamma_n(x_0, y)$. Le résultat est alors immédiat si aucun des deux mots
			ne contient de lettres de $\{d+1, \dots, 2d-2\}$ (en se rappelant que puisque $T_n^s$ est discerné,
			il n'y aura pas d'annulations). Si $\gamma_n(x_0, x)$ contient la lettre $(d+k)\in \{d+1, \dots, 2d-2\}$ et que
			$p_* (\gamma_n(x_0, x)) = p_* (\gamma_n(x_0, y))$, alors $\gamma_n(x_0, y)$ contient
			soit la paire $12$ (dans ce cas $T_{n+1}^s$ n'est pas discerné), soit une paire $(d+h)(h+2)$ pour
			un certain $h<k$ (dans ce cas $T_{n+h+1}^s$ n'est pas discerné) ; les deux cas
			contredisent la proposition \ref{prop:emmadiscer}. 
			\end{proof}

			\subsubsection{Réalisation}\label{subsubsec:real}
			On se permettra de confondre un point de $\mathscr{R}^d$ avec ses écritures (réduites ou non). Le sommet
			$x_0$ est toujours la racine de $T_0^s$ ; pour tout $1\le j\le d$, il existe un sommet $x_j$ tel que $(x_0, x_j, j)$ est
			une arête de $T_0^s$. On note $\nu_0$ l'application définie par
			\begin{center}
				\begin{tabular}{cccclc}
				$\nu_0$ & : & $\mathcal{V}_0^s$ & $\to$ & $\mathscr{R}^d$ & \\
				&& $x_0$ & $\mapsto$ & $O$ & \\
				&& $x_1$ & $\mapsto$ & $0^1$ & \\
				&& $x_j$ & $\mapsto$ & $(j-1)^{\eta^{d-j+1}}$ & $(2\le j\le d)$.
				\end{tabular}
			\end{center}

			On va construire par récurrence une suite $(\nu_n)_n$ d'applications avec, pour tout $n\in \mathds{N}$, $\nu_n : \mathcal{V}_n^s \to \mathscr{R}^d$.
			On suppose $\nu_{n-1}$ construite de telle sorte que si $(y_1, y_2, k)$ est une arête de $T_{n-1}^s$, alors 
			$\nu_{n-1}(y_1)^{-1}\nu_{n-1}(y_2) = j^p$ pour un certain $0\le j\le d-1$ et $|p| = \eta^{-(n-1)}\mathbf{V_t}(k)$.
			Tout sommet $x$ de $T_{n-1}^s$
			est également un sommet de $T_n^s$ par définition de la substitution d'arbre. Pour ces points, on définit
			\begin{center}
			$\nu_n(x) = \nu_{n-1}(x)$.
			\end{center}
			Si $y$ est un point de branchement de
			$\mathcal{V}_n^s\setminus \mathcal{V}_{n-1}^s$, alors il existe deux sommets $y_1$ et $y_2$ de $\mathcal{V}_{n-1}^s$ tels que $(y, y_1, 1)$ et $(y, y_2, d)$
			sont des arêtes de $T_n^s$ et $\nu_{n-1}(y_1)^{-1}\nu_{n-1}(y_2) = j^p$ pour un certain $0\le j\le d-1$ et $|p| = \eta^{-(n-1)+(d-1)}$.
			Pour $\alpha = \frac{p}{|p|}$, on définit alors
			\begin{center}
			$\nu_n(y) = \nu_{n-1}(y_1)j^{\alpha\eta^{-n}}$.
			\end{center}
			Pour ce même $y$, il existe $d-2$ sommets $z_1, \dots, z_{d-2}$ de
			$\mathcal{V}_n^s\setminus \mathcal{V}_{n-1}^s$ tels que pour tout $1\le h\le d-2$, $(y, z_h, (d+h))$ est une arête de $T_n^s$. On définit alors
			\begin{center}
			$\nu_n(z_h) = \nu_{n-1}(y_1)j^{\alpha\eta^{-n}}k^{\eta^{-n-h}}$
			\end{center}
			où $k = j+h [d]$.

			\begin{prop}\label{prop:emmaletlen}
			Pour tout $n\in \mathds{N}$, si $(x_i, x_j, k)$ est une arête de $T_n^s$, alors la longueur
			du segment $[\nu_n(x_i), \nu_n(x_j)]$ est $\mathbf{V_t}(k)\eta^{-n}$.
			\end{prop}

			Pour tout $n\in \mathds{N}$, on définit l'arbre $T_n$ de $\mathscr{T}^d$ comme l'enveloppe convexe des points de $\nu_n(\mathcal{V}_n^s)$.
			De manière évidente, $T_{n-1}\subset T_n$ quel que soit $n\in \mathds{N}^*$. De plus, tout point de $T_n$ est à une distance inférieure
			ou égale à $\eta^{-1-n}$ de $T_{n-1}$, ce qui fait de la suite $(T_n)_n$ une suite de Cauchy de $\mathscr{T}^d$ qui est complet.

			La suite de cet article est consacrée à une étude détaillée des propriétés géométriques et dynamiques de l'arbre limite $T_{\tau}$ défini par
			\begin{eqnarray}\label{eq:Ttau}
				T_{\tau} & = & \lim\limits_{n\to +\infty} T_n.
			\end{eqnarray}

			On notera $d_{T_{\tau}}$ la distance sur $T_{\tau}$.
			On remarque que pour tout $n\in \mathds{N}$, les points de $T_n$ sont de degré $1$, $2$, ou $d$, et qu'il en est de même pour les points de $T_{\tau}$.
			\begin{defn}
			On notera $\mathcal{V}_n = \nu_n(\mathcal{V}_n^s)$ l'union des points de branchement (points de degré $d$)
			et des points terminaux (points de degré $1$) de $T_n$ et $\mathcal{W}_n$ l'ensemble de ses points de branchement.
			\end{defn}
			On note que pour tout $n\in \mathds{N}$, $\nu_n$ est une bijection de $\mathcal{V}_n^s$ dans $\mathcal{V}_n$, et une bijection
			de $\mathcal{W}_n^s$ dans $\mathcal{W}_n$.

		\subsection{Etude combinatoire}
		Dans la section \ref{subsec:f_Q}, on associera un point de l'arbre limite $T_{\tau}$ (\ref{eq:Ttau}) à tout mot
		du système symbolique $\Omega^+$. On se servira notamment de la substitution d'arbre et des facteurs bispéciaux du langage
		afin d'établir une bijection entre les points de branchement de $T_{\tau}$ et les décalés du point fixe $\omega$.
		Ces résultats nécessitent une étude combinatoire détaillée du système symbolique.

		L'automate des préfixes-suffixes de $\sigma$ est décrit en figure \ref{fig:emmautom}.
		\begin{figure}[h!]
		\begin{center}
			\begin{psfrags}
			\psfrag{1}{\LARGE{$1$}}
			\psfrag{2}{\LARGE{$2$}}
			\psfrag{3}{\LARGE{$3$}}
			\psfrag{d}{\LARGE{$d$}}
			\psfrag{k}{\LARGE{$k$}}
			\psfrag{k-1}{\LARGE{$k-1$}}
			\psfrag{e12}{\Large{$(\epsilon, 1, 2)$}}
			\psfrag{12e}{\Large{$(1, 2, \epsilon)$}}
			\psfrag{e1e}{\Large{$(\epsilon, 1, \epsilon)$}}
			\psfrag{e3e}{\Large{$(\epsilon, 3, \epsilon)$}}
			\psfrag{eke}{\Large{$(\epsilon, k, \epsilon)$}}
			\scalebox{0.6}{\includegraphics{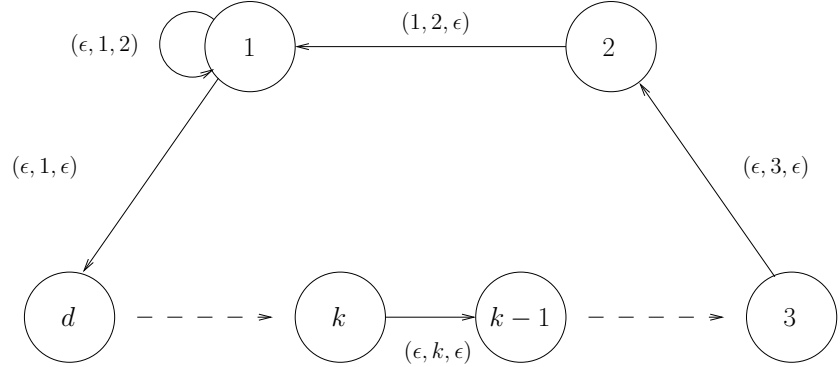}}
			\end{psfrags}
		\end{center}
		\vspace{-4mm}
		\caption{Automate des préfixes-suffixes de $\sigma$.}
		\label{fig:emmautom}
		\end{figure}
		On se reporte à la section \ref{sec:adps} concernant l'automate des préfixes-suffixes pour les notations. On note
		\begin{center}
			$\Omega_{per} = \{\lim\limits_{n\to +\infty} \sigma^{dn}(k.1), 1\le k\le d\}$
		\end{center}
		l'ensemble des mots de $\Omega$ périodiques par $\sigma$ et on définit par
		\begin{center}
			$\omega = \lim\limits_{n\to +\infty}\sigma^n(1)$
		\end{center}
		le point de $\Omega^+$ fixe par $\sigma$.

		Il est évident que si un mot est préfixe de $\omega$, alors il est spécial à gauche. On met ici en évidence les mots bispéciaux du langage,
		et on en déduit que les seuls mots spéciaux à gauche sont les préfixes de $\omega$.

		\begin{lem}
		Un mot $u$ ($|u| > 1$) de $\mathfrak{L}(\Omega)$ est bispécial si et seulement si il existe un mot $v$ de $\mathfrak{L}(\Omega)$ bispécial
		tel que $u = \sigma(v)p$ avec $p = 1$ si la dernière lettre de $v$ est $(d-1)$ et $p=\epsilon$ sinon.
		\end{lem}
		\begin{proof}
		On vérifie facilement que les mots $k1$ pour tout $1\le k\le d$ et $k(k+1)$ pour tout $1\le k\le d-1$ sont les seuls mots
		de longueur $2$ du langage. On note que toute lettre $k\ne d$ est prolongeable à droite par $1$ et $(k+1)$.

		Soit $v$ un facteur bispécial. On suppose que la dernière lettre de $v$ est $k\ne d-1$.
		\begin{figure}[h!]
		\begin{center}
		\begin{psfrags}
			\psfrag{a}{\LARGE{$\alpha$}}
			\psfrag{b}{\LARGE{$\beta$}}
			\psfrag{as}{\LARGE{$\sigma(\alpha)$}}
			\psfrag{bs}{\LARGE{$\sigma(\beta)$}}
			\psfrag{s}{\LARGE{$\sigma$}}
			\psfrag{v}{\LARGE{$v$}}
			\psfrag{vs}{\LARGE{$\sigma(v)$}}
			\psfrag{1}{\LARGE{$1$}}
			\psfrag{1s}{\LARGE{$12$}}
			\psfrag{k1}{\LARGE{$(k+1)$}}
			\psfrag{k1s}{\LARGE{$(k+2)$}}
			\scalebox{0.65}{\includegraphics{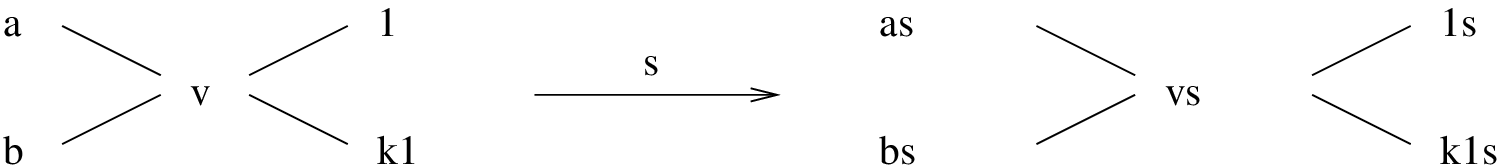}}
			\end{psfrags}
		\end{center}
		\vspace{-4mm}
		\end{figure}\\
		Les dernières lettres de $\sigma(\alpha)$ et $\sigma(\beta)$ sont distinctes si $\alpha\ne \beta$ et $k+2\ne 1$ ;
		$\sigma(v)$ est donc bispécial.

		On suppose maintenant que la dernière lettre de $v$ est $(d-1)$, et on remarque que $d1$ est le seul mot de
		longueur $2$ de préfixe $d$.
		\begin{figure}[h!]
		\begin{center}
		\begin{psfrags}
			\psfrag{a}{\LARGE{$\alpha$}}
			\psfrag{b}{\LARGE{$\beta$}}
			\psfrag{as}{\LARGE{$\sigma(\alpha)$}}
			\psfrag{bs}{\LARGE{$\sigma(\beta)$}}
			\psfrag{s}{\LARGE{$\sigma$}}
			\psfrag{v}{\LARGE{$v$}}
			\psfrag{vs}{\LARGE{$\sigma(v)$}}
			\psfrag{1}{\LARGE{$1$}}
			\psfrag{1s}{\LARGE{$12$}}
			\psfrag{k1}{\LARGE{$d1$}}
			\psfrag{k1s}{\LARGE{$112$}}
			\scalebox{0.65}{\includegraphics{eps/bisp1.eps}}
			\end{psfrags}
		\end{center}
		\vspace{-4mm}
		\end{figure}\\
		Il vient que $\sigma(v)1$ est un facteur bispécial.\\

		Soit $u$ un mot bispécial ($|u| > 1$) ; $u$ commence forcément par $1$ et on suppose que $k$ est la dernière lettre de $u$.
		En remarquant que $k\ne d$ et que $d11$
		est l'unique mot de longueur $3$ de suffixe $11$, on déduit que $u$ peut être prolongé à droite
		par $12$ et $(k+1)$. On suppose que $u$ peut être prolongé à gauche par $\alpha$ et $\beta$.
		Puisque $2$ n'est pas spécial à droite, on pourra supposer que $\alpha, \beta\in \{1, 12, 3,\dots d\}$.

		Si $k\ne 1$, il existe un mot $v$ de $\mathfrak{L}(\Omega)$ tel que $\sigma(v) = u$
		et $v$ peut être prolongé à droite par $1$, $k$ et à gauche par $\alpha^{\prime}, \beta^{\prime}$ avec
		$\sigma(\alpha^{\prime}) = \alpha$ et $\sigma(\beta^{\prime}) = \beta$.

		Si $k=1$, on note $u = u_p1$. $|u_p| > 1$ puisque $11$ n'est pas bispécial. $u_p$ peut être prolongé à droite par $112$ et $12$ et se termine forcément par $d$ (puisqu'on peut le prolonger par $11$).
		Il existe donc $v\in \mathfrak{L}(\Omega)$ tel que $\sigma(v) = u_p$ ($v$ se termine par $(d-1)$) et $v$ peut être
		prolongé à droite par $1$ et $d1$, et à gauche par $\alpha^{\prime}, \beta^{\prime}$ avec
		$\sigma(\alpha^{\prime}) = \alpha$ et $\sigma(\beta^{\prime}) = \beta$. 
		\end{proof}

		\begin{prop}\label{prop:emmabisp}
		Un mot $u\ne \epsilon$ de $\mathfrak{L}(\Omega)$ est bispécial si et seulement si $u = \sigma^{\alpha_n}(1)\dots \sigma^{\alpha_0}(1)$
		avec
		\begin{itemize}
			\item $\alpha_0 < d$,
			\item $\alpha_{k+1}-\alpha_k = d-1$ quel que soit $0\le k\le n-1$.
		\end{itemize}
		\end{prop}
		\begin{proof}
		On remarque que $1$ est le seul mot de longueur $1$ bispécial, et que $\sigma^k(1)$ se termine par la lettre $(k+1)$
		pour tout $k<d$. On conclut grâce au lemme précédent. 
		\end{proof}
		On remarque que par construction, tous les mots bispéciaux sont des préfixes de $\omega$, et ils sont tous prolongeables
		à gauche par toute lettre de $A$. De plus, tout mot spécial à gauche du langage est facteur de $\sigma^k(1)$ pour tout $k$ supérieur
		à un certain $k_0$ (par primitivité et puisque $\sigma^{n}(1)$ est un préfixe de $\sigma^{n+1}(1)$ pour tout $n\in \mathds{N}$), et on peut donc le prolonger à droite
		en un mot bispécial ; on en déduit que les préfixes de $\omega$ sont les seuls mots spéciaux à gauche de $\mathfrak{L}(\Omega)$.
		Le nombre de mots de longueur $n\ge 1$ dans $\mathfrak{L}(\Omega)$ est donc $(d-1)n+1$ et on a la proposition suivante.
		\begin{prop}\label{prop:uniqinfspegau}
		$\omega$ est l'unique mot de $\Omega^+$ spécial à gauche.
		\end{prop}

		\subsection{Application du système symbolique $\Omega^+$ dans l'arbre limite $T_{\tau}$}\label{subsec:f_Q}
		On va se servir d'une étude combinatoire de la substitution d'arbre pour construire une application $f_Q$ de $\Omega^+$ dans l'arbre $T_{\tau}$ (\ref{eq:Ttau});
		celle-ci nous permettra de définir un système d'isométries partielles qui représente sur $T_{\tau}$ l'action du décalage sur $\Omega^+$.

			\subsubsection{Sur les décalés du point fixe}
			On rappelle que $\omega = \lim\limits_{n\to +\infty} \sigma^n(1)$ et on note $\Omega^+_p = \{S^n(\omega), n\in \mathds{N}\}$.
			Dans un premier temps, on définit $f_Q$ de $\Omega^+_p$ dans $T_{\tau}$.
			Si $V = S^k(\omega)$ ($k\in \mathds{N}$) est un élément de $\Omega^+_p$, alors $V = u^{-1}\omega$ où $u$ est le préfixe
			de $\omega$ de longueur $k$. On va établir une bijection $f_0$ entre l'ensemble $\mathcal{W}_n^s$ des points de branchement (points de degré $d$) des arbres $T_n^s$
			et l'ensemble $\delta(\Omega^+_p) = \{u^{-1}\ ;$ $u$ \textit{est un préfixe de} $\omega\}$ ; si $x$ est un sommet de $T_k^s$ dont l'image est
			$u^{-1}\in \delta(\Omega^+_p)$, alors on définira $f_Q(u^{-1}\omega) = \nu_k(x)$ (voir le paragraphe \ref{subsubsec:real} pour la définition de l'application $\nu_k$).\\

			Soit $f_0$ l'application définie par
			\begin{equation}\label{eq:f_0}
				\begin{tabular}{ccccl}
				$f_0$ & : & $\bigcup\limits_{n\in \mathds{N}} \mathcal{W}_n^s$ & $\to$ & $F_d$\\
				&& $x$ & $\mapsto$ & $\sigma^k(p_*(\gamma_k(x_0, x)))$
				\end{tabular}
			\end{equation}
			où $k$ est n'importe quel entier tel que $x$ est un sommet de $T_k^s$, et $\gamma_k$ est la fonction chemin de $T_k^s$ (voir section \ref{subsec:arbsimp}).
			Par définition de la substitution d'arbre, l'égalité
			$\sigma(p_*(\gamma_{k+1}(x_0, x))) = p_*(\gamma_k(x_0, x))$ est vérifiée et $f_0(x)$ ne dépend pas du choix de $k$.

			\begin{defn}
			On dit qu'un mot $u^{-1}$ \textbf{apparaît} à l'étape $k$ s'il existe un point de branchement $x$ de $T_k^s$ tel que
			$u^{-1} = \sigma^k(p_*(\gamma_k(x_0, x)))$ et si $x$ n'est pas un sommet de $T_{k-1}^s$.
			\end{defn}
			On dira par convention que le mot $\epsilon$ est apparu à l'étape $-(d-2)$ et qu'aucun autre mot
			n'est apparu aux étapes $\le 0$. On se reporte à la figure
			\ref{fig:emmaboule1} et on rappelle qu'on a défini $T_0^s = \tau^{d-1}(X_2)$. Si un mot $u^{-1}$
			apparaît à une étape $n$, alors pour tout $d-1\le k\le 2d-2$, le mot $u^{-1}\sigma^{n+k}(d^{-1})$
			apparaît à l'étape $n+k$. De plus, pour tout $k\ge 2d-2$, le mot $u^{-1}\sigma^{n+k}(1^{-1})$ apparaît en $n+k$.
			On remarque cependant que si $u^{-1}\sigma^{n+k}(1^{-1})$ apparaît à l'étape $n+k$ ($k\ge 2d-2$), alors le mot
			$v^{-1} = u^{-1}\sigma^{n+k}(1^{-1})\sigma^{n+k}(d)$ est un mot apparu à une étape $<n+k$.
			On en déduit la proposition suivante.
			\begin{prop}\label{prop:emmalite}
			Tout mot $u^{-1}$ apparu à une étape $n$, vérifie
			$u^{-1}=v^{-1}\sigma^{n}(d^{-1})$, où $v^{-1}$ est apparu à une étape
			$m$, avec $n-(2d-2)\le m\le n-(d-1)$.
			\end{prop}

			La proposition précédente va nous permettre d'énumérer les mots qui apparaissent à une étape donnée.
			On va voir que l'apparition de ces mots est déterminée par les facteurs bispéciaux du langage. A cet effet,
			on définit $l_0=\epsilon$ et $l_m = \sigma^{m-1}(1^{-1})$ si $1\le m\le d-1$, et pour tout $m\ge d$,
			on définit $l_m = \sigma^{\alpha_0}(1^{-1})\dots \sigma^{\alpha_{n}}(1^{-1})$ ($n\ge 0$) avec
			\begin{itemize}
				\item $\alpha_0 = m-1 [d-1]$,
				\item pour tout $1\le i\le n$, $\alpha_i = \alpha_{i-1}+(d-1)$,
				\item $\alpha_n = m-1$.
			\end{itemize}
			\begin{prop}\label{prop:emmabispeqpldev}\ 
			\begin{itemize}
				\item Soit $u^{-1}\ne \epsilon$ un mot de $F_d$. Il existe $m\in \mathds{N}^*$ tel que $u^{-1} = l_m$
				si et seulement si $u$ est un mot bispécial de $\mathfrak{L}(\Omega)$ (cf. proposition \ref{prop:emmabisp}).
				\item Tout mot bispécial étant un préfixe de $\omega$, $l_k$ est un suffixe (propre) de $l_{k+1}$ pour tout $k\in \mathds{N}^*$.
			\end{itemize}
			\end{prop}
			La proposition suivante permettra de conclure à la bijectivité de $f_0$. On note que le raisonnement donné est tout aussi
			important que le résultat lui-même. On montre que les mots apparaissent par ordre croissant de longueur (partie $(4)$)
			et que le mot le plus long à apparaître à une étape donnée est l'inverse d'un bispécial (parties $(1)$ et $(2)$).
			Cette propriété a été constatée sur d'autres classes d'exemples, et pourrait servir de base à une théorie plus générale.
			\begin{prop}\label{prop:emmaleet}
			\ 
			\begin{itemize}
				\item $(1)$ Pour tout $1\le m\le d-1$, le mot $l_m$ est le seul mot à apparaître à l'étape $m$.
				\item $(2)$ Pour tout $m\ge d$, $l_m$ est le plus long mot apparaissant en $m$.
				\item $(3)$ L'image par $f_0$ de tout point de branchement de $T_m^s$ est un suffixe de $l_m$.
				\item $(4)$ Tout suffixe de $l_m$ est apparu à une étape $\le m$.
			\end{itemize}
			\end{prop}
			\begin{proof}
			La partie $(1)$ peut se lire directement sur la figure \ref{fig:emmaboule1}. On va démontrer les parties
			$(2)$, $(3)$ et $(4)$ par récurrence. On suppose les propriétés vraies aux rangs $< m$ et on suppose $m\ge d$.

			$(2)$ On déduit de la proposition \ref{prop:emmalite} (et de l'hypothèse de récurrence) que le mot le plus long apparaissant
			à l'étape $m$ est le mot $l_{m-(d-1)}\sigma^{m}(d^{-1}) = l_{m-(d-1)}\sigma^{m-1}(1^{-1})$ où $l_{m-(d-1)}$ est le plus long mot apparaissant en $(m-(d-1))$.
			L'hypothèse de récurrence permet immédiatement de conclure.

			$(3)$ Par hypothèse de récurrence, tout mot $u^{-1}$ apparu à une étape $\le (m-(d-1))$ est un suffixe de $l_{m-(d-1)}$ ;
			le mot $u^{-1}\sigma^{m-1}(1^{-1})$ (cf. proposition \ref{prop:emmalite}) est alors un suffixe de $l_{m-(d-1)}\sigma^{m-1}(1^{-1}) = l_m$.

			$(4)$ On note $c_k$ le mot le plus court à apparaître à l'étape $k$. Il nous suffit de prouver que $|c_m| = |l_{m-1}|+1$.
			Si $(m=d)$, alors on a\vspace*{1mm}
			\begin{itemize}
				\item $c_m = \sigma^{d}(d^{-1}) = \sigma^{d-1}(1^{-1})$, et\vspace*{1mm}
				\item $l_{m-1}=\sigma^{d-2}(1^{-1})$ d'après $(1)$.\vspace*{1mm}
			\end{itemize}
			Si $(d+1\le m\le 2d-1)$, alors on a\vspace*{1mm}
			\begin{itemize}
				\item $c_m = 1^{-1}\sigma^m(d^{-1}) = 1^{-1}\sigma^{m-1}(1^{-1})$, et\vspace*{1mm}
				\item $l_{m-1} = \sigma^{\alpha_0}(1^{-1})\sigma^{\alpha_1}(1^{-1})$ avec
				$\alpha_0 = (m-2)-(d-1)$ et $\alpha_1 = (m-2)$.\vspace*{1mm}
			\end{itemize}
			En remarquant que $\sigma^d(1^{-1}) = 1^{-1}\sigma^{d-1}(1^{-1})$, on obtient $l_{m-1} = \sigma^{m-1}(1^{-1})$.
			Dans les deux cas, on a $|c_m| = |l_{m-1}|+1$.

			On se place finalement dans le cas où $m>2d-1$. On déduit de la proposition \ref{prop:emmalite} que\vspace*{1mm}
			\begin{itemize}
				\item $c_m = c_{m-(2d-2)}\sigma^{m-1}(1^{-1})$, et\vspace*{1mm}
				\item $l_{m-1} = l_{(m-1)-(2d-2)}\sigma^{\alpha_{p-1}}(1^{-1})\sigma^{\alpha_p}(1^{-1})$ avec $\alpha_{p-1} = (m-2)-(d-1)$ et $\alpha_p = (m-2)$,\vspace*{1mm}
			\end{itemize}
			et on obtient ainsi $l_{m-1} = l_{(m-1)-(2d-2)}\sigma^{m-1}(1^{-1})$.\\
			L'égalité $|c_{m-(2d-2)}| = |l_{(m-1)-(2d-2)}|+1$ est vérifiée par hypothèse de récurrence et on conclut que $|c_m| = |l_{m-1}|+1$. 
			\end{proof}

			On déduit de cette proposition que l'application $f_0$ est surjective de $\bigcup\limits_{n\in \mathds{N}} \mathcal{W}_n^s$
			dans $\delta(\Omega^+_p)$. L'injection est assurée par la proposition \ref{prop:emmaprinj}
			et par le fait que $\sigma$ est un automorphisme.
			\begin{prop}
			L'application $f_0$ (\ref{eq:f_0}) est une bijection de l'ensemble $\bigcup\limits_{n\in \mathds{N}} \mathcal{W}_n^s$ des points de branchement
			des arbres $T_n^s$ dans l'ensemble $\delta(\Omega^+_p)$ des mots $u^{-1}$ de $F_d$ tels que $u$ est un préfixe (possiblement vide) de $\omega$.
			\end{prop}

			On définit finalement l'application $f_Q$ (pour l'instant restreinte à l'ensemble $\Omega^+_p$ des mots de l'orbite positive du point fixe $\omega$) par :
			\begin{equation}\label{eq:f_Q1}
				\begin{tabular}{ccccl}
				$f_Q$ & : & $\Omega^+_p$ & $\to$ & $T_{\tau}$\\
				&& $u^{-1}\omega$ & $\mapsto$ & $\nu_k(f_0^{-1} (u^{-1}))$
				\end{tabular}
			\end{equation}
			où $k$ est n'importe quel entier pour lequel $f_0^{-1} (u^{-1})$ est un sommet de $T_k^s$ et $\nu_k$ est l'application
			définie au paragraphe \ref{subsubsec:real}.

			\begin{rmk}
			On peut également définir l'application $f_0$ sur l'ensemble des points terminaux
			(en donnant la même définition que pour les points de branchement).
			Il est possible de montrer que $f_0$ est aussi une bijection
			de l'ensemble des points terminaux dans l'ensemble des mots $u$ de $\mathfrak{L}(\Omega)$
			tels que $u\omega\in \Omega^+$. On aurait ainsi pu définir directement l'image par
			$f_Q$ des mots $V$ de $\Omega^+$ tels que $S^k(V)=\omega$ pour un certain $k\in \mathds{N}^*$.
			La définition alternative donnée au paragraphe suivant est équivalente.
			\end{rmk}

			\subsubsection{Extension du domaine de définition de l'application $f_Q$}
			L'application $f_Q$ (\ref{eq:f_Q1}) associe un point de l'arbre $T_{\tau}$ (\ref{eq:Ttau}) à tout mot de l'ensemble $\Omega^+_p = \{S^n(\omega) ; n\in \mathds{N}\}$
			(où $\omega$ est le mot de $A^\mathds{N}$ fixe par $\sigma$). Le but de ce paragraphe est d'associer un point de l'arbre $T_{\tau}$
			à tout mot du système symbolique unilatère $\Omega^+$, et de montrer que la nouvelle application $f_Q$ ainsi obtenue
			est surjective de $\Omega^+$ dans $T_{\tau}$. Pour cela, on utilisera l'automate des préfixes-suffixes
			afin de produire des suites de mots dont les images par $f_Q$ sont
			convergentes dans $T_{\tau}$. On commence par un travail préliminaire sur certaines propriétés
			des développements en préfixes-suffixes et de l'application $f_0$ (\ref{eq:f_0}).\\

			Soit $u$ un préfixe de $\omega$. On appelle \textbf{écriture automatique} de $u^{-1}$ l'écriture
			\begin{itemize}
				\item $u^{-1} = \sigma^{\alpha_0}(1^{-1})\dots \sigma^{\alpha_p}(1^{-1})$, avec
				\item $\forall i, 0\le i < p$, $\alpha_{i+1}-\alpha_i \ge d$.
			\end{itemize}
			La suite $\alpha_0, \dots, \alpha_p$ détermine de manière unique un développement en préfixes-suffixes : celui de
			$S^{|u|}(\overline{\omega}.\omega)$, où $\overline{\omega}$ est n'importe quel mot infini à gauche tel que $\overline{\omega}.\omega\in \Omega$.
			De plus, l'écriture automatique est unique ; en effet, d'après la proposition \ref{prop:coolprefsuff},
			tout mot dont l'écriture automatique a une puissance maximale $\le m$ (quel que soit $m\in \mathds{N}$) est un suffixe propre de $\sigma^{m+1}(1^{-1})$.
			
			\begin{lem}\label{lem:automappar}
			Si un mot $u^{-1}$ apparaît à l'étape $n\ge 1$, alors la puissance maximale de son écriture automatique est soit $n-1$, soit $n$.
			\end{lem}
			\begin{proof}
			Pour tout $m\in \mathds{N}^*$, le mot le plus long à apparaître à l'étape $m$ est
			$l_m = \sigma^{\alpha_0}(1^{-1})\dots \sigma^{\alpha_{p}}(1^{-1})$ ($p\ge 0$) avec
			\begin{itemize}
				\item $\alpha_0 = m-1 [d-1]$,
				\item pour tout $1\le i\le p$, $\alpha_i = \alpha_{i-1}+(d-1)$,
				\item $\alpha_p = m-1$.
			\end{itemize}
			Si $m\le d-1$, le mot $\sigma^{m-1}(1^{-1})$ est le seul à apparaître à l'étape $m$. Si $m\ge d$,
			appliquer l'égalité $\sigma^d(1^{-1})=1^{-1}\sigma^{d-1}(1^{-1})$ (dans l'ordre décroissant des puissances) permet
			d'obtenir une écriture automatique de $l_m$ de puissance maximale $m$. Si $u^{-1}$ apparaît à l'étape $n\ge 1$,
			alors $|l_{n-1}| < |u^{-1}| \le |l_n|$. On conclut en utilisant la proposition \ref{prop:coolprefsuff}.
			\end{proof}

			Si le mot $u^{-1}$ apparaît à l'étape $n\ge 1$, alors il existe deux sommets $x$ et $y$ de $T_n^s$ tels que
			l'arbre décrit ci-dessous est un sous-arbre de $T_n^s$.
			\begin{figure}[h!]
			\begin{center}
				\begin{psfrags}
				\psfrag{1}{\LARGE{$1$}}
				\psfrag{d}{\LARGE{$d$}}
				\psfrag{x}{\LARGE{$x$}}
				\psfrag{y}{\LARGE{$y$}}
				\psfrag{fu}{\LARGE{$f_0^{-1}(u^{-1})$}}
				\scalebox{0.58}{\includegraphics{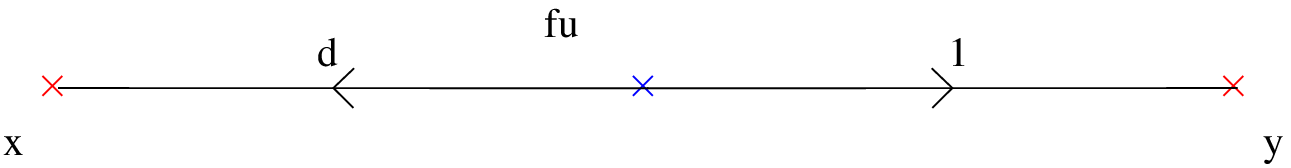}}
				\end{psfrags}
			\end{center}
			\vspace{-4mm}
			\end{figure}
			De plus, $(x, y, 2)$ est une arête de $T_{n-1}^s$. Les points de degré $1$ ne peuvent pas posséder d'arête sortante
			colorée $2$ ; on en déduit que $x$ est un point de branchement. On montre de plus la propriété suivante.
			\begin{lem}\label{lem:helpsurj}
			Soit $u^{-1}$ un mot apparaissant à l'étape $n\ge 1$ et soit $y$ le sommet de $T_n^s$ tel que $(f_0^{-1}(u^{-1}), y, 1)$
			est une arête de $T_n^s$. Si $n$ est la puissance maximale de l'écriture automatique de $u^{-1}$, alors
			$y$ est un point de branchement.
			\end{lem}
			\begin{proof}
			On suppose que l'écriture automatique de $u^{-1}$ est
			\begin{center}
				$u^{-1} = \sigma^{\alpha_0}(1^{-1})\dots \sigma^{\alpha_{p-1}}(1^{-1})\sigma^{n}(1^{-1})$.
			\end{center}
			On a alors $f_0(y) = \sigma^{\alpha_0}(1^{-1})\dots \sigma^{\alpha_{p-1}}(1^{-1})$ (on note que $f_0(y)$ peut être égal à $\epsilon$).
			On sait que $f_0$ est une bijection de $\bigcup\limits_{n\in \mathds{N}}\mathcal{W}_n^s$ dans $\delta(\Omega^+_p) = \{u^{-1}\ ;$ $u$ \textit{est un préfixe de} $\omega\}$.
			Ainsi, si $y$ n'était pas un point de branchement, la proposition \ref{prop:emmaprinj} serait contredite.
			\end{proof}\ 

			On va maintenant définir les images par $f_Q$ des mots de $\Omega^+\setminus \Omega^+_p$.
			Soit $V$ un mot de $\Omega^+\setminus \Omega^+_p$. Puisque $\omega$ est le seul mot de $\Omega^+$ spécial à gauche,
			il existe un unique mot $U$ infini à gauche tel que
			$U.V\in \Omega$. On suppose de plus que $\Gamma(U.V) = (p_i, a_i, s_i)_{i\in \mathds{N}}$.
			On note $u_0^{-1} = p_0^{-1}$ et pour tout $n\in \mathds{N}$, $u_{n+1}^{-1} = u_{n}^{-1}\sigma^{n+1}(p_{n+1}^{-1})$.
			\begin{prop}
			La suite $(f_Q(u_n^{-1}\omega))_n$ est une suite de Cauchy de $T_{\tau}$.
			\end{prop}
			\begin{proof}
			Soit $n\in \mathds{N}$. Le mot $p_{n+1}$ peut être égal à $\epsilon$ ou à $1$. Si $p_{n+1}=\epsilon$, alors
			$u_n = u_{n+1}$. On se place dans le cas où $p_{n+1} = 1$.
			Soit $m$ l'étape d'apparition du mot $u_{n+1}^{-1}$. Les points
			$f_0^{-1}(u_n^{-1})=x$ et $f_0^{-1}(u_{n+1}^{-1})=y$ sont des sommets de $T_m^s$. On déduit du lemme \ref{lem:automappar}
			que $(n+1 = m)$ ou $(n+1 = m-1)$. Si $(n+1 = m)$, alors $\gamma_m(y, x) = 1$, et si $(n+1 = m-1)$, alors $\gamma_m(y, x) = d$
			(on rappelle que $\gamma_m$ est la fonction chemin de $T_m^s$). On se reporte à la proposition \ref{prop:emmaletlen} pour conclure que dans les
			deux cas, on a
			\begin{center}
				$d_{T_{\tau}}(f_Q(u_n^{-1}\omega), f_Q(u_{n+1}^{-1}\omega)) = \eta^{-(n+1)}\ \mathbf{V_t}(1)$,
			\end{center}
			où $\mathbf{V_t}$ est le vecteur propre à gauche (associé à $\eta$) de la matrice tronc défini précédemment.
			\end{proof}
			Finalement, $f_Q$ est l'application de $\Omega^+$ dans $T_{\tau}$ définie par :
			\begin{center}
				$\forall~ V=u^{-1}\omega\in \Omega^+_p,~ f_Q(V) = \nu_k(f_0^{-1} (u^{-1}))$
			\end{center}si $f_0^{-1} (u^{-1})$ est un sommet de $T_k^s$ (cf. \ref{eq:f_Q1}), et
			\begin{equation}\label{eq:f_Q2}
				\forall~ V\in \Omega^+\setminus\Omega^+_p,~ f_Q(V) = \lim\limits_{n\to +\infty} f_Q(u_n^{-1}\omega)
			\end{equation}
			si $U.V$ est un mot de $\Omega$ tel que $\Gamma(U.V) = (p_i, a_i, s_i)_{i\in \mathds{N}}$, et si on a défini
			$u_0^{-1} = p_0^{-1}$ et $u_{n+1}^{-1} = u_{n}^{-1}\sigma^{n+1}(p_{n+1}^{-1})$ pour tout $n\in \mathds{N}$.\vspace{1mm}

			On va montrer que l'application $f_Q:\Omega^+\to T_{\tau}$ ainsi définie est surjective.
			Si $z$ est un point quelconque de $T_{\tau}$, on va choisir une suite $(z_n)_n$ de points de $\bigcup\limits_{n\in \mathds{N}}\mathcal{W}_n$
			(on rappelle que $\mathcal{W}_n$ est l'ensemble des points de $T_n$ de degré $d$)
			convergente vers $z$, et montrer l'existence d'une suite extraite $(z_n^\prime)_n$ telle que, pour tout $n\in \mathds{N}$,\vspace{1mm}
			\begin{itemize}
				\item $f_Q(u_n^{-1}\omega) = z_n^\prime$,\vspace{0.1cm}
				\item si $u_n^{-1} = \sigma^{\alpha_0}(1^{-1})\dots \sigma^{\alpha_{p}}(1^{-1})$ avec $\alpha_{i+1} - \alpha_{i} \ge d$ pour tout $0\le i < p$,\\
				alors $u_{n+1}^{-1} = \sigma^{\alpha_0}(1^{-1})\dots \sigma^{\alpha_{p}}(1^{-1})\sigma^{\alpha_{p+1}}(1^{-1})$ avec $\alpha_{p+1}-\alpha_p\ge d$.\vspace{0.1cm}
			\end{itemize}
			La suite $(u_n^{-1})_n$ nous permettra ainsi de retrouver le développement en préfixes-suffixes d'un mot $U.V\in \Omega$, et on pourra
			en conclure que $z = f_Q(V)$.\\

			Le choix de la suite $(z_n)_n$ est déterminé par l'approximation de $T_{\tau}$ par la suite $(T_n)_n$.
			\begin{defn}\label{defn:arcsimple}
			Soit $e=[s, t]$ un arc de $T_n$. Si $e$ est tel que
				\begin{itemize}
				\item $s$ et $t$ sont des éléments de $\mathcal{V}_n$ (l'ensemble des points de $T_n$ de degré $1$ ou $d$),
				\item l'intérieur de $e$ ne contient aucun élément de $\mathcal{V}_n$,
				\end{itemize}
			alors on dira que $e$ est un arc \textbf{simple}.
			\end{defn}
			Pour tout arc $[s, t]$ simple dans $T_n$ pour un certain $n$, on définit
			\begin{center}
				$T_{\tau}(s, t) = \{x\in T_{\tau} ; s\notin ]x, t[ \text{ et } t\notin ]x, s[\}$.
			\end{center}
			Essentiellement, $T_{\tau}(s, t)$ contient $s, t$ et la partie de $T_{\tau}$ comprise \og entre \fg~ $s$ et $t$.

			\begin{prop}\label{prop:emmasurj}
			$f_Q$ est surjective de $\Omega^+$ dans $T_{\tau}$.
			\end{prop}
			\begin{proof}
			On rappelle que pour tout $n\in \mathds{N}$, l'application $\nu_n$ est une bijection de
			$\mathcal{V}_n^s$ dans $\mathcal{V}_n$ et que sa restriction à $\mathcal{V}_{n-1}^s$ (si $n\ge 1$) est égale à $\nu_{n-1}$.\vspace{0.1cm}

			Soit $z$ un point de $T_{\tau}\setminus (\bigcup\limits_{n\in \mathds{N}} \mathcal{W}_n)$. Pour tout $n\in \mathds{N}$,
			il existe un unique arc simple $e = [x_n, y_n^\prime]$ de $T_n$ tel que $z\in T_{\tau}(x_n, y_n^\prime)$.
			On suppose par convention que $(\nu_n^{-1}(x_n), \nu_n^{-1}(y_n^\prime), j)$ avec $j\in A_\tau$
			est une arête de $T_n^s$. Sous cette condition, le point $x_n$ est forcément un point de branchement
			(les points de degré $1$ de $T_n^s$ ne peuvent pas posséder d'arêtes sortantes).
			Si $y_n^\prime$ est un point de branchement, alors on définit $y_n = y_n^\prime$ ; sinon,
			on prend $y_n = x_n$. On définit, pour tout $n\in \mathds{N}$,
			\begin{center}
				\begin{tabular}{lcc}
					$z_{2n}$ & $=$ & $x_n$,\\
					$z_{2n+1}$ & $=$ & $y_n$.
				\end{tabular}
			\end{center}
			La suite $(z_n)_n$ est une suite de $(\bigcup\limits_{n\in \mathds{N}} \mathcal{W}_n)$ qui converge vers $z$
			et on suppose que pour tout $n\in \mathds{N}$, $f_Q(u_n^{-1}\omega) = z_n$.\\

			On remarque que $z_0 = f_Q(\omega)$.
			Pour pouvoir conclure à la surjectivité de $f_Q$, il suffit de montrer que pour tout
			$k\in \mathds{N}$, si
			\begin{itemize}
				\item $u_k^{-1} = \sigma^{\alpha_0}(1^{-1})\dots \sigma^{\alpha_p}(1^{-1})$ avec $p\ge 1$, et
				\item $\forall i, 0\le i < p$, $\alpha_{i+1}-\alpha_i \ge d$,
			\end{itemize}
			alors il existe un entier $h$ tel que $u_h^{-1} = \sigma^{\alpha_0}(1^{-1})\dots \sigma^{\alpha_{p-1}}(1^{-1})$.
			On suppose que $u_k^{-1}$ est apparu à l'étape $n\ge 1$. L'exposant $\alpha_p$ vaut soit $n-1$, soit $n$ (lemme \ref{lem:automappar})
			et l'arbre décrit ci-dessous est un sous-arbre de $T_n^s$ (on rappelle que l'application $f_0$ est définie en \ref{eq:f_0}).
			\begin{figure}[h!]
			\begin{center}
				\begin{psfrags}
				\psfrag{1}{\LARGE{$1$}}
				\psfrag{d}{\LARGE{$d$}}
				\psfrag{x}{\LARGE{$\nu_n^{-1}(x_{n-1})$}}
				\psfrag{y}{\LARGE{$\nu_n^{-1}(y_{n-1}^\prime)$}}
				\psfrag{fu}{\LARGE{$f_0^{-1}(u_k^{-1})$}}
				\scalebox{0.58}{\includegraphics{eps/surjhelp.eps}}
				\end{psfrags}
			\end{center}
			\vspace{-4mm}
			\end{figure}\ \\
			Si $\alpha_p = n-1$, alors on a
			\begin{center}
				$f_0(\nu_n^{-1}(x_{n-1})) = \sigma^{\alpha_0}(1^{-1})\dots \sigma^{\alpha_{p-1}}(1^{-1})$.
			\end{center}
			Si $\alpha_p = n$, alors le point $y_{n-1}^\prime$ est un point de branchement (lemme \ref{lem:helpsurj}), et on a
			\begin{center}
				$f_0(\nu_n^{-1}(y_{n-1})) = \sigma^{\alpha_0}(1^{-1})\dots \sigma^{\alpha_{p-1}}(1^{-1})$.
			\end{center}
			\end{proof}

		\subsection{Interprétation des cylindres du système symbolique $\Omega^+$ dans l'arbre $T_{\tau}$}\label{subsec:cec}
		Dans la section \ref{sec:conjugaison}, on montre que l'application $f_Q$ (\ref{eq:f_Q2}) réalise une conjugaison entre
		le système symbolique engendré par $\sigma$ et un système d'isométries partielles sur $T_{\tau}$ (\ref{eq:Ttau}).
		En plus de fournir des informations précises sur la structure auto-similaire de $T_{\tau}$, la substitution d'arbre a ainsi produit
		un moyen efficace d'associer une trajectoire infinie de $\Omega^+$ à un point de l'arbre. Cela simplifie la compréhension
		de la représentation géométrique du système dynamique engendré par $\sigma$. On voudrait donner une méthode systématique
		pour associer une substitution d'arbre à un automorphisme donné. Dans la présente section, on étudie plus en détail l'exemple proposé
		et on met en évidence une relation fondamentale entre la substitution d'arbre et le système symbolique engendré par $\sigma$.
		Cette relation pourrait servir de base à des résultats plus généraux.

		Chaque arbre de la suite $(T_n)_n$ (convergente vers $T_{\tau}$) peut être décomposé en arcs simples (définition \ref{defn:arcsimple}) et chaque arc simple correspond à une partie de $T_{\tau}$.
		La décomposition en arcs simples induit donc une partition (modulo un ensemble fini de points) de $T_{\tau}$ ;
		notre but est de montrer que chacune de ces partitions se relève par $f_Q^{-1}$ en une partition
		(modulo un ensemble fini de points) en cylindres de $\Omega^+$.

			\subsubsection{Images par $f_Q$ des cylindres de $\Omega^+$}
			Si $[s, t]$ est un arc simple de $T_n$, il existe $k$, $(1\le k\le 2d-2)$, tel que le segment $[s, t]$
			de $T_{n+k}$ contient exactement $1$ point de branchement ; on le note $z(s, t)$.
			L'antécédent de $z(s, t)$ par $f_Q$ dans $\Omega^+_p$ est noté $u(s, t)^{-1}\omega$.

			Pour tout mot $u$ de $\mathfrak{L}(\Omega)$, on note $P_u$ l'ensemble des mots $V$ de $\Omega^+$
			tels que $uV$ est encore dans $\Omega^+$.
			Dans ce paragraphe, on montre par récurrence que pour tout $n\in \mathds{N}$ et pour tout
			arc simple $[s, t]$ de $T_n$, on a $f_Q(P_{u(s, t)}) = T_{\tau}(s, t)$.

			\begin{rmk}
			Si $u$ est un préfixe de $\omega$, alors l'élément $u^{-1}\omega$ de $\Omega^+_p$ est dans $P_u$.
			\end{rmk}
			La proposition suivante fera office d'initialisation pour la récurrence à suivre. L'arbre
			$T_0^s$ étant constitué des $d$ arêtes $(x_0, x_j, j)$ (où $x_0$ est la racine et $1\le j\le d$), on définit
			$z_j = \nu_0(x_j)$ pour tout $0\le j\le d$ (voir paragraphe \ref{subsubsec:real} pour la
			définition de $\nu_0$).
			\begin{prop}
			Pour tout $1\le j\le d$, on a $T_{\tau}(z_0, z_j) = f_Q(P_{u(z_0, z_j)})$.
			\end{prop}
			\begin{proof}
			On déduit de la figure \ref{fig:emmaboule1} que $u(z_0, z_1)^{-1} = \sigma^{d-1}(1^{-1})$ et que
			pour tout $2\le j\le d$, on a $u(z_0, z_j)^{-1} = \sigma^{j-2}(1^{-1})$.
			Le point $z_0 = f_Q(\omega)$ est à l'intersection des $T_{\tau}(z_0, z_j)$ et $\omega\in P_{u(z_0, z_j)}$ quel que soit $1\le j\le d$
			(le mot bi-infini $\lim\limits_{n\to +\infty} \sigma^{dn}(j.1)$ existe quel que soit $j$).

			Si $z\in T_{\tau}(z_0, z_j)\cap (\bigcup\limits_{n\in \mathds{N}} \mathcal{W}_n)$ et $z\ne z_0$, alors $z$ est l'image par $f_Q$ d'un mot $w^{-1}\omega$ de $\Omega^+_p$
			tel que~~ $w^{-1} = u(z_0, z_j)^{-1}\sigma^{k_1}(d^{-1})\sigma^{k_2}(d^{-1})\dots \sigma^{k_m}(d^{-1})$ avec\vspace{0.1cm}
			\begin{itemize}
				\item $j^\prime+(d-1)\le k_1\le j^\prime+(2d-2)$ si $u(z_0, z_j)^{-1}$ apparaît à l'étape $j^\prime$ (proposition \ref{prop:emmalite}),\vspace{0.1cm}
				\item $k_i+(d-1)\le k_{i+1}\le k_i+(2d-2)$ pour tout $1\le i < m$ (toujours par \ref{prop:emmalite}).\vspace{0.1cm}
			\end{itemize}
			Le mot $u(z_0, z_j)^{-1}$ est un préfixe de $w^{-1}$ et on déduit que $w^{-1}\omega\in P_{u(z_0, z_j)}$.\vspace{0.1cm}

			Soit $z$ un point de $T_{\tau}(z_0, z_j)\setminus (\bigcup\limits_{n\in \mathds{N}} \mathcal{W}_n)$. En reprenant le principe de la preuve de la proposition
			\ref{prop:emmasurj}, on trouve une suite $(x_n)_n$ d'éléments de $T_{\tau}(z_0, z_j)\cap (\bigcup\limits_{n\in \mathds{N}} \mathcal{W}_n)$
			telle que, pour tout $n\in \mathds{N}$,\vspace{0.1cm}
			\begin{itemize}
				\item $f_Q(w_n^{-1}\omega) = x_n$,\vspace{0.1cm}
				\item si $w_n^{-1} = \sigma^{\alpha_0}(1^{-1})\dots \sigma^{\alpha_{p}}(1^{-1})$ avec $\alpha_{i+1} - \alpha_{i} \ge d$ pour tout $0\le i < p$,\\
				alors $w_{n+1}^{-1} = \sigma^{\alpha_0}(1^{-1})\dots \sigma^{\alpha_{p}}(1^{-1})\sigma^{\alpha_{p+1}}(1^{-1})$ avec $\alpha_{p+1}-\alpha_p\ge d$.\vspace{0.1cm}
			\end{itemize}
			La suite $(w_n)_n$ détermine un unique développement en préfixes-suffixes qui a un unique antécédent $U.V$ par $\Gamma^{-1}$.
			Il existe un rang $k$ tel que pour tout $n\ge k$, le mot $w_n^{-1}$ admet $u(z_0, z_j)^{-1}$ comme préfixe (théorème \ref{thm:prgamma}), ce qui fait
			de $u(z_0, z_j)$ un suffixe de $U$ ; on en déduit que $V\in P_{u(z_0, z_j)}$. De plus, l'égalité $f_Q(V) = z$ est vérifiée, et on a finalement $z\in f_Q(P_{u(z_0, z_j)})$.
			On conclut que $T_{\tau}(z_0, z_j)\subset f_Q(P_{u(z_0, z_j)})$.\\

			Soit $z$ un point de $f_Q(P_{u(z_0, z_j)})$ et $V$ un antécédent de $z$ dans $P_{u(z_0, z_j)}$.
			Si $V$ est un mot de $\Omega^+_p$ et $f_Q(V)\in T_{\tau}(z_0, z_k)$, avec $k\ne j$, alors on vient
			de voir que $V$ est également dans $P_{u(z_0, z_k)}$.
			Les dernières lettres des mots $u(z_0, z_j)$ sont deux à deux distinctes. Le mot $\omega$ étant le seul mot de $\Omega^+$ spécial à gauche
			(proposition \ref{prop:uniqinfspegau}), on a $V = \omega$. On en déduit $z = f_Q(V)$
			est dans $T_{\tau}(z_0, z_j)$.

			Si $V$ n'est pas dans $\Omega^+_p$, il existe un mot $U$ infini à gauche tel que $U.V\in \Omega$ et
			$u(z_0, z_j)$ est un suffixe de $U$. On suppose que $\Gamma(U.V) = (p_i, a_i, s_i)_{i\ge 0}$ et on note, pour tout
			$n\in \mathds{N}$, $w_{n}^{-1} = p_0^{-1}\sigma(p_1^{-1})\dots \sigma^n(p_n^{-1})$. D'après le théorème
			\ref{thm:prgamma}, il existe
			un entier $k$ tel que $u(z_0, z_j)^{-1}$ est un préfixe de $w_{n}^{-1}$ pour tout $n\ge k$. La suite $(f_Q(w_{n}^{-1}\omega))_{n\ge k}$
			est une suite de Cauchy de $T_{\tau}(z_0, z_j)$ (compact en tant que fermé du compact $T_{\tau}$) et converge vers $f_Q(V)$.
			Finalement, $z = f_Q(V)\in T_{\tau}(z_0, z_j)$ et $f_Q(P_{u(z_0, z_j)})\subset T_{\tau}(z_0, z_j)$. 
			\end{proof}
			On généralise maintenant cette propriété à tout arc simple de tout arbre $T_n$. La proposition suivante
			n'est qu'une étape dans la démonstration du théorème \ref{thm:emmagenarpar} à venir. Elle nous permettra
			d'associer à tout arbre $T_n$ une décomposition (pour l'instant grossière) de $\Omega^+$ ; le sens réel de cette décomposition
			est donné dans le paragraphe suivant.
			\begin{prop}\label{prop:emmarcgen}
			Pour tout $n\in \mathds{N}$ et pour tout arc simple $[s, t]$ de $T_n$, on a $f_Q(P_{u(s, t)}) = T_{\tau}(s, t)$.
			\end{prop}
			\begin{proof}
			L'initialisation de la récurrence a été faite dans la proposition précédente et on suppose la
			propriété vraie au rang $n-1$.

			Soit $w^{-1}$ un mot apparaissant à l'étape $n$. On note $y_0 = f_Q(w^{-1}\omega)$ ; il existe
			$d$ points $y_1, \dots, y_d$ de $\mathcal{V}_n$ tels que pour tout $1\le j\le d$, $[y_0, y_j]$ est
			un arc simple de $T_n$. De plus, par hypothèse de récurrence, il existe deux points, disons
			$y_1$ et $y_2$, tels que $T_{\tau}(y_1, y_2) = f_Q(P_w)$.

			On veut se ramener à un cas similaire à celui de l'initialisation. Pour cela, il faut prendre en compte
			le fait qu'au moins un des points $y_1, y_2$ est un point de branchement (alors qu'ils étaient
			tous les deux de degré $1$ dans l'initialisation).
			
			On suppose que $y_1$ et $y_2$ sont tous les deux dans $\mathcal{W}_n$. Les mots $w^{-1}\sigma^n(d)\omega$
			et $w^{-1}\sigma^n(1)\omega$ sont des antécédents par $f_Q$ de $y_1$ et $y_2$ ;
			on supposera par exemple que
			$y_1 = f_Q(w^{-1}\sigma^n(d)\omega)$ et $y_2 = f_Q(w^{-1}\sigma^n(1)\omega)$. 
			Il est évident que $y_1\in T_{\tau}(y_0, y_1)$ et $y_2\in T_{\tau}(y_0, y_2)$ et il faut maintenant s'assurer que
			les mots $w^{-1}\sigma^n(d)\omega$ et $w^{-1}\sigma^n(1)\omega$ sont des
			éléments de $P_{u(y_0, y_1)}$ et $P_{u(y_0, y_2)}$ respectivement. Sachant que
			$u(y_0, y_1) = \sigma^{n+d-1}(d)w$ et $u(y_0, y_2) = \sigma^{n+d}(d)w$, il suffit
			de montrer que les mots $\sigma^{n+d-1}(d)\sigma^n(d)\omega$
			et $\sigma^{n+d}(d)\sigma^n(1)\omega$ sont effectivement des éléments
			de $\Omega^+$ et on conclut que c'est le cas en se rappelant que $\sigma^d(1) = \sigma^{d-1}(1)1$.

			Le reste de la preuve (pour le rang $n$) se fait en reprenant le raisonnement donné dans la proposition précédente. 
			\end{proof}

			\subsubsection{Chaque arbre de la suite $(T_n)_n$ détermine une partition en cylindres de $\Omega^+$}
			Jusqu'à la fin de cet article, toute partition de $\Omega^+$ sera en fait une partition modulo un ensemble fini.
			Pour tout $m\in \mathds{N}$, on note $\mathscr{P}_m$ la partition en cylindres de $\Omega^+$ définie par
			\begin{center}
				$\mathscr{P}_m = \{P_u ; u\in \mathfrak{L}(\Omega)\ et\ |u|=m\}$.
			\end{center}
			Dans ce paragraphe, on fait correspondre chaque arbre $T_n$ à l'une de ces partitions.
			Quel que soit $n\in \mathds{N}$, on explicite un entier $m$ tel que
			tout arc simple $[s, t]$ de $T_n$ vérifie $f_Q(P_u) = T_{\tau}(s, t)$ pour un certain $P_u$ de $\mathscr{P}_m$ ;
			réciproquement, pour tout $P_u$ de $\mathscr{P}_m$, il existe un arc simple $[s, t]$ de $T_n$ tel que
			$f_Q(P_u) = T_{\tau}(s, t)$.

			\begin{lem}\label{lem:rangecyl}
			Soit $w^{-1}$ un mot apparu à l'étape $n$ ($-(d-2)$ si $w=\epsilon$) ($w^{-1}\omega\in \Omega^+_p$).
			Soit $j$ un entier tel que $d-1\le j\le 2d-2$ ; on note
			$u_j^{-1}$ la première lettre de $\sigma^{n+j}(d^{-1})$. L'égalité
			$P_{\sigma^{n+j}(d)w} = P_{u_jw}$ est vérifiée.
			\end{lem}
			\begin{proof}
			L'inclusion $P_{\sigma^{n+j}(d)w}\subset P_{u_jw}$ est évidente puisque $u_jw$ est un suffixe de
			$\sigma^{n+j}(d)w$. Soit $V$ un élément de $P_{u_jw}\cap \Omega^+_p$. Le mot $V$ est dans $P_w$ et $f_Q(V)$ est dans
			$f_Q(P_w)\cap (\bigcup\limits_{n\in \mathds{N}} \mathcal{W}_n)$. Le mot $V$ est donc un mot de $P_{\sigma^{n+k}(d)w}$
			pour un certain $k$, $d-1\le k\le 2d-2$. Les dernières lettres des mots $\sigma^{n+k}(d)$, $d-1\le k\le 2d-2$
			sont deux à deux distinctes. Si $k\ne j$, alors $wV = \omega$ ($\omega$ est le seul mot de $\Omega^+$ spécial à gauche),
			et $V$ appartient également à $P_{\sigma^{n+j}(d)w}$. Si $V\ne w^{-1}\omega$, on a forcément $k=j$.
			Finalement, tout élément de $P_{u_jw}\cap \Omega^+_p$ est dans $P_{\sigma^{n+j}(d)w}$.

			Soit $V$ un élément de $P_{u_jw}$ et $U$ un mot infini à gauche tel que $U.V\in \Omega$ et $u_jw$ est un suffixe de $U$.
			Soit de plus $\overline{\omega}$ un mot infini à gauche tel que $\overline{\omega}.\omega\in \Omega$. Par
			minimalité, on peut trouver une suite $(\alpha_n)_n$ d'éléments de $\mathds{N}$ telle que la suite
			$(S^{\alpha_n}(\overline{\omega}.\omega))_n$ converge vers $U.V$. On a ainsi mis en évidence une suite
			d'éléments de $P_{u_jw}\cap \Omega^+_p$ ($\subset P_{\sigma^{n+j}(d)w}$) convergente vers $V$.
			L'ensemble $P_{\sigma^{n+j}(d)w}$ étant ouvert-fermé, on a $V\in P_{\sigma^{n+j}(d)w}$. 
			\end{proof}
			Il vient que tout suffixe $w^\prime$ de $\sigma^{n+j}(d)w$ de longueur $\ge |u_jw|$ vérifie $P_{\sigma^{n+j}(d)w} = P_{w^\prime}$
			(les notations sont celles du lemme précédent). La longueur du suffixe $w^\prime$ à choisir sera donné dans le théorème \ref{thm:emmagenarpar}.

			Il faut encore s'assurer que l'on retrouve effectivement une partition de $\Omega^+$ ; puisqu'aucun
			détail n'a été donné concernant la non-injectivité de $f_Q$, on pourrait imaginer l'existence
			d'un cylindre $P_w$ tel que $f_Q(\Omega^+\setminus P_w) = T_{\tau}$.
			Le problème aurait pu être réglé en comptant le nombre d'arêtes des arbres $T_n$ d'une part et le cardinal
			des partitions $\mathscr{P}_n$ (donné par le nombre de mots de longueur $n$ de $\mathfrak{L}(\Omega)$) de l'autre. On préfère
			cependant faire appel aux propriétés de l'application $f_Q$.
			\begin{lem}\label{lem:intercyl}
			Soient $w_1$ et $w_2$ deux mots de $\mathfrak{L}(\Omega)$ tels que $w_1\ne w_2$ et $|w_1| = |w_2|$.
			Alors $f_Q(P_{w_1})\cap f_Q(P_{w_2})$ comporte au plus $1$ point.
			\end{lem}
			\begin{proof}
			Si le mot $w^{-1}$ apparaît à l'étape $n$, alors les mots $w^{-1}\sigma^{n+k}(d^{-1})\omega$,
			pour $(d-1\le k\le 2d-2)$ (cf. proposition \ref{prop:emmalite}) sont des éléments de $P_w$.
			On en déduit que si $[s_1, t_1]$ et $[s_2, t_2]$ sont des arcs simples de $T_{n_1}$ et $T_{n_2}$ respectivement,
			avec $n_1 \le n_2$, alors $T_{\tau}(s_2, t_2)\subset T_{\tau}(s_1, t_1)$ si et seulement si $u(s_1, t_1)$ est un suffixe de
			$u(s_2, t_2)$ ; l'ensemble $f_Q(P_{u(s_1, t_1)})\cap f_Q(P_{u(s_2, t_2)})$ comporte au plus $1$ point ($s_1$ ou $t_1$) sinon.

			Les préfixes de $\omega$ sont les seuls mots spéciaux à gauche (cf. proposition \ref{prop:emmabisp}) ;
			si $w_1$ et $w_2$ sont deux mots de $\mathfrak{L}(\Omega)$ tels que $w_1\ne w_2$ et $|w_1| = |w_2|$,
			alors il existe un unique $v_1$ (resp. $v_2$) tel que
			\begin{itemize}
				\item $v_1$ (resp. $v_2$) est un préfixe de $\omega$,
				\item $w_1$ (resp. $w_2$) est un suffixe de $v_1$ (resp. $v_2$),
				\item $P_{v_1} = P_{w_1}$ (resp. $P_{v_2} = P_{w_2}$).
			\end{itemize}
			On peut supposer que $|v_1|\ge |v_2|$. Puisque $w_1\ne w_2$ et $|w_1| = |w_2|$, alors $v_2$ n'est pas un suffixe
			de $v_1$ et on conclut que  $f_Q(P_{w_1})\cap f_Q(P_{w_2})$ comporte au plus $1$ point. 
			\end{proof}

			On peut maintenant associer une partition $\mathscr{P}_m$ à chaque arbre $T_n$.
			\begin{thm}\label{thm:emmagenarpar}
			Si $n=0$, on prend $m=1$, et si $n\in\mathds{N}^*$, on suppose que 
			$(m-1)$ est la longueur du plus long mot apparaissant à l'étape $n$. On a alors :
				\begin{itemize}
				\item pour tout arc simple $[s, t]$ de $T_n$, il existe un unique élément $P_u$ de $\mathscr{P}_m$ tel que $f_Q(P_u) = T_{\tau}(s, t)$,
				\item réciproquement, pour tout $P_u\in \mathscr{P}_m$, il existe un arc simple $[s, t]$ de $T_n$ pour lequel $f_Q(P_u) = T_{\tau}(s, t)$,
				\end{itemize}
			On dit que $T_n$ \textbf{détermine} la partition $\mathscr{P}_m$.
			\end{thm}
			\begin{proof}
			Soit $[s, t]$ un arc simple de $T_n$. Il existe un mot $w^{-1}$ (apparu à l'étape $r<n$) tel que
			$u(s, t)^{-1} = w^{-1}\sigma^{r+j}(d^{-1})$ pour un certain $j$, $(d-1\le j\le 2d-2)$.
			D'après la proposition \ref{prop:emmarcgen}, $f_Q(P_{u(s, t)}) = T_{\tau}(s, t)$ et par le lemme \ref{lem:rangecyl},
			$P_{u(s, t)} = P_{u_jw}$ si $u_j$ est la dernière lettre de $\sigma^{r+j}(d)$. Le mot $w^{-1}$ est apparu l'étape $r<n$
			et $u(s, t)^{-1}$ n'est pas encore apparu ; on a donc $|u_jw|\le m\le |u(s, t)|$, et on en déduit l'existence d'un mot
			$u$ de longueur $m$ tel que
			\begin{itemize}
				\item $u$ est un suffixe de $u(s, t)$,
				\item $u_jw$ est un suffixe de $u$,
				\item $P_{u(s, t)} = P_u$.
			\end{itemize}
			L'unicité découle directement du lemme \ref{lem:intercyl}.\\

			Soit $u$ un mot de longueur $m$. Il existe un arc simple $[s, t]$ de $T_n$ tel que $f_Q(P_u)\cap T_{\tau}(s, t)$ comporte
			une infinité de points et on vient de montrer qu'il existe un mot $w$ de longueur $m$ tel que $f_Q(P_w) = T_{\tau}(s, t)$.
			On conclut que $u=w$ par le lemme \ref{lem:intercyl}. 
			\end{proof}

			\subsubsection{Propriétés des partitions déterminées}
			Ce paragraphe étudie les partitions de $\Omega^+$ qui peuvent être déterminées par les arbres $T_n$.
			On note $\mu$ l'unique mesure de probabilité sur $\Omega^+$ invariante par décalage. Pour tout $u\in \mathfrak{L}(\Omega)$,
			la mesure de $P_u$ est définie comme la fréquence du mot $u$ dans $\omega$ :
			\begin{center}
				$\mu (P_u) = \lim\limits_{n\to +\infty} \frac{1}{n} \#\{0\le k\ < n ; \omega_k\dots \omega_{k+|u|-1} = u\} > 0$
			\end{center}
			où $\#$ désigne le cardinal et $\omega_j$ ($j\in \mathds{N}$) est la $j$-ième lettre de $\omega$.
			Si $P_u$ et $P_v$ sont les éléments d'une partition $\mathscr{P}_n$ quelconque, on dit que $P_u\sim P_v$ si et seulement si
			$\mu(P_u) = \mu(P_v)$. On va mettre en évidence une relation entre les partitions déterminées par les arbres $T_n$ et
			les cardinaux des ensembles $\mathscr{P}_n / \sim$.

			Pour tout mot $u$ de $\mathfrak{L}(\Omega)$, on note $C_u$ l'ensemble des mots de $\Omega^+$ dont $u$ est préfixe.
			La mesure $\mu$ étant invariante par le décalage, l'égalité $\mu(P_u)=\mu(C_u)$ est
			vérifiée quel que soit $u\in \mathfrak{L}(\Omega)$. On travaillera sur les ensembles $C_u$ dans la suite de ce paragraphe.\\

			On rappelle que les préfixes de $\omega$ sont les seuls mots spéciaux à gauche.
			Le réel $\lambda$ est la valeur propre dominante de la matrice d'incidence de $\sigma$.
			La mesure définie est telle que pour tout $u$ dans le langage, $\mu (\sigma(C_u)) = \lambda^{-1}\mu (C_u)$. Le lemme suivant permettra
			d'utiliser cette propriété par la suite.
			\begin{lem}\label{lem:emmacyl15}
			Si $u$ est un mot de $\mathfrak{L}(\Omega)$ dont la dernière lettre est différente de $d$, alors $\sigma (C_u) = C_{\sigma(u)}$.
			\end{lem}
			\begin{proof}
			L'inclusion $\sigma (C_u)\subset C_{\sigma(u)}$ est évidente. Soit $V\in C_{\sigma(u)}$, $V=\sigma(u)V_1$, $V_1\in \Omega^+$, et puisque $d$ n'est pas la dernière
			lettre de $u$, alors $1$ n'est pas la dernière lettre de $\sigma(u)$ et $V_1$ ne commence pas par $2$. Or, tout mot (infini) commençant par
			$l\ne 2$ a un antécédent par $\sigma$ dans $\Omega^+$, c'est-à-dire qu'il existe $V_2\in \Omega^+$ tel que $\sigma(V_2)=V_1$. Une manière simple de s'en assurer est d'écrire 
			le développement en préfixes-suffixes d'un mot $U.\alpha V_1\in \Omega$, $\alpha\in A$. Dans ce cas, $V=\sigma(uV_2)$, assurant
			que $V$ est dans $\sigma (C_u)$. 
			\end{proof}

			$\lambda$ vérifie $\lambda^d = \lambda^{d-1}+1$ et on en déduit les égalités
				\begin{center}
					\begin{tabular}{ccc}
					$\lambda^{2d-2} = \sum\limits_{0\le k\le d-1} \lambda^k$ & et & $1 = \sum\limits_{0\le k\le d-1} \lambda^{-(d-1+k)}$.
					\end{tabular}
				\end{center}
			On va calculer toutes les valeurs possibles des mesures des cylindres. Il est évident que si $u\in \mathfrak{L}(\Omega)$ et $\alpha\in A$,
			on a $\mu(C_u) = \mu(C_{\alpha u})$ si et seulement si $u$ n'est pas spécial à gauche. On s'intéresse donc au cas où $u$ est spécial à gauche.
			\begin{prop}\label{prop:emmamespos}
			Soit $u$ un préfixe de $\omega$ (possiblement $\epsilon$) et $\mu(C_u)=x$. Alors $\mu(C_{1u})$, $\mu(C_{2u})$, \dots, $\mu(C_{du})$
			prennent les valeurs (pas forcément respectives) $\lambda^{-(d-1+k)}x$, $0\le k\le d-1$.
			\end{prop}
			\begin{proof}
			Si $u$ est un préfixe de $\omega$, alors $u$ est préfixe de $\sigma(u)$. On pourra supposer $u$ spécial à droite, en remarquant
			que si ce n'est pas le cas, il existe un mot $u_0$ spécial à droite (dont $u$ est préfixe) tel que $C_u = C_{u_0}$.
			
			Il n'existe qu'une seule lettre $l_0$ de $\{1, \dots, d\}$ pour laquelle $\sigma(l_0)u$ est encore spécial à droite ;
			on choisit $l\in \{1, \dots, d\} ; l\ne l_0$.
			On déduit de l'expression des préfixes de $\omega$ spéciaux à droite (les bispéciaux)
			qu'il n'existe aucun préfixe $v$ de $\sigma(u)$ tel que $u$ est préfixe de $v$, $|u| < |v| < |\sigma(u)|$ et $v$ est spécial à droite. On en conclut
			que $C_{\sigma(l)\sigma(u)} = C_{\sigma(l)u}$. Le mot $u$ étant spécial à droite, sa dernière lettre est différente de $d$, et par \ref{lem:emmacyl15},
			$\sigma(C_{lu}) = C_{\sigma(l)\sigma(u)}$, ce qui donne $\sigma(C_{lu}) = C_{\sigma(l)u}$.
                        On rappelle que $2$ est non spécial à gauche. On a ainsi, pour tout $l\ne l_0$ :
			\begin{center}
				\begin{tabular}{rclc}
				$\mu (C_{lu})$ & $=$ & $\lambda\mu (C_{(l+1)u})$ & si $l\ne d$,\\
				$\mu (C_{lu})$ & $=$ & $\lambda\mu (C_{1u})$ & si $l = d$.
				\end{tabular}
			\end{center}
			On note $y = \mu (C_{(l_0+1)u})$ si $l_0\ne d$ et $y = \mu (C_{1u})$ si $l_0=d$ ;
			on a alors $\sum\limits_{0\le k\le d-1} \lambda^{-k}y = x$, ce qui donne $y=\lambda^{-(d-1)}x$. 
			\end{proof}

			On va pouvoir préciser le cardinal des ensembles $\mathscr{P}_n / \sim$ après un dernier lemme.
			Pour tout $n\in \mathds{N}$, on note $E_n = \{\mu(C_u) ; u\in \mathfrak{L}(\Omega), |u|=n\}$.
			\begin{lem}\label{lem:emmamaxsim}
			Si $u$ est un préfixe du point fixe $\omega$, alors $\mu (C_u)$ est un élément maximal de $E_{|u|}$. De plus,
			$u$ est spécial à droite si et seulement si pour tout $v\in \mathfrak{L}(\Omega)$ tel que $|v|=|u|, v\ne u$, on a $\mu(C_u) > \mu(C_v)$.
			\end{lem}
			\begin{proof}
			Soit $v\in \mathfrak{L}(\Omega)$ un mot de longueur $|u|$. Si $v$ n'est pas un préfixe du point fixe, alors il existe un préfixe $w$ du point fixe,
			tel que $|w|>|v|$, $v$ est un suffixe (propre) de $w$ et $\mu (C_{w})=\mu (C_v)$ ; en effet si $v$ n'est pas spécial à
			gauche et $\alpha v\in \mathfrak{L}(\Omega)$ ($\alpha\in \{1, \dots, d\}$), alors $\mu (C_{\alpha v})=\mu (C_v)$. Or, puisque $u$ est un préfixe
			de $w$, alors $C_u\supset C_w$ et $\mu (C_u)\ge \mu (C_w) = \mu (C_v)$.\\

			Le mot $v$ est toujours un mot de longueur $|u|$ non préfixe de $\omega$, le mot $u$ est préfixe de $\omega$ et le mot
			$w$ est un préfixe de $\omega$ tel que $v$ est un suffixe propre de $w$ et $\mu (C_{w})=\mu (C_v)$.
			On note $w=uw_0$ ; si $u$ est spécial à droite, alors il existe $w_1$ tel que $|w_1|=|w_0|$, $w_1\ne w_0$ et $uw_1$ est dans le langage.
			Dans ce cas, $C_u\supset C_{uw_0}\cup C_{uw_1}$ et $\mu (C_u) > \mu (C_w) = \mu (C_v)$ (car $C_{uw_0}$ et $C_{uw_1}$ sont de mesures
			strictement positives).
			
			Si pour tout $v$ tel que $|v|=|u|, v\ne u$, on a $\mu(C_u) > \mu(C_v)$, on note $\alpha$ une lettre de $\{1, \dots, d\}$ telle que
			$u\alpha$ est dans le langage. Si $v$ est le suffixe de $u\alpha$ de longueur $|u|$, alors $v$ n'est pas spécial à gauche
			et $\mu (C_v) = \mu (C_{u\alpha})$. On a alors $\mu(C_u) > \mu(C_{u\alpha})$, ce qui implique que $u$ est spécial à droite. 
			\end{proof}
			\begin{prop}
			Pour tout $n\in \mathds{N}$, on note $u_n$ le préfixe de longueur $n$ de $\omega$.
				\begin{itemize}
				\item $E_1 = \bigcup\limits_{0\le k\le d-1} \{\lambda^{-(d-1+k)}\}$,\vspace*{1mm}
				\item $\forall n\ge 2$,
					\begin{itemize}
					\item si le préfixe $u_{n-1}$ de $\omega$ de longueur $n-1$ est spécial à droite, alors il existe $j\in \mathds{N}$
					tel que $E_n = \bigcup\limits_{0\le k\le 2d-3} \{\lambda^{-(j+k)}\}$,
					\item sinon, il existe $j\in \mathds{N}$ tel que
					$E_n = \bigcup\limits_{0\le k\le 2d-2} \{\lambda^{-(j+k)}\}$.
					\end{itemize}
				\end{itemize}
			\end{prop}
			\begin{proof}
			Par récurrence.\\
			On obtient directement de \ref{prop:emmamespos} que
			$E_1 = \bigcup\limits_{0\le k\le d-1} \{\lambda^{-(d-1+k)}\}$. Le mot $1$ est le seul mot de longueur $1$ spécial à gauche,
			$\mu (C_{1}) = \lambda^{-(d-1)}$ d'après \ref{lem:emmamaxsim},
			et on obtient par \ref{prop:emmamespos} que les $\mu (C_{h1})$, $1\le h\le d$ prennent les valeurs $\lambda^{-(2d-2+k)}$, $0\le k\le d-1$.
			Si $l\ne 1$, $l$ n'est pas spécial à gauche et $\mu (C_{(l-1)l}) = \mu (C_l)$.
			On en déduit que $E_2 = \bigcup\limits_{0\le k\le 2d-3} \{\lambda^{-(d+k)}\}$.
			On suppose que les propriétés annoncées sont vraies à tous les rangs $\le n$.

			Si $u_n$ et $u_{n-1}$ sont spéciaux à droite, il existe $j\in \mathds{N}$ tel que $E_n = \bigcup\limits_{0\le k\le 2d-3} \{\lambda^{-(j+k)}\}$
			par hypothèse de récurrence, et par \ref{lem:emmamaxsim}, $C_{u_n}$ est le seul cylindre de mesure $\lambda^{-j}$. On peut déduire de
			la proposition \ref{prop:emmamespos} que $E_{n+1} = \bigcup\limits_{1\le k\le 2d-2} \{\lambda^{-(j+k)}\}$.

			Si $u_n$ est spécial à droite, et $u_{n-1}$ ne l'est pas, alors il existe $j\in \mathds{N}$ tel que $E_n = \bigcup\limits_{0\le k\le 2d-2} \{\lambda^{-(j+k)}\}$ et
			$C_{u_n}$ est le seul cylindre de mesure $\lambda^{-j}$. On en déduit que $E_{n+1} = \bigcup\limits_{1\le k\le 2d-2} \{\lambda^{-(j+k)}\}$.

			Si $u_n$ n'est pas spécial à droite, et $u_{n-1}$ est spécial à droite, alors il existe $j\in \mathds{N}$ tel que $E_n = \bigcup\limits_{0\le k\le 2d-3} \{\lambda^{-(j+k)}\}$ et
			$C_{u_n}$ n'est pas le seul cylindre de mesure $\lambda^{-j}$ ; $E_{n+1} = \bigcup\limits_{0\le k\le 2d-2} \{\lambda^{-(j+k)}\}$.

			Si $u_n$ et $u_{n-1}$ ne sont spéciaux à droite ni l'un ni l'autre, il existe $j\in \mathds{N}$ tel que $E_n = \bigcup\limits_{0\le k\le 2d-2} \{\lambda^{-(j+k)}\}$ et
			$C_{u_n}$ n'est pas le seul cylindre de mesure $\lambda^{-j}$. On a alors $E_{n+1}=E_n$.

			Les propriétés passent donc aux successeurs. 
			\end{proof}

			On peut finalement conclure que
			\begin{itemize}
				\item $\#(\mathscr{P}_1 / \sim) = d$, et pour $n\ge 2$,
				\item $\#(\mathscr{P}_n / \sim) = 2d-2$~~ s'il existe un bispécial de longueur $n-1$, et
				\item $\#(\mathscr{P}_n / \sim) = 2d-1$~~ sinon.
			\end{itemize}
			Ce résultat est à rapprocher du théorème \ref{thm:emmagenarpar} et de la proposition \ref{prop:emmabispeqpldev}.
			\begin{thm}\label{thm:determination}
			$T_0$ détermine $\mathscr{P}_1$. Pour tout $n\in \mathds{N}^*$, l'arbre $T_n$ détermine une partition
			$\mathscr{P}_{m}$ de $\Omega^+$ avec $\#(\mathscr{P}_{m} / \sim) = 2d-2$. Toute
			partition $\mathscr{P}_{m}$ telle que $\#(\mathscr{P}_{m} / \sim) = 2d-2$ est déterminée
			par un arbre $T_n$.
			\end{thm}
			Il est à noter que ce théorème donne un sens au cardinal de l'alphabet $A_\tau$ considéré pour la substitution d'arbre.

		\subsection{Système d'isométries partielles et conjugaison en mesure}\label{sec:conjugaison}
		On définit finalement un système d'isométries partielles sur l'arbre limite $T_{\tau}$ (\ref{eq:Ttau}),
		et on montre que ce système d'isométries traduit géométriquement l'action du décalage sur le système symbolique substitutif.

		On rappelle que $f_Q$ est l'application définie en \ref{eq:f_Q2}. Pour toute lettre $a$ de $A$, on définit l'application
		\begin{center}
		\begin{tabular}{ccccc}
			$\varphi_a$ & : & $f_Q(C_a)$ & $\to$ & $f_Q(P_a)$\\
			&& $x = f_Q(V)$, $V=aV^\prime$ & $\mapsto$ & $\varphi_a(x) = f_Q(V^\prime)$.
		\end{tabular}
		\end{center}
		\begin{prop}
		Pour tout $a\in A$, $\varphi_a$ est une isométrie.
		\end{prop}
		\begin{proof}
		Il suffit de vérifier la propriété pour les points de branchement.
		On rappelle que $\Omega^+_p = \{S^n(\omega), n\in \mathds{N}\}$. Soient $u^{-1}\omega$ et $v^{-1}\omega$
		deux mots de $C_a\cap \Omega^+_p$. On suppose que $f_Q(u^{-1}\omega), f_Q(v^{-1}\omega)$ et leurs images
		$f_Q(a^{-1}u^{-1}\omega), f_Q(a^{-1}v^{-1}\omega)$ sont des éléments de $\mathcal{W}_n$ (l'ensemble des points de degré $d$ de $T_n$).
		On rappelle que $\gamma_n$ est la fonction chemin (voir section \ref{subsec:arbsimp}) de $T_n^s$ et que l'application $f_0$ est définie en \ref{eq:f_0} ; on note\vspace*{1mm}
		\begin{itemize}
			\item $w = \gamma_n (f_0^{-1}(u^{-1}), f_0^{-1}(v^{-1}))$, et\vspace*{1mm}
			\item $w_a = \gamma_n (f_0^{-1}(a^{-1}u^{-1}), f_0^{-1}(a^{-1}v^{-1}))$.\vspace*{1mm}
		\end{itemize}
		Les mots $w$ et $w_a$ vérifient $\sigma^n(p_*(w)) = uv^{-1}$ et $\sigma^n(p_*(w_a)) = uaa^{-1}v^{-1}$.
		Puisque $f_0^{-1}(u^{-1})$, $f_0^{-1}(v^{-1})$, $f_0^{-1}(a^{-1}u^{-1})$ et $f_0^{-1}(a^{-1}v^{-1})$ sont des points
		de branchement, $w$ et $w_a$ sont des mots de $(A\cup \overline{A})^*$ (ils ne contiennent aucune lettre $\overline{k}, k$
		avec $(d+1)\le k\le (2d-2)$). On en déduit que $w = w_a$ et on conclut immédiatement que
		$d_{T_{\tau}} (f_Q(u^{-1}\omega), f_Q(v^{-1}\omega)) = d_{T_{\tau}} (f_Q(a^{-1}u^{-1}\omega), f_Q(a^{-1}v^{-1}\omega))$.
		\end{proof}

		Ainsi, pour tout $a\in A$, l'application $\varphi_a$ est une bijection isométrique qui vérifie
		\begin{center}
			$(f_Q\circ S)(V) = (\varphi_a\circ f_Q)(V)$
		\end{center}
		pour tout élément $V$ de $C_a$.\\

		On montre en plus que l'intersection $f_Q(C_a)\cap f_Q(C_b)$ entre deux domaines distincts ($a, b\in A$ et $a\ne b$)
		est constituée d'au plus un point. Pour cela, on précise la non-injectivité de l'application $f_Q$.
		Le lemme suivant est légèrement plus fort que ce dont nous avons besoin, mais il explicite une
		propriété intéressante de $f_Q$.
		\begin{lem}\label{lem:memeimage}
		Si $V_1$ et $V_2$ sont deux mots de $\Omega^+$ tels que $f_Q(V_1)=f_Q(V_2)$, alors il existe
		un mot infini à gauche $U$ tel que $U.V_1$ et $U.V_2$ sont des mots de $\Omega$.
		\end{lem}
		\begin{proof}
		On suppose qu'il existe deux mots distincts $w_1, w_2\in A^*$ avec $|w_1|=|w_2|$
		et tels que $V_1\in P_{w_1}, V_2\in P_{w_2}$. Si $f_Q(V_1)=f_Q(V_2)$, alors
		d'après le lemme \ref{lem:intercyl}, $f_Q(V_1)=f_Q(V_2) = f_Q(S^k(\omega))$ pour un certain
		$k\in \mathds{N}$.

		Il reste à montrer que pour tout $k\in \mathds{N}$, le point $f_Q(S^k(\omega))$
		n'a qu'un seul antécédent. On montre la propriété pour $f_Q(\omega)$ ; un raisonnement similaire
		pourra être utilisé pour ses décalés. On a déjà établi que $f_Q$ est injective sur
		$\Omega^+_p = \{S^n(\omega) ; n\in \mathds{N}\}$. Si $V\in \Omega^+\setminus \Omega^+_p$,
		alors il existe un mot $U$ infini à gauche tel que $U.V\in \Omega$ et on obtient de la décomposition
		en préfixes-suffixes que $V\in P_{\sigma^{k_2}(1)\sigma^{k_1}(1)}$, pour $k_1, k_2\in \mathds{N}$. Or, une simple
		observation de la figure \ref{fig:emmaboule1} montre l'équivalence :
		\begin{center}
		$f_Q(\omega)\in f_Q(P_w)\Leftrightarrow w = \sigma^k(1)$ pour un certain $k\in \mathds{N}$.
		\end{center}
		\end{proof}

		Le théorème \ref{thm:determination} assure que pour tout $a\in A$, l'ensemble $f_Q(P_a)$ est connexe et fermé ;
		on en déduit qu'il en est de même pour $f_Q(C_a)$. L'intersection $f_Q(C_a)\cap f_Q(C_b)$ avec $a\ne b$ est
		également une partie connexe fermée de $T_{\tau}$ et peut donc posséder $0$, $1$, ou une infinité d'éléments.
		Si l'ensemble $f_Q(C_a)\cap f_Q(C_b)$ possède une infinité d'éléments, alors il contient un point de branchement,
		et on a montré dans le lemme \ref{lem:memeimage} qu'un point de branchement possède un unique antécédent par $f_Q$.
		On en conclut que la proposition suivante est vérifiée.
		\begin{prop}\label{prop:intercyl1}
		Si $a$ et $b$ sont deux éléments distincts de $A$, alors l'intersection $f_Q(C_a)\cap f_Q(C_b)$ possède au plus $1$ élément.
		\end{prop}

		On définit l'ensemble des singularités par
		\begin{eqnarray*}
			\mathfrak{S} & = & \bigcup\limits_{\substack{a, b\in A\\ a\ne b}} (f_Q(C_a)\cap f_Q(C_b))\cup \bigcup\limits_{\substack{a, b\in A\\ a\ne b}}(f_Q(P_a)\cap f_Q(P_b))
		\end{eqnarray*}
		et on note $\mathfrak{S}^*$ l'ensemble des points appartenant aux orbites des points de $\mathfrak{S}$ sous l'action du système d'isométries. L'ensemble $\mathfrak{S}$ est fini d'après
		le lemme \ref{lem:intercyl} et la proposition \ref{prop:intercyl1}, ce qui fait de $\mathfrak{S}^*$ un ensemble dénombrable.

		On note $\beta$ la dimension de Hausdorff de $T_{\tau}$ et $\mathcal{H}^\beta$ sa mesure de Hausdorff associée (voir \cite{Fal}
		pour une introduction à ces concepts et \cite{Cou} pour le calcul de cette dimension). Le réel $\beta$ est forcément $\ge 1$ puisque $T_{\tau}$ contient un intervalle, et on en déduit
		que $\mathcal{H}^\beta (\mathfrak{S}^*) = 0$.

		Tout point de $T_{\tau}^* = T_{\tau}\setminus \mathfrak{S}^*$ est dans le domaine d'une unique isométrie $\varphi_a$ et on note $\varphi$ l'application bijective définie par
		\begin{center}
		\begin{tabular}{ccccl}
			$\varphi$ & : & $T_{\tau}^*$ & $\to$ & $T_{\tau}^*$\\
			&& $x$ & $\mapsto$ & $\varphi_a(x)$ si $x$ est dans le domaine de définition de $\varphi_a$.
		\end{tabular}
		\end{center}
		\begin{thm}
		Le système symbolique $(\Omega^+, S, \mu)$ (où $\mu$ est l'unique mesure de probabilité invariante par $S$)
		engendré par la substitution $\sigma$ est conjugué en mesure au système dynamique $(T_{\tau}^*, \varphi, \mathcal{H}^\beta)$
		engendré par le système d'isométries et on obtient le diagramme commutatif suivant.
		\begin{figure}[h!]
		\begin{center}
			\begin{psfrags}
			\psfrag{O+}{\Large{$\Omega^+$}}
			\psfrag{S}{\Large{$S$}}
			\psfrag{fQ}{\Large{$f_Q$}}
			\psfrag{L}{\Large{$T_{\tau}$}}
			\psfrag{phi}{\Large{$\varphi$}}
			\psfrag{pp}{\hspace{-6mm}$(presque~ partout)$}
			\scalebox{0.7}{\includegraphics{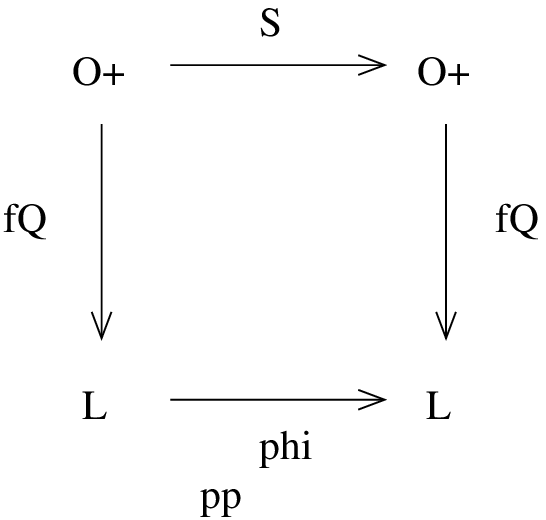}}
			\end{psfrags}
		\end{center}
		\vspace{-4mm}
		\label{fig:conjugaison}
		\end{figure}
		\end{thm}
		On explique dans la section \ref{sectionrepulsif} comment
		l'arbre et le système d'isométries peuvent être construit en utilisant les résultats de \cite{GJLL} et \cite{LL03}. Il était cependant important
		de montrer que l'on pouvait retrouver ces résultats sur cette classe d'exemples de manière indépendante, en exploitant la substitution d'arbre,
		et en insistant sur le fait que la substitution d'arbre explique clairement l'organisation des points de l'arbre suivant leurs orbites
		(ce qui fait défaut aux autres constructions).

		\subsection{L'arbre $T_{\tau}$ vu dans le fractal de Rauzy}

		Pour $d=3$, la substitution est de type Pisot ; sa matrice d'incidence a une valeur propre réelle $>1$ et deux valeurs
		propres complexes conjuguées de modules $<1$. En remarquant que pour toute lettre $i$, le mot $\sigma^{d}(i)$
		commence par le mot $12$, on déduit que la condition de forte coïncidence est vérifiée :
		pour tout $i,j\in A=\{1, 2, 3\}$, il existe $k,n$ (ici $2, d$) tels que la $k$-ième lettre de $\sigma^n(i)$ et la $k$-ième lettre de $\sigma^n(j)$
		sont égales et les abélianisés des préfixes de longueur $(k-1)$ de $\sigma^n(i)$ et $\sigma^n(j)$ sont égaux.
		Ces propriétés permettent d'assurer (voir \cite{AI}, \cite[chapitre 3]{Sie}, \cite{CanSie2}) que le système dynamique symbolique
		engendré par $\sigma$ est conjugué en mesure à un échange de domaines sur une partie compact de $\mathds{R}^2$ : le fractal de Rauzy associé à $\sigma$.

		Pour tout mot $u$ de $F(A)=F_3$, on note $|u|_i$ (pour $i\in A\cup A^{-1}$) le nombre d'occurrences de la lettre $i$
		dans $u$ et on définit $[u]$ comme le vecteur (de $\mathds{Z}^3$) dont la $j$-ème coordonnée ($j\in A$) est $|u|_j - |u|_{j^{-1}}$.
		On note $\Pi$ le plan contractant associé aux deux valeurs propres de modules $<1$, et $\pi_*$
		l'application de $F_3$ dans $\Pi$ qui à un mot $u$ associe la projection de $[u]$ dans $\Pi$ parallèlement à
		la direction dilatante (associée à la valeur propre dominante).

		On définit l'application $\pi$ de $\Omega^+_p = \{S^n(\omega) ; n\in \mathds{N}\}$ dans $\Pi$ par :
		\begin{itemize}
			\item $\pi(\omega) = (0, 0, 0)$,
			\item pour tout mot $V = u^{-1}\omega$ de $\Omega^+_p$, $\pi(V) = \pi_*(u^{-1})$.
		\end{itemize}
		Le \textbf{fractal de Rauzy} $R$ associé à $\sigma$ est l'adhérence dans $\Pi$ de $\pi(\Omega^+_p)$.\\

		Pour toute partie $X$ de $\Omega^+$, on définit $\overline{\pi}(X)$ comme l'adhérence dans $\Pi$ de $\pi(X\cap \Omega^+_p)$.
		On donne finalement un moyen de représenter les arbres $T_n$ dans le fractal de Rauzy en définissant l'application suivante
		($\mathcal{P}$ désigne l'ensemble des parties).
		\begin{center}
			\begin{tabular}{ccccc}
			$\zeta$ & $:$ & $\mathcal{P}(T_{\tau})$ & $\to$ & $\mathcal{P}(R)$\\
			&& $X$ & $\mapsto$ & $\overline{\pi}(f_Q^{-1}(X))$.
			\end{tabular}
		\end{center}

\captionsetup[figure]{justification=centering}

		On représente sur la figure \ref{fig:emmafinalwin123} les images par $\zeta$ des premiers arbres de la substitution d'arbre. Les points marqués d'une croix
		(sur $\zeta (T_5)$, $\zeta (T_6)$ et $\zeta (T_8)$) indiquent des défauts de la représentation planaire ;
		ce sont des couples de points de $\Omega^+_p$ dont les images par $f_Q$ sont distinctes, mais qui ont même image par $\pi$.

		\begin{figure}[h!]
		\begin{center}
			\begin{psfrags}
			\scalebox{0.31}{\includegraphics{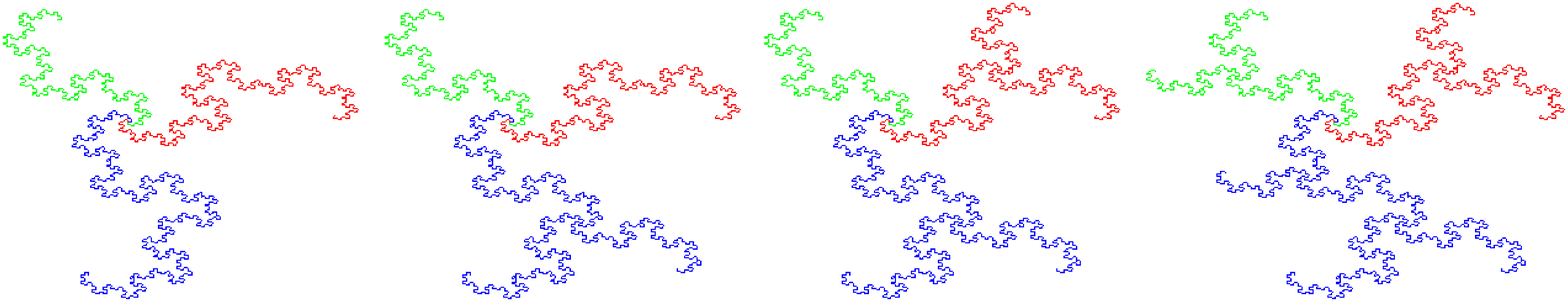}}\\ \ \\
			\scalebox{0.38}{\includegraphics{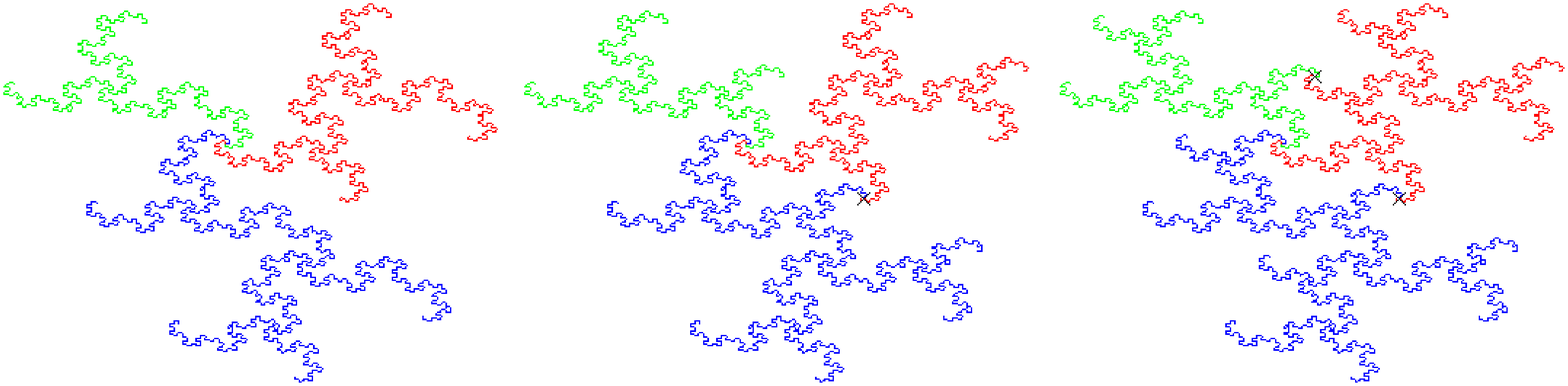}}\\ \ \\
			\scalebox{0.54}{\includegraphics{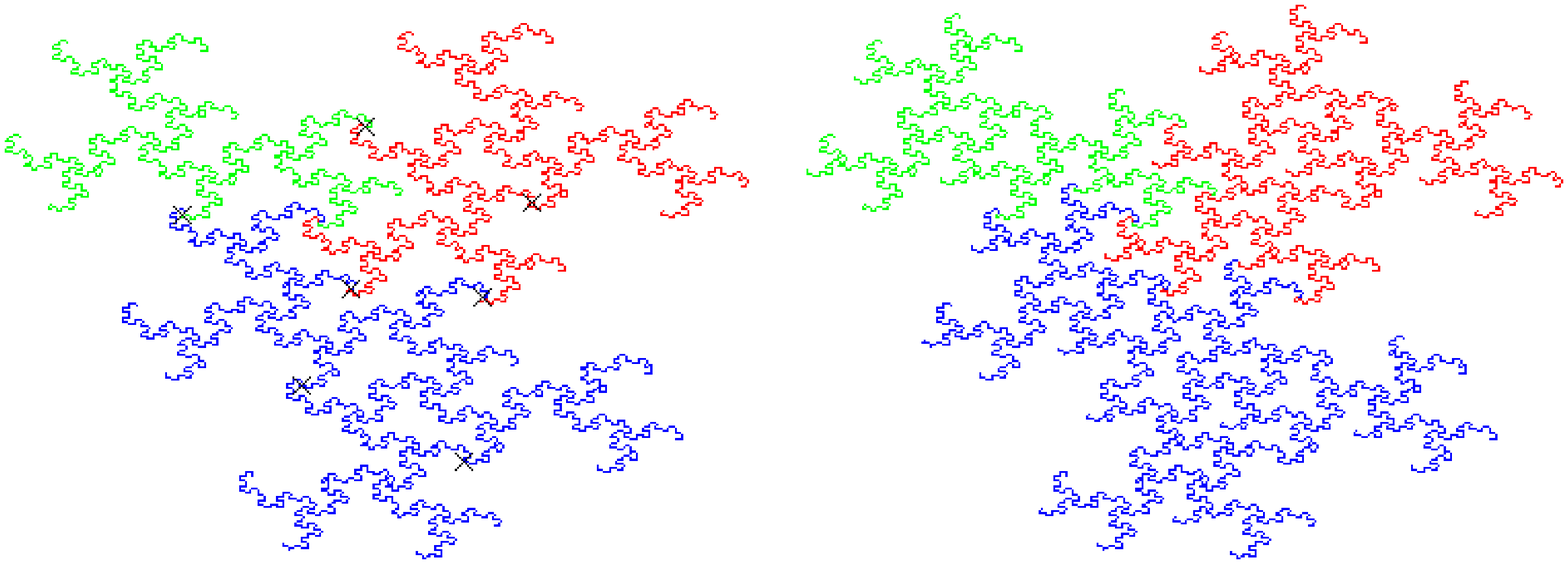}}
			\end{psfrags}
		\end{center}
		\vspace{-4mm}
		\caption{Représentation de $\zeta (T_0)$, $\zeta (T_1)$, $\zeta (T_2)$, $\zeta (T_3)$, $\zeta (T_4)$,\hspace{\linewidth} $\zeta (T_5)$, $\zeta (T_6)$, $\zeta (T_8)$ et $\zeta (T_{10})$.}
		\label{fig:emmafinalwin123}
		\end{figure}

		La figure \ref{fig:emmawin1} illustre le théorème \ref{thm:emmagenarpar} ; les arcs simples bleu, rouge et vert de $T_0$
		déterminent respectivement $P_a$, $P_b$ et $P_c$.
		\begin{figure}[h!]
		\begin{center}
			\begin{psfrags}
			\scalebox{0.6}{\includegraphics{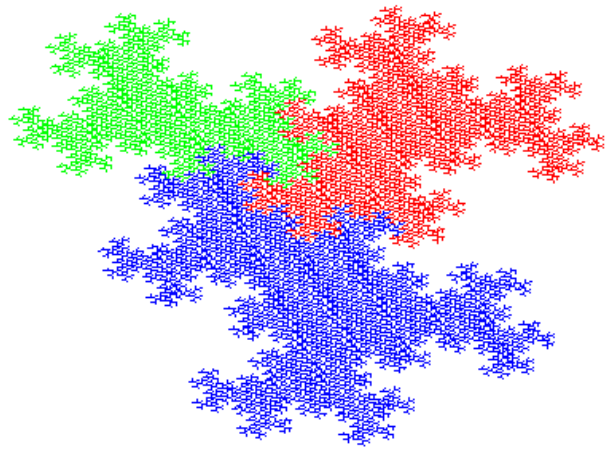}}
			\end{psfrags}
		\end{center}
		\vspace{-4mm}
		\caption{Représentation de $\zeta (T_{22})$.}
		\label{fig:emmawin1}
		\end{figure}

		La figure \ref{fig:emmawin2} illustre le fait que $T_4$ détermine $\mathscr{P}_7$.
		On représente le fractal de Rauzy en attribuant une même couleur à $\overline{\pi}(P_u)$ et $\overline{\pi}(P_v)$ (où $P_u$ et $P_v$ sont dans $\mathscr{P}_7$)
		si $\mu(P_u)=\mu(P_v)$ ; les domaines $\overline{\pi}(P_u)$ et $\overline{\pi}(P_v)$ sont égaux à translation près. L'image par $\zeta$
		de $T_4$ est également représentée ; on note la correspondance entre les arcs simples et les cylindres.

		\begin{figure}[h!]
		\begin{center}
			\begin{psfrags}
			\scalebox{0.37}{\includegraphics{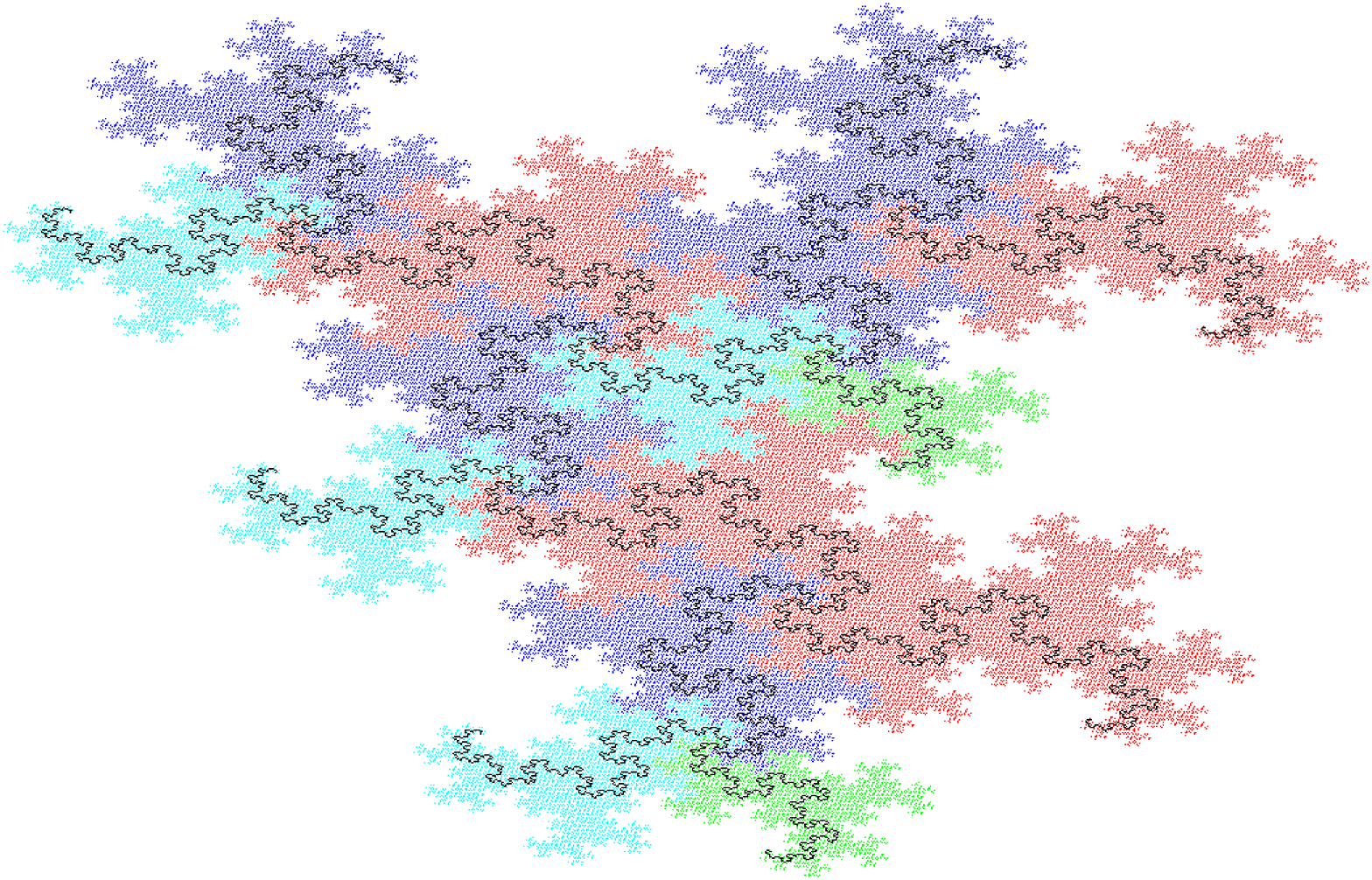}}
			\end{psfrags}
		\end{center}
		\vspace{-4mm}
		\caption{Représentation de $R$ et $\zeta (T_{4})$.}
		\label{fig:emmawin2}
		\end{figure}

\section{Arbre répulsif et c{\oe}ur compact}\label{sectionrepulsif}

	Dans \cite{GJLL}, les auteurs associent un arbre réel muni d'une action du groupe libre par isométries à tout automorphisme de groupe libre.
	Dans cette section, on construit l'arbre $T_{\Phi^{-1}}$ associé à l'automorphisme $\sigma^{-1}$ (de l'exemple proposé), et on montre
	que le compact $T_{\tau}$ (\ref{eq:Ttau}) construit précédemment peut être vu comme une partie de $\overline{T}_{\Phi^{-1}}$, le complété métrique de $T_{\Phi^{-1}}$.

		\subsection{Représentants topologiques}
		On note $\text{Aut}(F_n)$ le groupe des automorphismes du groupe libre $F_n$. Si $w$ est un élément de $F_n$, on note $i_w$ l'automorphisme défini pour tout élément
		$g$ de $F_n$ par $i_w(g) = w^{-1}gw$ ; $i_w$ est un \textbf{automorphisme intérieur} ou \textbf{conjugaison}. On note $\text{Inn}(F_n)$ l'ensemble des conjugaisons
		de $\text{Aut}(F_n)$ et $\text{Out}(F_n) = \text{Aut}(F_n) / \text{Inn}(F_n)$. On dit que $\text{Inn}(F_n)$ est l'ensemble des \textbf{automorphismes intérieurs}, et $\text{Out}(F_n)$ est l'ensemble des
		\textbf{automorphismes extérieurs}.

		L'étude des représentants topologiques est faite dans \cite{BH}. Une très jolie introduction peut également être trouvée dans \cite{ABHS}.
		On se restreint ici au cas des représentants topologiques sur la rose à $n$ pétales $R_n$ ;
		c'est le graphe topologique constitué d'un sommet $*$ et de $n$ arcs orientés
		(espaces topologiques homéomorphes à un segment de $\mathds{R}$) de $*$ vers $*$.
		Un \textbf{chemin} (fini) de $R_n$ est une immersion (application continue et localement injective) d'un segment $[a, b]\subset \mathds{R}$
		dans $R_n$. Cette définition implique qu'un chemin est automatiquement réduit.

		On identifie le groupe fondamental $\pi_1(R_n, *)$ de $R_n$ au groupe libre $F_n$.
		Une \textbf{équivalence d'homotopie} $f:R_n\to R_n$ est une application continue qui induit un 
		automorphisme de $\pi_1(R_n, *)$.
		Cet automorphisme est défini à composition par conjugaison près, puisque l'identification entre $\pi_1 (R_n, *)$ et $\pi_1 (R_n, f(*))$ dépend
		du choix d'un chemin de $*$ à $f(*)$. L'application $f$ induit donc un automorphisme extérieur.

		Un \textbf{représentant topologique} d'un automorphisme extérieur $\Phi\in \text{Out}(F_n)$ sur la rose $R_n$ est
		une application $f:R_n\to R_n$ telle que :
		\begin{itemize}
			\item l'image d'un sommet est un sommet,
			\item l'image d'un arc est un chemin (réduit) de $R_n$,
			\item $f$ induit $\Phi$ sur $F_n\simeq \pi_1 (R_n, *)$ (en particulier, $f$ est une équivalence d'homotopie).
		\end{itemize}

		Un chemin $\chi$ de $R_n$ est dit \textbf{légal} pour un représentant topologique $f$ si pour tout
		$n\in \mathds{N}$, $f^n(\chi)$ est un chemin.

		\subsection{L'arbre invariant de $\sigma^{-1}$}\label{subsec:arbinv}
		\begin{thm}[\cite{GJLL}, \cite{LL08}]\label{thm:arbreinvariant}
		Pour tout automorphisme $\alpha$ du groupe libre $F$, il existe un arbre réel $T$ (appelé arbre invariant de $\alpha$) tel que
			\begin{itemize}
			\item $F$ agit sur $T$ par isométries de manière non-triviale ($F$ ne fixe aucun point de $T$), minimale
			(il n'y pas de sous-arbre propre invariant par $F$), les stabilisateurs d'arcs sont triviaux,
			\item il existe $\eta_\alpha$ et une homothétie $H:T\to T$ de facteur $\eta_\alpha$ telle que, pour tout point
			$P$ de $T$, et pour tout $w\in F$,
			\vspace{1mm}
			\begin{center}
				$\alpha(w)H(P) = H(wP)$.
			\end{center}
			\vspace{1mm}
			\item si $\eta_\alpha > 1$, l'action est à orbites denses (voir \cite[proposition 3.10]{Pau}).
			\end{itemize}
		\end{thm}

		On construit ici l'arbre invariant de $\sigma^{-1}$ (où $\sigma$ est l'automorphisme défini dans la section précédente) ; on pourra se reporter à \cite[partie E]{GJLL} pour la construction de cet arbre dans un cadre plus général.
		On note $\Phi^{-1}$ la classe extérieure de $\sigma^{-1}$. On rappelle que $\eta$ est la valeur propre dominante de la matrice d'incidence de $\sigma^{-1}$ et que le vecteur
		\begin{center}
			$\mathbf{V_{\sigma^{-1}}}=[1\ \eta^{d-1}\ \eta^{d-2}\ \dots\ \eta^2\ \eta]$
		\end{center}
		est un vecteur propre à gauche associé à $\eta$.

		$R_d$ désigne la rose à $d$ pétales. Un chemin $\chi : [0, 1]\to R_d$ où $\chi(0)=\chi(1)=*$ détermine un élément $<\chi>$ du groupe fondamental $\pi_1(R_d, *)$ de $R_d$.
		On identifie $\pi_1(R_d, *)$ avec le groupe libre $F_d$ ; si $e$ est un arc de $R_d$, $<e>$ est un élément de $A\cup A^{-1}$.
		On munit $R_d$ de la métrique telle que si $e_1, e_2, \dots, e_d$ sont les arcs de $R_d$, alors $e_i$ est de longueur $\mathbf{V_{\sigma^{-1}}}(i)$.

		On note $h_0$ l'équivalence d'homotopie de $R_d$ dans $R_d$ vérifiant, pour tout arc
		$e$ de $R_d$, $<h_0(e)> = \sigma^{-1}(<e>)$. L'application $h_0$ multiplie la longueur de tout chemin légal par $\eta$ et
		c'est un représentant topologique de $\Phi^{-1}$. On remarque de plus que tout arc de $R_d$ est un chemin légal.

		On note $p : \widetilde{R_d}\to R_d$ une projection du revêtement universel
		$\widetilde{R_d}$ de $R_d$. La métrique sur $R_d$ induit une distance $d_0$
		sur $\widetilde{R_d}$ et $F_d$ agit sur $\widetilde{R_d}$ par isométries.
		On définit une application (continue) $h: \widetilde{R_d}\to \widetilde{R_d}$ qui vérifie $p\circ h(\widetilde{e}) = h_0(e)$
		pour tout arc $e$ de $R_d$ et pour tout relevé $\widetilde{e}$ de $e$, et telle que
		\begin{center}
			$\sigma^{-1}(w)h = hw$
		\end{center}
		pour tout élément $w$ de $F_d$. On note $\mathcal{P}$ le relevé de $*$ fixe par $h$.

		On définit pour tout $k\in \mathds{N}^*$ la pseudo-distance
		\begin{center}
			$d_k (x, y) = \eta^{-k} d_0 (h^k(x), h^k(y))$
		\end{center}
		sur $\widetilde{R_d}$ et on note $d_\infty = \lim\limits_{k\to +\infty} d_k$.

		Un chemin de $\widetilde{R_d}$ est légal s'il est le relevé d'un chemin légal de $R_d$. L'application $h$
		étend également tout chemin légal de $\eta$, et on a la proposition suivante.
		\begin{prop}\label{prop:distlegalpath}
		Si $\chi : [0, 1]\to \widetilde{R_d}$ est un chemin légal de $\widetilde{R_d}$ avec $\chi([0, 1])=[x, y]$, alors la suite $(d_k(x, y))_{k\in \mathds{N}}$ est constante
		et on a $d_\infty(x, y)=d_0(x, y)$.
		\end{prop}
		La fonction $d_\infty$ est une pseudo-distance sur $\widetilde{R_d}$ et une distance sur $T_{\Phi^{-1}}=\widetilde{R_d} / \sim$ où $x\sim y$ si et seulement si $d_\infty (x, y)=0$.
		Le groupe libre $F_d$ agit encore sur $T_{\Phi^{-1}}$ par isométries et $h$ induit $H$ sur $T_{\Phi^{-1}}$. L'application $H$ est une homothétie de rapport $\eta$ qui vérifie
		\begin{center}
			$\sigma^{-1}(w)H = Hw$
		\end{center}
		pour tout élément $w$ de $F_d$. On note encore $\mathcal{P}$ le point de $T_{\Phi^{-1}}$ fixe par $H$. Finalement, $T_{\Phi^{-1}}$ vérifie les propriétés du théorème précédent.
		C'est l'arbre invariant de $\sigma^{-1}$ ; il est également appelé arbre répulsif de $\sigma$.

			\subsubsection*{L'application $Q$}\label{subsec:appQ}
			Le groupe libre $F_d$ agit sur l'arbre réel $T_{\Phi^{-1}}$ de manière non-triviale, minimale, avec stabilisateurs d'arcs triviaux.
			L'action est de plus à orbites denses puisque $\eta > 1$.
			Sous ces conditions, il est possible de construire l'application $Q$ décrite dans \cite{LL03} par G. Levitt et M. Lustig.
			L'application $Q$ est équivariante et surjective de $\partial F_d$ dans $\overline{T}_{\Phi^{-1}}\cup \partial T_{\Phi^{-1}}$
			(où $\overline{T}_{\Phi^{-1}}$ est le complété métrique de $T_{\Phi^{-1}}$ et $\partial T_{\Phi^{-1}}$ son bord de Gromov). Elle vérifie la proposition suivante.
			\begin{prop}[\cite{LL03}]\label{prop:niceQ}
			Soit $V\in \partial F$ ; pour tout $Z$ de $\overline{T}_{\Phi^{-1}}$, si la suite $(v_n)_n$ de $F$ tend vers $V$
			(lorsque $n\to +\infty$) et si la suite $(v_nZ)_n$ converge (lorsque $n\to +\infty$) vers un point $R$
			de $\overline{T}_{\Phi^{-1}}$, alors $R=Q(V)$.
			\end{prop}
			La propriété d'équivariance se traduit, pour tout élément $u$ de $F_d$ et tout élément $V$ de $\partial F_d$, par $Q(uV) = uQ(V)$.

	\subsection{Le compact $T_{\tau}$ est une partie de $\overline{T}_{\Phi^{-1}}$}
		La distance $d_\infty$ et l'homothétie $H$ s'étendent naturellement à $\overline{T}_{\Phi^{-1}}$ ;
		on les note encore $d_\infty$ et $H$.
		On va établir une bijection isométrique de $T_{\tau}$ (\ref{eq:Ttau}) dans $Q(\Omega^+)$.
		On rappelle que $f_Q$ est l'application définie en \ref{eq:f_Q2} et qu'elle est surjective
		du système symbolique $\Omega^+$ dans l'arbre $T_{\tau}$ (proposition \ref{prop:emmasurj}).
		On montre d'abord que pour tout $u^{-1}\omega,v^{-1}\omega$ de $\Omega^+_p$, on a
		\begin{center}
			$d_\infty(Q(u^{-1}\omega), Q(v^{-1}\omega)) = d_{T_{\tau}}(f_Q(u^{-1}\omega), f_Q(v^{-1}\omega))$
		\end{center}
		On rappelle que $\mathcal{P}$ est le point de $T_{\Phi^{-1}}$ fixe par $H$.
		\begin{lem}\label{lem:emmafixT}
		Pour tout élément $w$ de $F_d$, $Q(w\omega) = w\mathcal{P}$.
		\end{lem}
		\begin{proof}
		Par la propriété $\sigma^{-1}(w)H = Hw$ pour tout $w$ de $F_d$, on déduit
		\begin{center}
			$d_\infty (\mathcal{P}, \sigma^n(1)\mathcal{P}) = \eta^{-n}d_\infty(\mathcal{P}, 1\mathcal{P})$.
		\end{center}
		Ainsi, la suite $(\sigma^n(1))_n$ converge vers $\omega$, et la suite
		$(\sigma^n(1)\mathcal{P})_n$ converge vers $\mathcal{P}$ ; d'après la proposition \ref{prop:niceQ}, $Q(\omega) = \mathcal{P}$.
		On conclut grâce à l'équivariance de $Q$. 
		\end{proof}
		\begin{prop}\label{prop:emmapreconsdist}
		Pour tout $u^{-1}\omega,v^{-1}\omega$ de $\Omega^+_p$,
		\begin{center}
			$d_\infty(Q(u^{-1}\omega), Q(v^{-1}\omega)) = d_{T_{\tau}}(f_Q(u^{-1}\omega), f_Q(v^{-1}\omega))$.
		\end{center}
		\end{prop}
		\begin{proof}
		On rappelle que $\mathbf{V_{\sigma^{-1}}}$ désigne le vecteur $[1\ \ \eta^{d-1}\ \ \eta^{d-2}\ \ \dots\ \ \eta^2\ \ \eta]$.

		On suppose que $f_0^{-1}(u^{-1})=x$, $f_0^{-1}(v^{-1})=y$ (voir \ref{eq:f_0}) et que $x$ et $y$ sont des points de branchement de $T_n^s$ pour un certain
		$n\in \mathds{N}$. On suppose de plus que $p_*(\gamma_n(x, y)) = w$ ($w\in F_d$) ; les points $x$ et $y$ étant des points de branchement,
		le mot $\gamma_n(x, y)$ ne contient aucune lettre $k, \overline{k}$ avec $d+1\le k\le 2d-2$.
		On peut ainsi déduire de la proposition \ref{prop:emmaletlen} que
		\begin{center}
			$d_{T_{\tau}}(f_Q(u^{-1}\omega), f_Q(v^{-1}\omega)) = \eta^{-n}\sum\limits_{i=0}^{|w|-1} \mathbf{V_{\sigma^{-1}}}(w_i)$,
		\end{center}
		où $|w|$ est la longueur de $w$, $w_i$ est la $i$-ème lettre de $w$ et $\mathbf{V_{\sigma^{-1}}}(w_i)$ désigne la $k$-ième coordonnée de $\mathbf{V_{\sigma^{-1}}}$ si $w_i=k$ ou $w_i=k^{-1}$.

		Utilisant la proposition précédente, l'équivariance de $Q$ et la propriété $\sigma^{-1}(w)H = Hw$, on obtient :
		\begin{center}
			$d_\infty(Q(u^{-1}\omega), Q(v^{-1}\omega)) = d_\infty(u^{-1}\mathcal{P}, u^{-1}\sigma^n(w)\mathcal{P}) = d_\infty(\mathcal{P}, \sigma^{n}(w)\mathcal{P}) = \eta^{-n}d_\infty(\mathcal{P}, w\mathcal{P})$.
		\end{center}
		L'arbre $T_n^s$ est discerné (voir section \ref{subsec:arbsimp}) quel que soit $n\in \mathds{N}$ ; on en déduit que pour tout $k\in \mathds{N}$,
		appliquer $\sigma^{-k}$ à $w$ ne produira aucune annulation (cf. proposition \ref{prop:troncagrees}). Le chemin $[\mathcal{P}, w\mathcal{P}]$ est donc un chemin légal de $\widetilde{R_d}$
		et la proposition \ref{prop:distlegalpath} assure que
		$d_\infty(\mathcal{P}, w\mathcal{P}) = \sum\limits_{i=0}^{|w|-1} \mathbf{V_{\sigma^{-1}}}(w_i)$. 
		\end{proof}

		Il s'agit d'étendre la propriété à l'ensemble $\Omega^+$. Soit $V\in \Omega^+\setminus \Omega^+_p$ et soit $U$ un mot infini à gauche tel que
		$U.V\in \Omega$ et $\Gamma(U.V) = (p_i, a_i, s_i)_{i\in \mathds{N}}$. Soit $(u_n^{-1})_n$ la suite définie par $u_0^{-1} = p_0^{-1}$
		et $u_{n+1}^{-1} = u_n^{-1}\sigma^{n+1}(p_{n+1}^{-1})$ pour tout $n\in \mathds{N}$. On a défini $f_Q(V)$ comme la limite dans $T_{\tau}$ de la suite
		$(f_Q(u_n^{-1}\omega))_n$, et on veut montrer que la suite $(Q(u_n^{-1}\omega))_n$ de $Q(\Omega^+)$ converge vers $Q(V)$.

		Pour tout mot $W = (W_i)_{i\in -\mathds{N}^*}$ infini à gauche, on définit le mot $W^{-1}$ de $\partial F_d$
		par $W^{-1} = (W_{j})_{j\in \mathds{N}}$ avec $W_j = W_{-j-1}^{-1}$.
		La suite $(u_n^{-1})$ est une suite de $F_d$ convergente vers $U^{-1}\in \partial F_d$ (d'après le théorème \ref{thm:prgamma}).
		De plus, pour tout $n\in \mathds{N}$, on a
		\begin{center}
			$d_\infty(Q(u_n^{-1}\omega), Q(u_{n+1}^{-1}\omega)) = \eta^{-(n+1)}d_\infty(\mathcal{P}, p_{n+1}^{-1}\mathcal{P})$.
		\end{center}
		Quel que soit $n\in \mathds{N}$, $p_n$ est égal soit à $1$, soit à $\epsilon$ ; la suite $(Q(u_n^{-1}\omega))_n$
		est donc une suite de Cauchy de $Q(\Omega^+)$ et converge vers $Q(U^{-1})$ d'après la proposition \ref{prop:niceQ}.
		\begin{prop}\label{prop:emmatrick}
		Si $V\in \Omega^+$ et $U.V\in \Omega$ ($U$ étant un mot infini à gauche), alors $Q(V) = Q(U^{-1})$.
		\end{prop}
		\begin{proof}
		On commence par montrer la proposition sur les mots de $\Omega_{per}$. La suite $(\sigma^{nd}(1))_n$
		converge vers $\omega$ et  pour tout $1\le k\le d$, la suite $(\sigma^{nd}(k^{-1}))_n$ converge
		vers un élément $\omega_k$ de $\partial F_d$. On déduit de la relation $\sigma^{-1}(w)H = Hw$ ($w\in F_d$) que
		\begin{center}
			$d_\infty (\mathcal{P}, \sigma^{nd}(1)\mathcal{P}) = \eta^{-nd} d_\infty (\mathcal{P}, 1\mathcal{P})$\\
			$d_\infty (\mathcal{P}, \sigma^{nd}(k^{-1})\mathcal{P}) = \eta^{-nd} d_\infty (\mathcal{P}, k^{-1}\mathcal{P})$.
		\end{center}
		Les suites $(\sigma^{nd}(1)\mathcal{P})_n$ et $(\sigma^{nd}(k^{-1})\mathcal{P})_n$ sont donc des suites de Cauchy convergentes
		vers $\mathcal{P}$. On déduit de la proposition \ref{prop:niceQ} que $Q(\omega) = Q(\omega_k) = \mathcal{P}$ quel que soit $1\le k\le d$.

		Si $U.V\in (\bigcup\limits_{n\in \mathds{Z}}S^n (\Omega_{per}))$, alors il existe un élément $w$ de $F_d$
		tel que $V = w\omega$ et $U^{-1} = w\omega_k$ pour un certain $1\le k\le d$. La proposition reste vraie par
		équivariance de $Q$.\\

		Soit $U.V\in \Omega\setminus (\bigcup\limits_{n\in \mathds{Z}}S^n (\Omega_{per}))$ et $\Gamma (U.V) = (p_i, a_i, s_i)_{i\in \mathds{N}}$.
		On note $v_0 = a_0s_0$, $u_0^{-1} = p_0^{-1}$ et pour tout $n\in \mathds{N}$, $v_{n+1} = v_{n}\sigma^{n+1}(s_{n+1})$ et $u_{n+1}^{-1} = u_{n}^{-1}\sigma^{n+1}(p_{n+1}^{-1})$.
		Par le théorème \ref{thm:prgamma}, les suites $(v_n)_n$ et $(u_n^{-1})_n$ sont convergentes vers $V$ et $U^{-1}$
		respectivement. De plus, les égalités
		\begin{center}
			$d_\infty(v_n\mathcal{P}, v_{n+1}\mathcal{P}) = d_\infty(\mathcal{P}, \sigma^{n+1}(s_{n+1})\mathcal{P}) = \eta^{-(n+1)} d_\infty(\mathcal{P}, s_{n+1}\mathcal{P})$\\
			$d_\infty(u_n^{-1}\mathcal{P}, u_{n+1}^{-1}\mathcal{P}) = d_\infty(\mathcal{P}, \sigma^{n+1}(p_{n+1}^{-1})\mathcal{P}) = \eta^{-(n+1)} d_\infty(\mathcal{P}, p_{n+1}^{-1}\mathcal{P})$
		\end{center}
		sont vérifiées ; les suites $(v_n\mathcal{P})_n$ et $(u_n^{-1}\mathcal{P})_n$ sont des suites de Cauchy de $T_{\Phi^{-1}}$ et convergent dans $\overline{T}_{\Phi^{-1}}$
		vers $Q(V)$ et $Q(U^{-1})$ respectivement. On déduit de la proposition \ref{prop:coolprefsuff}
		que $u_nv_n = \sigma^{n+1}(k)$ pour une certaine lettre $k$ de $A$ ; ainsi,
		\begin{center}
			$d_\infty (u_n^{-1}\mathcal{P}, v_n\mathcal{P}) = d_\infty (\mathcal{P}, \sigma^{n+1}(k)\mathcal{P}) = \eta^{-(n+1)} d_\infty(\mathcal{P}, k\mathcal{P})$,
		\end{center}
		et on peut conclure que $Q(V) = Q(U^{-1})$. 
		\end{proof}
		On obtient ainsi la propriété suivante.
		\begin{prop}\label{prop:emmapostconsdist}
		Pour tout $V_1,V_2$ de $\Omega^+$, $d_\infty(Q(V_1), Q(V_2)) = d_{T_{\tau}}(f_Q(V_1), f_Q(V_2))$.
		\end{prop}

		Les propriétés combinatoires de la substitution d'arbre, couplées avec celles de l'automate des préfixes-suffixes, ont ainsi permis
		de définir une application $f_Q:\Omega^+\to T_{\tau}$ qui réplique, au sens de la propriété \ref{prop:emmapostconsdist}, l'application
		$Q:\Omega^+\to Q(\Omega^+)$. On peut finalement définir une bijection isométrique de $T_{\tau}$ dans $Q(\Omega^+)$.
		\begin{thm}\label{thm:bijiso}
		L'application $\xi$ définie pour tout $V$ de $\Omega^+$ par
		\begin{center}
			\begin{tabular}{ccccc}
			$\xi$ & : & $T_{\tau}$ & $\to$ & $Q(\Omega^+)$\\
			& & $f_Q(V)$ & $\mapsto$ & $Q(V)$
			\end{tabular}
		\end{center}
		est une bijection isométrique.
		\end{thm}
		\begin{figure}[h!]
		\begin{center}
			\begin{psfrags}
			\psfrag{O+}{\Large{$\Omega^+$}}
			\psfrag{QO+}{\Large{$Q(\Omega^+)$}}
			\psfrag{Q}{\Large{$Q$}}
			\psfrag{fQ}{\Large{$f_Q$}}
			\psfrag{L}{\Large{$T_{\tau}$}}
			\psfrag{xi}{\Large{$\xi$}}
			\scalebox{0.7}{\includegraphics{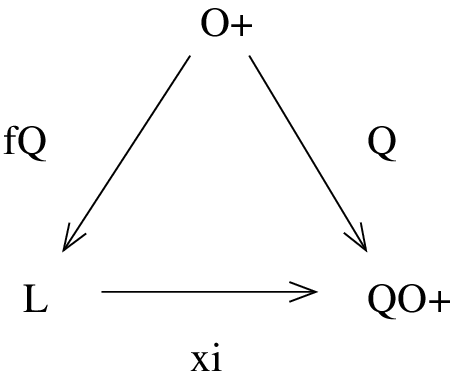}}
			\end{psfrags}
		\end{center}
		\vspace{-4mm}
		\label{fig:xidiagram}
		\end{figure}

	\section{Généralisation}\label{sec:generalisation}
	On discute ici des difficultés liées à la généralisation des constructions et des résultats énoncés précédemment pour la classe d'exemples
	étudiée. On parlera de substitutions inversibles afin de pouvoir réutiliser les définitions et concepts de la dynamique symbolique, mais
	la discussion concerne tous les automorphismes de groupe libre.\\

	Soit $\alpha$ une substitution primitive inversible sur un alphabet $A$ ; on voit $\alpha$ comme un automorphisme de $F(A)$, le groupe libre de base $A$.
	On note $(\Omega^+_\alpha, S)$ le système dynamique engendré par $\alpha$.
	On suppose que l'automorphisme $\alpha^{-1}$ admet un représentant train-track, c'est-à-dire un représentant topologique $f:G\to G$
	(cette fois $G$ peut être différent de la rose) tel que tout arc de $G$ est un chemin légal. On peut alors construire l'arbre
	$T_{\alpha^{-1}}$ invariant de $\alpha^{-1}$ (\cite{GJLL}) et l'application $Q_{\alpha^{-1}} : \partial F(A)\to \overline{T}_{\alpha^{-1}}\cup \partial T_{\alpha^{-1}}$
	associée (\cite{LL03}). L'action par isométries du groupe libre sur $T_{\alpha^{-1}}$ permet de représenter l'action du décalage sur $\Omega^+_\alpha$ par un système
	d'isométries partielles sur $Q_{\alpha^{-1}}(\Omega^+_\alpha)$. Notre but est de construire le compact $Q_{\alpha^{-1}}(\Omega^+_\alpha)$ par substitution d'arbre.

	Les propriétés géométriques de ce compact sont étudiées dans \cite{CH} : suivant l'automorphisme $\alpha$ considéré, il peut être soit une forêt finie (une union finie
	d'arbres réels disjoints), soit un Cantor dont l'enveloppe convexe est un arbre réel. La non-connexité dans le cas d'une forêt finie n'est évidemment pas un obstacle.
	Pour ce qui est du cas Cantor, le problème peut être abordé de deux manières. La première consiste simplement à essayer de construire l'enveloppe convexe par substitution d'arbre,
	alors que la seconde consiste à définir des substitutions de graphe, qui permettraient d'obtenir directement l'ensemble de Cantor. Il semblerait que la meilleure solution
	soit en fait de coupler les deux approches. Dans \cite{BK}, M. Boshernitzan et I. Kornfeld étudie l'un de ces cas ; ils représentent la dynamique d'un automorphisme
	par une translation d'intervalles. Il n'est évidemment pas très avantageux de décrire un intervalle par substitution d'arbre (par contre il serait intéressant
	de retrouver le Cantor), mais le phénomène de translation de domaines peut également se produire sur des arbres plus compliqués. Dans ce cas, obtenir l'enveloppe
	convexe par substitution d'arbre permettrait de représenter explicitement ces translations.\\

	La représentation par substitution d'arbre repose sur un principe simple qui nécessite d'être adressé : un point (en se rappelant que
	ce point correspond à une orbite) placé à une étape finie est définitivement placé. C'est pour cette raison qu'il est essentiel que
	les arbres (simpliciaux) générés par la substitution d'arbre soient des arbres discernés. Le phénomène peut se résumer
	simplement grâce à la proposition \ref{prop:distlegalpath} (la notion de discernement est équivalente à celle de chemin légal).
	Essentiellement, si un arbre de la suite n'est pas discerné, alors il contient deux points dont le chemin est codé par $\overline{k}k$
	(ou $k\overline{k}$) pour une certaine lettre $k$ ; la substitution d'arbre a alors produit deux points distincts de même orbite.

	On se reporte à la proposition \ref{prop:troncagrees} qui constate que, sur l'exemple proposé, le choix du tronc est complètement
	déterminé par l'automorphisme inverse. Ce choix a été rendu possible par le fait que l'automorphisme inverse était train-track
	et sans chemin de Nielsen. Combinatoirement, on dit qu'un automorphisme $\alpha^{-1}$ de $F(A)$ est train-track si pour tout
	$a\in A$ et pour tout $n\in \mathds{N}$, $\alpha^{-n}(a)$ ne produit pas d'annulation ; il est sans chemin de Nielsen si
	pour tout mot $w$ de $F(A)$, il existe $N\in \mathds{N}$ tel que pour tout $n\ge N$, si $w_r$ est l'expression réduite de $\alpha^{-N}(w)$,
	alors $\alpha^{N-n}(w_r)$ ne produit pas d'annulation.

	Voici un exemple d'automorphisme qui ne vérifie pas la propriété train-track. Il s'agit de l'inverse de la substitution
	$a\mapsto ab, b\mapsto ac, c\mapsto a$ (dite de Tribonacci).
	\begin{center}
		\begin{tabular}{ccccl}
		$\phi^{-1}$ & : & $a$ & $\mapsto$ & $c$\\
		&& $b$ & $\mapsto$ & $c^{-1}a$\\
		&& $c$ & $\mapsto$ & $c^{-1}b$.\\
		\end{tabular}
	\end{center}
	Appliquer $\phi^{-2}$ à la lettre $c$ produit une annulation : $\phi^{-2}(c) = \phi^{-1}(c^{-1}b) = b^{-1}cc^{-1}a$.

	On donne également un automorphisme qui a un chemin de Nielsen. C'est l'inverse de la substitution $a\mapsto ba, b\mapsto babac, c\mapsto b$.
	\begin{center}
		\begin{tabular}{ccccl}
		$\phi^{-1}$ & : & $a$ & $\mapsto$ & $c^{-1}a$\\
		&& $b$ & $\mapsto$ & $c$\\
		&& $c$ & $\mapsto$ & $a^{-1}a^{-1}b$.\\
		\end{tabular}
	\end{center}
	On peut vérifier facilement que $\phi^{-1}$ est train-track. Cependant, appliquer $\phi^{-1}$ au mot $ac$ produit une annulation :
	$\phi^{-1}(ac) = c^{-1}aa^{-1}a^{-1}b = c^{-1}a^{-1}b$. Le mot obtenu contient
	$(ac)^{-1}$, et on en déduit qu'appliquer $\phi^{-1}$ successivement produira toujours une annulation.

	Il existe des méthodes pour contourner ces difficultés. En fait, il est montré dans \cite{BH} que tout automorphisme
	dit \textit{iwip} (irreducible with irreducible power) admet un représentant train-track (et les auteurs donnent un algorithme pour le trouver). Par des procédés comparables,
	on peut parfois se débarasser de certains chemins de Nielsen (c'est le cas pour l'exemple précédent). Cependant,
	même dans le cas iwip, il est parfois impossible d'effacer tous les chemins de Nielsen, et cela semble être une des limites
	de la représentation par substitution d'arbre.\\

	Enfin, même si on a pu déterminer le tronc de la substitution d'arbre (c'est-à-dire qu'on a trouvé un représentant train-track sans chemin de Nielsen
	de l'automorphisme $\alpha^{-1}$), il reste à décider les branchements à effectuer. Dans la section \ref{subsec:cec}, on a expliqué
	les branchements de l'exemple proposé en montrant qu'ils respectaient les propriétés combinatoires du système dynamique. Il semble
	que ces propriétés soient un bon point de départ à la généralisation.

\newpage


\addcontentsline{toc}{section}{Références}



\bibliographystyle{alpha}
\bibliography{bibli}{}

\end{document}